\documentclass[11pt,reqno]{amsart}
\usepackage[margin=2.5cm]{geometry}

\usepackage{amsmath, amsfonts, amssymb, amscd, amsthm, amsbsy, amscd}
\usepackage{bbm, epsf, calc, enumerate, verbatim, cite, setspace, bm}
\usepackage{standalone}

\usepackage{graphicx}
\usepackage{enumitem}
\usepackage[usenames,dvipsnames]{color}
\usepackage{mathrsfs, upgreek}
\usepackage{thmtools}
\usepackage{thm-restate}
\usepackage[margin=1cm]{caption}
\usepackage{mathtools}
\usepackage{accents}
\usepackage{wrapfig}
\usepackage{multirow}
\usepackage{float}
\usepackage{blkarray}
\usepackage{tikz}
\usepackage{yfonts}
\usetikzlibrary{arrows.meta}
\graphicspath{{simulations/}}
\usepackage{bm}
\usepackage{stmaryrd}
\usepackage{booktabs}
\usepackage{longtable}
\usepackage{subcaption}
\usepackage{trimclip}
\usepackage[colorlinks=true, linkcolor=blue, citecolor=red, pdfborder={0 0 0},pagebackref,hypertexnames=false]{hyperref} %last option fixed some faulty links
\usepackage{cleveref}

\usetikzlibrary{arrows,matrix,positioning}
\usetikzlibrary{decorations.pathreplacing,decorations.pathmorphing,calc}

\let\oldtocsubsection=\tocsubsection
\renewcommand{\tocsubsection}[2]{\hspace*{1.0cm}\oldtocsubsection{#1}{#2}}
\let\oldtocsubsubsection=\tocsubsubsection
\renewcommand{\tocsubsubsection}[2]{\hspace*{1.8cm}\oldtocsubsubsection{#1}{#2}}

\theoremstyle{plain}
 \newtheorem{theorem}{Theorem}[section]
 \newtheorem{lemma}[theorem]{Lemma}
 \newtheorem{prop}[theorem]{Proposition}
 \newtheorem{cor}[theorem]{Corollary}
 
 \newtheorem{fact}[theorem]{Fact}

\theoremstyle{definition}
 \newtheorem{defn}[theorem]{Definition}

 \newtheorem{claim}[theorem]{Claim}

\theoremstyle{remark}
 \newtheorem{remark}[theorem]{Remark}

\numberwithin{equation}{section}
\numberwithin{theorem}{section}



\newcommand\nc\newcommand
\nc\dmo\DeclareMathOperator
\nc\bs\boldsymbol
\nc{\labbr}[1]{{\sc\lowercase{#1}}}
\nc{\abbr}[1]{{\sc\uppercase{#1}}}
\nc{\ignore}[1]{}

\nc{\cf}{cf.\ }
\nc{\ie}{ i.e. \ }
\nc{\eg}{e.g.\ }

\nc{\overeq}[1]{\stackrel{\text{\tiny #1}}{=}}
\nc{\overleq}[1]{\stackrel{\text{\tiny #1}}{\le}}

\nc{\filldi}{{\hfill$\Diamond$}}

\dmo\per{per}

\renewcommand*{\backref}[1]{\ifx#1\relax \else Page #1 \fi}
\renewcommand*{\backrefalt}[4]{
  \ifcase #1 \footnotesize{(Not cited.)}
  \or        \footnotesize{(Cited on page~#2.)}
  \else      \footnotesize{(Cited on pages~#2.)}
  \fi
}
\DeclareRobustCommand*{\fat}[1]{{\accentset{(\!\trimbox{0pt 1.1ex}{\ensuremath{\string^}}\!)}{#1}}}

\nc{\red}[1]{{\color{red} #1}}
\nc{\blue}[1]{{\color{blue} #1}}
\nc{\green}[1]{{\color{green} #1}}
\nc{\cyan}[1]{{\color{cyan} #1}}
\definecolor{purple}{rgb}{0.9,0,0.8}
\nc{\purple}[1]{{\color{purple} #1}}
\definecolor{gray}{rgb}{0.5,0.5,0.5}
\nc{\gray}[1]{{\color{gray} #1}}
\nc{\note}[1]{{\blue{\textup{[#1]}}}}

\nc{\al}{\alpha} \nc{\Al}{\Alpha}
\nc{\be}{\beta} \nc{\Be}{\Beta}
\nc{\gam}{\gamma} \nc{\Gam}{\Gamma}
\nc{\eps}{\varepsilon}
\nc{\epp}{\epsilon}
\nc{\vep}{\varepsilon}
\nc{\de}{\delta} \nc{\De}{\Delta}
\nc{\lam}{\lambda} \nc{\Lam}{\Lambda}
\nc{\ka}{\kappa} \nc{\Ka}{\Kappa}
\nc{\vpi}{\varpi} \nc{\vph}{\varphi}
\nc{\om}{\omega} \nc{\Om}{\Omega}
\nc{\si}{\sigma} \nc{\Si}{\Sigma}

\nc{\A}{\mathbb{A}}
\nc{\B}{\mathbb{B}}
\nc{\C}{\mathbb{C}}
\nc{\D}{\mathbb{D}}
\nc{\E}{\mathbb{E}}
\nc{\F}{\mathbb{F}}
\nc{\G}{\mathbb{G}}

\nc{\bI}{\mathbb{I}}
\nc{\J}{\mathbb{J}}
\nc{\K}{\mathbb{K}}

\nc{\M}{\mathbb{M}}
\nc{\N}{\mathbb{N}}
\nc{\bO}{\mathbb{O}}

	\renewcommand{\P}{\mathbb{P}}
\nc{\Q}{\mathbb{Q}}
\nc{\R}{\mathbb{R}}
\nc{\bS}{\mathbb{S}}
\nc{\T}{\mathbb{T}}
\nc{\U}{\mathbb{U}}
\nc{\V}{\mathbb{V}}
\nc{\W}{\mathbb{W}}
\nc{\X}{\mathbb{X}}
\nc{\Y}{\mathbb{Y}}
\nc{\Z}{\mathbb{Z}}

\nc{\cA}{\mathcal{A}}
\nc{\cB}{\mathcal{B}}
\nc{\cC}{\mathcal{C}}
\nc{\cD}{\mathcal{D}}
\nc{\cE}{\mathcal{E}}
\nc{\cF}{\mathcal{F}}
\nc{\cG}{\mathcal{G}}
\nc{\cH}{\mathcal{H}}
\nc{\cI}{\mathcal{I}}
\nc{\cJ}{\mathcal{J}}
\nc{\cK}{\mathcal{K}}
\nc{\cL}{\mathcal{L}}
\nc{\cM}{\mathcal{M}}
\nc{\cN}{\mathcal{N}}
\nc{\cO}{\mathcal{O}}
\nc{\cP}{\mathcal{P}}
\nc{\cQ}{\mathcal{Q}}
\nc{\cR}{\mathcal{R}}
\nc{\cS}{\mathcal{S}}
\nc{\cT}{\mathcal{T}}
\nc{\cU}{\mathcal{U}}
\nc{\cV}{\mathcal{V}}
\nc{\cW}{\mathcal{W}}
\nc{\cX}{\mathcal{X}}
\nc{\cY}{\mathcal{Y}}
\nc{\cZ}{\mathcal{Z}}
\nc{\co}{\scriptsize$\cO$\normalsize}

\nc\notni{\not\owns}
\def\({\left(}
\def\){\right)}
\nc{\ii}{\mathrm{i}}

\nc{\ls}{\lesssim}
\nc{\gs}{\gtrsim}
\nc\lllog{\stackrel{\text{\tiny log}}{\ll}}
\nc\gglog{\stackrel{\text{\tiny log}}{\gg}}

\nc{\cond}{ middle|\,}
\dmo{\I}{{I}}
\dmo{\II}{{II}}
\dmo{\argmin}{arg\,min}
\dmo{\argmax}{arg\,max}
\dmo{\area}{area}
\dmo{\diag}{diag}
\dmo{\diam}{diam}
\dmo{\esssup}{ess\,sup}
\dmo{\im}{im}

\dmo{\loc}{loc}
\dmo{\sgn}{sgn}
\dmo{\supp}{supp}
\dmo{\vol}{vol}
\dmo{\dist}{{dist}}

\dmo{\real}{Re}

\DeclareFontFamily{U}{mathx}{\hyphenchar\font45}
\DeclareFontShape{U}{mathx}{m}{n}{
      <5> <6> <7> <8> <9> <10>
      <10.95> <12> <14.4> <17.28> <20.74> <24.88>
      mathx10
      }{}
\DeclareSymbolFont{mathx}{U}{mathx}{m}{n}
\DeclareFontSubstitution{U}{mathx}{m}{n}
\DeclareMathAccent{\widecheck}{0}{mathx}{"71}
\DeclareMathAccent{\wideparen}{0}{mathx}{"75}
\nc{\ol}{\overline}
\renewcommand{\t}{\tilde}
\nc{\ul}{\underline}
\nc{\wt}{\widetilde}
\nc{\wh}{\widehat}
\nc{\wch}{\widecheck}

\nc{\decto}{\downarrow}
\nc{\incto}{\uparrow}
\nc{\la}{\leftarrow}
\nc{\ra}{\rightarrow}
\nc{\La}{\Leftarrow}
\nc{\Ra}{\Rightarrow}
\nc{\lla}{\longleftarrow}
\nc{\lra}{\longrightarrow}
\nc{\Lla}{\Longleftarrow}
\nc{\Lra}{\Longrightarrow}
\nc{\iffa}{\Leftrightarrow}
\nc{\hra}{\hookrightarrow}
\nc{\inj}{\hookrightarrow}
\nc{\surj}{\twoheadrightarrow}
\nc{\bij}{\leftrightarrow}
\nc{\longbij}{\longleftrightarrow}
\nc{\dirlim}{\lim_{\lra}}
\nc{\invlim}{\lim_{\lla}}
\nc{\cra}{\stackrel{\sim}{\ra}}

\dmo{\adj}{adj}
\dmo{\Hom}{Hom}
\dmo{\proj}{proj}
\dmo{\tr}{tr}
\dmo{\Tr}{Tr}
\nc{\Span}{\operatorname{span}}
\def \top {\mathsf{T}}

\def \op {\mathrm{op}}
\dmo{\rank}{rank}

\dmo{\Var}{Var}
\dmo{\Corr}{Corr}
\dmo{\Cov}{Cov}
\dmo{\1}{\mathbf{1}}
\dmo{\ind}{\mathbbm{1}}
\nc{\eqd}{\stackrel{\text{\tiny $d$}}{=}}
\nc{\dto}{\stackrel{d}{\to}}

\nc{\iid}{\abbr{iid}}
\dmo{\Ber}{Ber}
\dmo{\Bin}{Bin}
\dmo{\Poi}{Poi}
\nc{\ER}{Erd\H{o}s--R\'enyi }

\dmo{\Loe}{\preceq}
\dmo{\Roe}{\succeq}

\nc{\HH}{{\mathbb{H}}}
\nc{\rateb}{{\lam}}
\nc{\rateon}{{\nu}}
\nc{\rateoff}{{\mu}}
\nc{\data}{\text{data}}
\nc{\spec}{\text{spec}}
\nc{\sph}{\text{sph}}
\nc{\prj}{\text{proj}}
\nc{\bfx}{\mathbf{x}}
\nc{\tree}{\mathbb{T}}
\nc{\treen}{\tree_2^{(n)}}

\nc{\Dom}{{D}}
\nc{\rd}{{\star}}
\dmo{\road}{{road}}
\dmo{\field}{{field}}

\nc\bx{{\bs x}}
\nc{\ubx}{{\ul{\bx}}}
\nc{\hbx}{{\hat{\bx}}}
\nc{\hx}{{\hat x}}
\nc{\ua}{{\uparrow}}
\nc{\da}{{\downarrow}}
\nc{\raiz}{{\emptyset}}
\nc{\Unif}{\text{Unif}}
\nc\eigv{\mathbf{w}}

\nc{\alg}{G} 
\nc{\Xalg}{X_\alg} 

\nc{\nick}[1]{{\color{blue} #1}}
\nc{\Nick}[1]{{\emph{\nick{[#1]}}}}
\nc{\NICK}[1]{{\bfseries\emph{\nick{[#1]}}}}

\newcommand\sof\newcommand
\sof{\sofia}[1]{{\color{RedViolet} #1}}
\sof{\Sof}[1]{{\bfseries\emph{[#1]}}}

\newcommand{\norm}[1]{\| #1 \|}
\newcommand{\ip}[2]{\left\langle #1, #2 \right\rangle}

\nc\Prob\P
\nc\mc\mathcal

\definecolor{strictfill}{HTML}{E1F5EE}
\definecolor{strictline}{HTML}{0F6E56}
\definecolor{stricttext}{HTML}{04342C}
\definecolor{newfill}{HTML}{FAECE7}
\definecolor{newline}{HTML}{993C1D}
\definecolor{newtext}{HTML}{4A1B0C}
\definecolor{neutralfill}{HTML}{F1EFE8}
\definecolor{neutralline}{HTML}{5F5E5A}
\definecolor{neutraltext}{HTML}{2C2C2A}

\definecolor{neutralfill}{RGB}{245,245,245}
\definecolor{strictfill}{RGB}{220,235,250}
\definecolor{newfill}{RGB}{255,225,225}
\definecolor{paramfill}{RGB}{224,236,247}
\definecolor{paramline}{RGB}{32,84,150}
 
\tikzset{
  neutral/.style    = {draw=black, fill=neutralfill, rounded corners, align=center, 
                       font=\small, inner sep=4pt, line width=0.6pt},
  strict/.style     = {draw=black, fill=strictfill, rounded corners, align=center, 
                       font=\small, inner sep=4pt, line width=0.6pt},
  new/.style        = {draw=red!70, fill=newfill, rounded corners, align=center, 
                       font=\small, inner sep=4pt, line width=0.6pt},
  param/.style      = {draw=paramline, fill=paramfill, rounded corners, align=center,
                       font=\scriptsize, text width=3.2cm, inner sep=3pt, line width=.6pt},
  arr/.style        = {->,>=stealth, draw=black, line width=0.7pt},
  arrnew/.style     = {->,>=stealth, draw=red!70, line width=0.7pt},
  arrparam/.style   = {->,>=stealth, draw=paramline, line width=0.7pt},
  samelink/.style   = {<->,>=stealth, dashed, draw=black!45, line width=.5pt},
  newline/.style    = {draw=red!70, fill=red!30}
}

\title[Spectral Recovery of 
Point Clouds from Noisy Geometric Graphs]{Spectral Recovery of 
Point Clouds\\ from Noisy Geometric Graphs}

\author[T.\ Brailovskaya]{Tatiana Brailovskaya$^\dagger$}
\address{Department of Mathematics, Duke University, 120 Science Dr, Durham, NC 27710, USA}\email{ }

\author[N.\ A.\ Cook]{Nicholas A.\ Cook$^{\star}$}
\thanks{${}^\star$${}^\ddagger$Supported in part by NSF grant DMS-2553253}

\author[S.\ Poinelli]{Sofia Poinelli$^\ddagger$}

\email{${}^\dagger$tatiana.brailovskaya@duke.edu, ${}^\star$nicholas.cook@duke.edu,
${}^\ddagger$sofia.poinelli@duke.edu}

\email{ }

\date{\today}

\begin{document}

\begin{abstract}
We study the problem of recovering low-dimensional latent geometry from a random geometric graph generated by noisy, high-dimensional data. 
Specifically, we analyze the performance of a spectral embedding algorithm on the Signal+Noise Graph Model, in which vertices are associated to points perturbed by Gaussian noise, and edges are included for pairs whose inner product exceeds a specified alignment threshold. In the high-dimensional regime where the number $n$ of points and the ambient dimension $d$ both tend to infinity, we show that under a spectral gap condition, the top eigenvectors and eigenvalues of the graph's adjacency matrix can be used to approximately recover the point cloud up to an orthogonal transformation.
We illustrate our results on point clouds sampled from nested spheres and high-dimensional sinusoid curves.
\end{abstract}

\maketitle

\setcounter{tocdepth}{1}
\tableofcontents

\section{Introduction} 

Recovering low-dimensional geometric structure from high-dimensional, noisy data is a fundamental problem in modern data science. Large datasets are often hypothesized to lie near low-dimensional manifolds embedded in a high-dimensional ambient space, but observations are typically corrupted by noise. Some standard approaches to this problem build an affinity or kernel matrix from the data and extract coordinates via spectral decomposition, as in classical Multidimensional Scaling (MDS), kernel PCA, Laplacian eigenmaps, and diffusion maps. These methods generally require the kernel to be continuous and the data to be densely sampled. 

The vast majority of analysis of these algorithms has been under the \emph{manifold hypothesis}, that the data lies exactly on the low-dimensional manifold, i.e.\ there is no high-dimensional noise. 
In this setting, affinity matrices -- which may be viewed as Laplacian matrices for weighted graphs with the data points as vertices -- have been shown to converge, in various senses, to the Laplace–Beltrami operator of the manifold, at rates depending on the intrinsic dimension \cite{Sin06,GK06,DWW21}. 
In the high-dimensional additive-noise setting, much less is known. 
The works \cite{EKW16, DW23} quantify how adding noise to a signal affects the spectrum of a kernel or graph Laplacian matrix. 
The work \cite{ding2021learning} analyzes an embedding algorithm constructed from the spectral decomposition of a kernel matrix, with kernel function satisfying some regularity hypotheses. 

A separate body of work has studied \emph{latent space} random graph models, in which vertices $i$ are associated to random elements $y_i$ of a set $\cL$ of ``labels'' (or ``latent features''). Conditional on the labels $(y_1,\dots, y_n)$, each pair $\{i,j\}$ is included as an edge independently with probability $q(y_i,y_j)$, for a fixed \emph{affinity} function $q:\cL^2\to[0,1]$. A central problem is to determine conditions under which the labels $y_i$ can be (approximately) recovered from an observation of the graph, with high probability. The most well-studied examples are Stochastic Block Models (SBM), in which $\cL$ is a set of fixed cardinality $k$, inducing a random partition of the vertex set $[n]$ into $k$ \emph{communities}. In certain parameter ranges there is an \emph{information-computation gap}, where recovery of the communities is possible, but cannot be achieved by any polynomial-time algorithm. We refer to the \cite{Abbe2018} for a thorough review of these and related results for SBMs. 

Random geometric graphs (RGGs) are a special class of latent space models in which $\cL$ is equipped with a metric $\rho$, and $q$ takes the form $q(x,y) = 1_{\rho(x,y)\le r}$ for a threshold parameter $r$; when $\cL$ has a Hilbert space structure one may alternatively take $q$ of the form $q(x,y)=1_{\langle x,y\rangle\ge t}$. Classically, RGGs have been studied for the case that $\cL$ is a Euclidean space of fixed dimension \cite{Penrose-RGG}. The influential work \cite{BDER} studied the following hypothesis testing problem for the case that the dimension $d$ grows with the number of vertices $n$, and the latent points $y_i$ are drawn independently and uniformly from the unit sphere $\cS^{d-1}$: for what range of $d=d(n)$ is an RGG distinguishable with high probability from an Erd\H os--R\'enyi graph of the same expected edge density $p$? Results from \cite{BDER} motivated further progress and a general conjecture in \cite{LMSY} that the models are distinguishable for $d=o( (nh(p))^3)$ and indistinguishable when $d=\omega((nh(p))^3)$, where $h(p)$ is the binary entropy function. A proof in the affirmative for the case $d>n$ has recently appeared in \cite{DMSWX}.

In \cite{LiSchramm2024} Li and Schramm introduce the Gaussian Mixture Block Model (GMBM), a random geometric graph in which labels $y_i$ are sampled independently from a mixture of two Gaussians on $\R^d$ centered at $\pm\mu e_1$ (on the first coordinate axis).
Decomposing $y_i=x_i+z_i$ for $x_i$ uniform in $\{+\mu e_1,-\mu e_1\}$ and  $z_i$ independent Gaussian noise, we can view $x_i$ as a clean signal and $z_i$ as noise. They establish bounds on the signal-to-noise ratio (SNR) $\mu$ under which (a) the noisy data $(y_i)_{i=1}^n$ and (b) the uncorrupted labels $(x_i)_{i=1}^n$ can be recovered with high probability from the graph. Their recovery algorithm is based on the spectral decomposition of the graph adjacency matrix. It was recently shown in \cite{fernandez2026spectral} that exact recovery of the labels $x_i$ can be achieved with high probability using a semidefinite program.

In this work, we extend the problem of high-dimensional data recovery from a geometric graph to a more general data model.
We introduce the Signal+Noise Graph Model (SNGM) -- see \Cref{def:SNGM} -- in which edges are determined by thresholding inner products for points $y_i=x_i+z_i$ obtained by perturbing an arbitrary random or deterministic point cloud $(x_i)_{i=1}^n$ by independent Gaussian noise $(z_i)_{i=1}^n$. 
Our main result, \Cref{thm:main}, provides conditions under which an algorithm based on the spectral decomposition of the graph adjacency matrix succeeds in approximately recovering the clean point cloud $(x_i)_{i=1}^n$ up to a global orthogonal transformation. 
We apply this result in \Cref{thm:numerically_low_rk_denoising} to point clouds sampled from a distribution having numerically low-rank population covariance matrix, verifying that natural geometric objects, such as spheres, high-dimensional sinusoid curves with well-separated  frequencies, satisfy the required regularity conditions.

This construction generalizes the GMBM framework of Li and Schramm from a two-component mixture to a numerically low-rank geometric signal. 
 The resulting adjacency matrix in the GMBM setup is effectively a non-linear, discontinuous, spiked kernel random matrix.
 From this perspective, our work connects the study of high-dimensional RGGs to the broader theory of spiked random matrix models. In the null setting where there is no underlying signal, the spectrum of kernel random matrices generated from isotropic data have been analyzed in \cite{elkaroui2010information,cheng2013spectrum,lu2022equivalence,DLMY23}, with limiting spectral distributions described by deterministic laws. Approximate recovery of isotropic data was studied in \cite{EMP22,Araya20,MWX26}, and the recent work \cite{MaMu26} considers anisotropic Gaussian data $y_i\sim N(0, \Sigma)$.
 In the spiked setting, \cite{guionnet2023spectral} studies non-linear Wigner spiked models with continuous kernels, identifying signal-to-noise thresholds for eigenvector alignment. 

Apart from \cite{LiSchramm2024}, the work most closely related to the present work is \cite{ding2021learning}, which studies an embedding algorithm for noisy data based on the spectral decomposition of a kernel matrix. 
The main differences are: 
(i) \cite{ding2021learning} considers a smooth kernel function with a scaling that yields a dense kernel matrix, whereas our kernel function is a discontinuous indicator yielding a  potentially very sparse matrix; 
(ii) their data model takes the clean data $x_i$ to be i.i.d.\ samples from a continuous distribution on a low-dimensional sub-manifold, whereas our results apply to full-rank deterministic point clouds;
(iii) on the other hand, while we restrict to Gaussian noise, \cite{ding2021learning} considers general isotropic sub-Gaussian noise. 

Our framework relaxes the strict low-rank assumption common in the spiked random matrix literature. Rather than requiring the clean signal data to lie exactly in a $k$-dimensional subspace we allow that the principal components grow with $n, d$, provided the tail singular values decay sufficiently fast and no individual point dominates the cloud's structure.

\subsection{Notation}\label{sec:notation}

An index of notation is provided in Appendix \ref{app:notation}.

\emph{Asymptotic notation.} $C,c, C_0$ etc.\ denote positive absolute constants. For two sequences $a_n, b_n$, we write $\omega(a_n) = b_n$ or $a_n = o(b_n)$ if $\lim_{n\to\infty} |a_n/b_n| = 0$. We use $\tilde{\omega}(a_n)= b_n$ and $a_n \lllog b_n$ to indicate that $b_n \ge a_n \log ^Cn$ for a sufficiently large constant $C$. 
We write $a=O(b)$ or $a\lesssim b$ to mean there exists some constant $C>0$ such that $|a|\leq C\cdot b$, and $a_n\lessapprox b_n$ if  $|a_n|\leq b_n\cdot \log^C n$. We also write $a_n=\Theta(b_n)$ to indicate that there exists 
$n_0\in \Z_{\geq 0}$ such that $a_n\ls b_n\ls a_n$ for all $n>n_0$.
Similarly, we use $a_n\asymp b_n$ to indicate that 
$a_n\ls b_n\ls a_n$
for all $n\in \Z_{\geq 0}$.

\emph{Linear algebra.}
We denote the unit sphere in $\R^m$ by $\mathcal{S}^{m-1}$.
For a 
symmetric matrix $B \in \R^{n \times n}$, we denote its eigenvalues by $
\lambda_1(B)\geq \dots \geq \lambda_n(B)$. 
The singular values of a rectangular matrix $B$ are denoted $\sigma_i(B):=\lambda_i(B^\top B)^{1/2}$. 
We denote the spectral operator norm by $\norm{B}_{\op}$ and the Frobenius norm by $\norm{B}_F$. For any matrix $B$, let $\Pi_{B}^{(k)}$ denote the orthogonal projection matrix onto the subspace spanned by the top $k$ eigenvectors of $BB^\top$.
We use the notation $\odot$ to denote the Hadamard (entrywise) product between two matrices of the same dimension. The operator $\diag$ is context dependent, for a matrix $B \in \R^{n \times n}$, $\diag(B)$, it zeros out the off-diagonal entries of $B$, whereas for scalars $a_1, \dots, a_n\in \R$, $\diag(a_1,\dots, a_n)$ it creates a diagonal matrix.

\subsection{AI usage disclosure}

Artificial intelligence tools were used only to aid in the creation of graphics for the figures. The authors take full responsibility for all content of the paper.

\section{Model setup and results}

\subsection{The Model}

To study the extent to which high-dimensional data can be recovered from binary information on pairwise alignment of data points after corruption by additive noise, we introduce to the following general random graph model.

\begin{defn}[Signal+Noise Graph Model (SNGM)]\label{def:SNGM} 
Given a `clean' point cloud $\cX=\{x_1,\dots, x_n\}\subset \R^d$ and
edge-density parameter $p\in [0,1]$,
the Signal+Noise Graph Model (SNGM) is a distribution over 
$n$-vertex graphs $G\sim \mathcal{G}_{n,d}(p,\cX)$ 
generated as follows:
\begin{itemize}
    \item Independently sample 
    Gaussian vectors $z_1,\dots,z_n\sim \mathcal{N}(0,\frac{1}
    {d}I_d)$, 
    and set $y_i:=x_i+z_i$.
    \item Let $V(G)=[n]$, and add an edge $\{i,j\}$ to $E(G)$ if and only if $\langle y_i,y_j\rangle\geq \tau$, where
    $\tau(p,n,d,\cX)$ is 
    taken such that
    \begin{equation}
        \label{def:taup}
        \frac{1}{\binom{n}{2}}\sum_{i<j}\Prob\left(\langle y_i,y_j\rangle\geq \tau\right)=p\,.
    \end{equation}
\end{itemize}
We denote by $X,Y,Z$ the $n\times d$ matrices with rows $x_i^\top, y_i^\top,z_i^\top$, respectively, and $\Sigma:= \frac1nX^\top X$ the sample covariance matrix for the clean point cloud. 
\end{defn}

\begin{remark}[GMBM as a special case of the SNGM]
The GMBM of \cite{LiSchramm2024} is recovered by taking
$x_i$ i.i.d.\ and uniform from $\{-\mu e_1, \mu e_1\}$, independent of the noise $Z$, so that $\Sigma =\mu^2 e_1 e_1^\top$ is rank one
and the noisy data $y_i$ encodes a symmetric and centered two-community structure. 
\end{remark}

\begin{remark}
    [Colored noise]
    \label{rmk:color}
    Our results below admit straightforward extensions to the more general setup that $z_i$ are iid centered Gaussian vectors with invertible covariance matrix $\frac1dC$.
    Indeed, our main result \Cref{thm:main} extends as written to allow any $C$ satisfying $\|C\|_\op,\|C^{-1}\|_\op =O(1)$, by multiplying the data $X$ on the right by $C^{-1/2}$ to reduce to the white noise case $C=I_d$.
\end{remark}

\begin{remark}
In the context of latent space graph models, the vectors that dictate the edge formation ($y_i$ in this case) would traditionally be referred to as \emph{latent vectors}. However, to clearly distinguish the underlying clean point cloud $\cX$ from the noisy samples $\{y_i\}_{i=1}^n$ in our model, we will refer to $x_i$ and $y_i$ as \emph{clean} and \emph{noisy} vectors (or points), respectively.
\end{remark}

\begin{figure}
    \centering
    \includegraphics[width=0.4\linewidth]{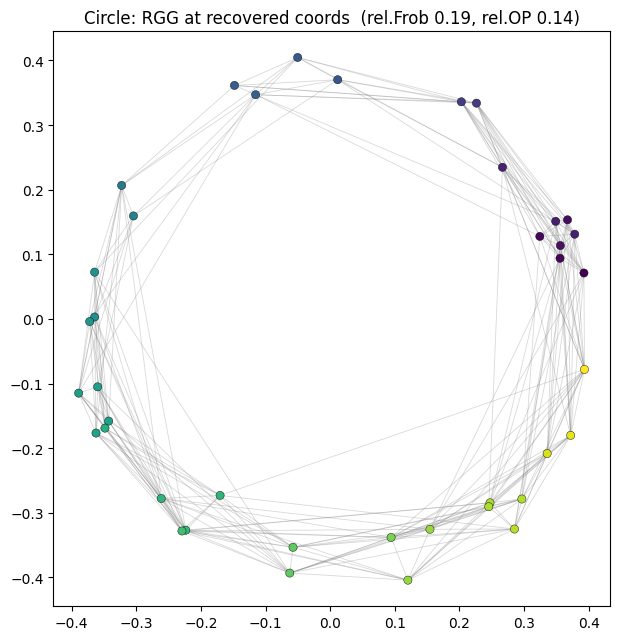}
    \caption{A 40-vertex subgraph of a SNGM on $n=3000$ vertices with density parameter $p=0.1$. The clean point cloud $\cX=\{x_i\}_{i=1}^n$ consists of iid uniform samples from the origin-centered circle of radius $r=\sqrt{0.17}\approx 0.41$ in a 2-dimensional subspace of $\R^d$ for $d=500$, to which $d$-dimensional Gaussian noise is added.
    In the figure the vertices are embedded in $\R^2$ using the spectral algorithm from \Cref{sec:algorithm} applied to the full $n\times n$ SNGM adjacency matrix.}
    \label{fig:SNGM subgraph}
\end{figure}

\subsection{Spectral embedding algorithm}
\label{sec:algorithm}

We analyze the following spectral algorithm to approximately recover a point cloud (up to a global orthogonal transformation) from a geometric graph $G$ generated from noisy measurements of a point cloud $\cX=\{x_1,\dots, x_n\}\subset \R^d$. The algorithm returns a point cloud $\cX^\alg=\{x_1^\alg,\dots, x_n^\alg\}\subset \R^k$ that one hopes retains approximately the same metric structure as $\cX$.\\

\noindent\fbox{%
  \begin{minipage}{\dimexpr\textwidth-2\fboxsep-2\fboxrule}
    \emph{Inputs:} A graph $G$ over vertex set $[n]$ and integers $k,d\in[n]$ with $k\le d$. 
\begin{enumerate}
    \item Compute the edge density $\hat p=\hat p(G):= \binom{n}{2}^{-1}|E(G)|$.
    \item 
    Compute eigenpairs $\{\eta_i, \mathbf{w}_i\}^{k}_{i=1}$ of the $n\times n$ adjacency matrix $A$, where\\ $A_{ij}=\1[\{i,j\}\in E(G)]$, and eigenvalues are labeled $\eta_0\geq\dots\geq \eta_{n-1}$.
    
    \item \emph{Output:} rows $x_1^\alg,\dots, x_n^\alg\in\R^k$ of the $n\times k$ embedding matrix  
    \begin{equation}
        \label{def:embedding}
        \Xalg:= 
        \sqrt{\alpha(\hat p)} \mathbf W_k D_k^{1/2}
    \end{equation}
    where $\mathbf W_k=[\mathbf{w}_1,\dots, \mathbf{w}_k]$, 
    $D_k:= 
    (\max(\eta_i,0))_{i=1}^k$,
    and we define a scaling function
    \begin{equation}
        \label{def:scaling}
        \alpha:(0,1)\to\R^+,\qquad 
        \alpha(q):= \sqrt{\frac{2\pi}{d-k}}\exp\bigg( \frac{d-k}{2d} \Psi^{-1}(q)^2\bigg) 
    \end{equation}
    with $\Psi(t):= \frac1{\sqrt{2\pi}}\int_t^\infty e^{-s^2/2}ds$ the complementary distribution function of a Gaussian.
\end{enumerate}
  \end{minipage}%
}

\quad\\

Note that since the input graph for the algorithm depends on the data only through inner products, one can only hope to recover the point cloud up to a global orthogonal transformation. Hence, we will evaluate the performance of the approximation through recovery of the Gram matrix $XX^\top = (\langle x_i,x_j\rangle)_{i,j=1}^n$.

The parameter $k$ is an estimate for the effective rank of the clean data matrix $X$.
The above algorithm treats $k$ as an input, but if it is unknown then it can be selected in step 2 based on a gap in the spectrum of $A$.

\subsection{Main result} 
\label{sec:result}

Our main technical result is the following theorem, showing that the spectral embedding algorithm of \Cref{sec:algorithm} performs well on the SNGM (\Cref{def:SNGM}) under suitable hypotheses.
Recall our conventions for asymptotic notation from \Cref{sec:notation}.

\begin{theorem}
\label{thm:main}
Suppose that $n,d\in \Z_+$, $p\ge0$ satisfy
\begin{equation}  
    \omega(\log^{16}n)\le d<n\,,\qquad
\omega\bigg(\frac{\log^{18}n}n\bigg)\le p\le 0.49\,. 
\label{assu:density}
\end{equation}
Let $X\in \R^{n\times d}$ such that
\begin{equation}
    \frac1n\|X\|_F^2 =\Tr \Sigma \leq \frac1{\sqrt{d} \log^2 n}     \label{assu:trace}
\end{equation}
with rows $x_1^\top,\dots, x_n^\top$ satisfying
\begin{equation}    \label{assu:no-outliers}
    \max_{1\le i\le n}\|x_i\|_2^2 \lesssim
    \frac 1n \|X\|_F^2\,.
\end{equation}
Denoting by $\mu_1\ge\cdots\ge\mu_d\ge0$ the singular values of $\frac1{\sqrt{n}}X$, assume 
\begin{align}
    \mu_k &\ge \omega(d^{-1/2})
    \qquad \text{and} \qquad
\sum_{i> k}\mu_i^2= o\left(\min(d^{-1}, p^3)\right),
    \label{assu:gap-gen}
\end{align}
for some integer $k$.
If $G\sim \mathcal{G}_{n,d}(p,
\mathcal X)$, then for $\Xalg$ as in \eqref{def:embedding}, with 
probability $1-O(n^{-10})$,
\begin{equation}
    \label{bd:main}
    \norm{XX^\top-\Xalg \Xalg^\top}_\op\lessapprox \cE_\op(X):= \frac{n\mu_1}{\sqrt{d}} +\frac{\mu_1^2}{\mu_k^2} \left(  n\sqrt{d}\mu_1^2\Tr\Sigma +  \sqrt{\frac{n}{dp}} 
    + \frac{n}{p\sqrt{d}} 
    \Big( \sum_{i>k}\mu_i^2\Big)^{1/4}
    \right)\,.
\end{equation}
\end{theorem}

We will first prove \Cref{thm:main} in the special case that the clean data matrix $X$ is of rank $k$, in which case 
\begin{equation}
    \label{bd:rank-k}
    \cE_\op(X)=\frac{n\mu_1}{\sqrt{d}} +\frac{\mu_1^2}{\mu_k^2} \left(  n\sqrt{d}\mu_1^2\Tr\Sigma +  \sqrt{\frac{n}{dp}} 
    \right)\,.
\end{equation}
See \Cref{thm:lowrank}.
Since $\|XX^\top\|_\op = n\mu_1^2$, from \eqref{bd:rank-k} we see that 
in the rank-$k$ case, we guarantee a relative error $\|XX^\top -\Xalg \Xalg^\top\|_\op/\|XX^\top\|_\op = o(1)$ when, for instance, 
\[
d\lessapprox np \qquad \text{and} \qquad  \mu_k \gglog \frac1{\sqrt{d}} \gglog \frac{\mu_1^2}{\mu_k^2} \Tr \Sigma ,
\]
where $a \gglog b $ means $a \ge b\cdot\log^Cn $ for a sufficiently large constant $C$. (One can relax one of the implicit polylogarithmic factors by tightening another.)

\begin{remark}[On the Implicit polylogarithmic Factors]
    See 
    \Cref{sec:overview}
    for refined bounds with explicit logarithmic factors, which we have not attempted to optimize. 
\end{remark}
\begin{remark}\label{rmk:k} Our hypotheses imply that $k=o(\sqrt{d}/\log^2 n)$. Indeed, from the 
bounds in \eqref{assu:trace}, \eqref{assu:gap-gen} it follows 
$\omega(\frac{k}{d})\leq k\mu^2_k\leq \Tr\Sigma\leq \frac{1}{\sqrt{d}\log^2 n}$. 
\end{remark}

\begin{remark}
    The upper bound of 0.49 in \eqref{assu:density} could be replaced by any fixed number less than $\frac12$.
\end{remark}

As a consequence of \Cref{thm:main} 
we obtain conditions under which the embedding $\Xalg$ recovers most of the dot products $\langle x_i,x_j\rangle$ with vanishing error. 

\begin{cor}
\label{cor: Frob cor1}
With $\cE_\op(X)$ as in \eqref{bd:main}, let   
\begin{align}
    \cE_F(X) &:= \sqrt{k}\cE_\op(X) + n 
    \Big( \sum_{i>k}\mu_i^4\Big)^{1/2}
%    \label{def:error-frob}
    \notag\\
    &=n\sqrt
       {\frac{k}{d}}\Bigg[\mu_1
+ \frac{\mu_1^2}{\mu_k^2}
 \bigg(d\mu_1^2\Tr\Sigma 
 +\sqrt{\frac1{np}} +\frac1p 
 \Big( \sum_{i>k}\mu_i^2\Big)^{1/4}
 \bigg)\Bigg]+ n
  \Big( \sum_{i>k}\mu_i^4\Big)^{1/2}.   
 \notag
\end{align}
Under the same assumptions as \Cref{thm:main}, and the same event under which \eqref{bd:main} holds, we have
\begin{equation}
    \label{eqn:frob_cor}
    \norm{XX^\top-\Xalg \Xalg^\top}_F\lessapprox \cE_F(X)\,.
\end{equation}
Moreover, for any $\varepsilon \in(0,1]$ and a sufficiently large constant $C>0$, and $x_i^\alg$ the $i$-th rows of $\Xalg$
\begin{equation}\label{eqn:inner_prod_recover}
    \vert\{(i,j): i\neq j, |\langle x_i,x_j\rangle -\langle x_i^\alg,x_j^\alg\rangle |\geq 
    \frac{\log^C n}{n\sqrt{\eps}}\cE_F (X)\}\lvert\leq n^2\varepsilon,
\end{equation}
and 
\begin{equation}\label{eqn:norm_recovery}
    \Bigg\lvert \bigg\{i\in [n]: |\|x_i\|^2_2-\|x_i^\alg\|^2_2|\geq \frac{\log^Cn } {\eps n} \Big(2k\cE_\op(X)+n
    \sum_{i>k}\mu_i^2 
    \Big)\bigg\} \Bigg\rvert \le \varepsilon n.
\end{equation}
In particular, 
\begin{equation}\label{eqn:norm_recovery2}
    \Bigg\lvert \bigg\{i\in [n]: |\|x_i\|_2-\|x_i^\alg\|_2|\geq \frac{\log^{C/2} n }{\sqrt{\eps n}} \Big(2k\cE_\op(X)+n 
    \sum_{i>k}\mu_i^2
    \Big)^{1/2} \bigg\} \Bigg\rvert \le \varepsilon n\,.
\end{equation}
\end{cor}

\begin{remark}[Relative Recovery] \label{rmk:norms_recovery}
    If one wishes to obtain a relative error statement, implied by \eqref{eqn:inner_prod_recover},\eqref{eqn:norm_recovery}, \eqref{eqn:norm_recovery2}, this would only be feasible for inner products and norms that are large enough compared to the additive errors in the aforementioned equations. We make such a statement in \Cref{cor:relative_recovery} for ``good" sets: samples having large enough inner products and norms.
\end{remark}

\begin{remark}[Error terms]
 The bound \eqref{bd:main} is a sum of four contributions. The first term, $n\mu_1/\sqrt{d}$, is the \textit{concentration error}: it measures how well the top-$k$ eigenspace of the noisy Gram matrix $YY^\top$ aligns with that of $XX^\top$ once the $d-k$ orthogonal noise coordinates are projected away (Lemma \ref{lem:projection-low}). 
 The second term $n(\mu_1^{4}/\mu_k^{2})\sqrt d\Tr\Sigma$,
 is the cost of approximating the discontinuous edge indicator by its low-degree Gegenbauer expansion. It grows with the signal magnitude $\Tr\Sigma$ and quantifies the breakdown of the first-order spherical-harmonic approximation when the signal is too large. The third contribution $\mu_1^2/\mu_k^2\sqrt{n/(dp)}$ is the \textit{density term}:
 if the graph is too sparse there is too little information to recover the geometry. Relative to $\|XX^\top\|_\op=n\mu_1^2$, this density term is $\mu_k^{-2}(ndp)^{-1/2}$, which is $o(1)$, provided 
$\omega(1)\le d\le o(np)$ 
 up to polylogarithms. The fourth term $n\mu_1^2/\mu_{k}^2 
 ( \sum_{i>k}\mu_i^2)^{1/4}
 (p\sqrt{d})^{-1}$ is the contribution of the tail end of the spectrum of the point-cloud, it quantifies the cost of estimating data that is not strictly low-dimensional. Moreover, our methods approximate the adjacency matrix of a SNGM by slightly perturbed spherical data. The multiplicative factor of $\mu_1^2/\mu_k^2$ can be interpreted as the penalty for having a highly anisotropic point-cloud. Specifically it quantifies how well the assumption \eqref{assu:no-outliers} is met.
\end{remark}

\begin{remark}[Sufficient Conditions for Vanishing Error]\label{rmk:vanishing_error}
Under a strengthening of the hypotheses in \Cref{thm:main}, the right hand side of the relative operator norm  error is $o(1)$ provided that:
\begin{itemize}
         \item[$(i)_\op$] $(\mu_1/\mu_k)^{2}\Tr\Sigma=o(d^{-1/2})$,
         \item[$(ii)_\op$] 
         $np=\omega\big(\frac{1}{\mu_k^{4}d}\big)$.
     \end{itemize}
    
   $\cE_F$ from \Cref{cor: Frob cor1} is $o(1)$ provided that:
         \begin{itemize}
         \item[$(i)_F$] $\mu_k^2\sqrt{d}=\omega(\mu_1)$ ,
         \item[$(ii)_F$] $(\mu_1/\mu_k)^{4}\Tr\Sigma=o(d^{-1/2})$,
         \item[$(iii)_F$] 
         $np=\omega\big(\frac{\mu_1^{4}}{\mu_k^{8}d}\big)$.
     \end{itemize}
     The regime $\mu_1/\mu_k =O(1)$, satisfies conditions $(i)_F$, $(ii)_F$ and $(i)_\op$ via $\Tr\Sigma\lesssim 1/(\sqrt{d}\log^2 n)$ and \eqref{assu:gap-gen}.
     In the hardest case of a critically small, well-conditioned signal ($\mu_1/\mu_k=O(1)$, $\mu_k^{2}\asymp 1/d$) this means we must have $np=\omega(d)$. In other words, the average degree must dominate the ambient dimension. Under $(i)_\op$--$(ii)_\op$  the spectral embedding recovers $XX^{T}$ up to rotation with relative operator norm error $o(1)$ and Frobenius norm respectively with $(i)_F$--$(iii)_F$.
     Moreover, in \Cref{thm:main}, if one wishes for the truncation terms in terms of $\sum_{i>k}\mu_i^2$ to be $o(1)$, note that the right-hand side of \eqref{thm:lowrank} (the $\mu_{k+1}=0$ case) agrees with \eqref{thm:main} up to an additive factor, thus it suffices to impose the stronger tail decay $
     \sum_{i>k}\mu_i^2 \leq o( \mu_k^8p^4d^2)$. We choose a weaker hypothesis \eqref{assu:gap-gen} that keeps these term explicit.
\end{remark}

\begin{remark}[On the assumptions in \Cref{thm:main}]
\label{rmk: on the conditions}
\quad\\
$\bullet$ $\Tr\Sigma\leq 1/\sqrt{d}\log^2 n$:
    While it appears counter-intuitive to cap the strength of the signal $\Tr(\Sigma)$, in general \emph{some} upper bound is necessary for recovery. Indeed, if the points are too separated in latent space, then the SNGM may be disconnected, in which case it contains no information about the relative position of points in latent space between vertices in distinct disconnected components. We have made no attempt to prove that this is the optimal upper bound. Moreover the proof techniques employed in this paper, rely on representing the adjacency matrix of the SNGM by low-order spherical harmonics, which is a better approximation when the noise and signal are comparable in magnitude. However, this approximation fails for larger signal-to-noise ratios (SNR)- one may expect higher-order harmonics to contribute in the larger SNR regime. Note moreover, that the proof of \Cref{thm:main} shows that one can relax this assumption to $\Tr\Sigma\lesssim \frac{1}{\sqrt{d}\log^2 n}$, but to avoid overcrowding the error with constants, we keep `$\leq$'. 

$\bullet$
\emph{No Outliers:}  
 The spherical approximation is why we also have the assumption \eqref{assu:no-outliers}, which we refer to as the `no outlier' condition on the clean vectors, in the top $k$ eigenvector directions. We want the most important eigenvectors to be not too far from a sphere in $k$ dimensions for the Gegenbauer approximation to be `good'. Hence the `no-outlier' condition on the projected point cloud. Intuitively, by restricting the maximum squared norm of any projected point to be at most a multiple of the average of the norms, we prevent extreme outliers from distorting the underlying projected point cloud.

\end{remark}

\begin{remark}[More general kernel functions]
Finally, we note that our arguments apply to more general kernel matrices than the graph adjacency matrix $A$ (for which the kernel function is an indicator) -- see \eqref{phi-expand}. We limit the presentation to the step kernel function for concreteness. 
\end{remark}

\subsection{Applications to randomly sampled point clouds}\label{subsec:random_sampling}

\Cref{def:SNGM} and \Cref{thm:main} are stated for a fixed set $\cX$ of points in $\R^d$.
In this section we apply
our results 
to data sampled from continuous manifolds. 

We begin by proving a generic result for i.i.d. data $\cX = \{x_1, \ldots, x_n\}$ sampled from a distribution $\nu$ in $\R^d$ with compact support. We do not directly work with assumptions on a manifold, nor do we assume anything further about the distribution $\nu$ that we sample from, but rather we consider $\Sigma^{\infty}$, the ``true'' data covariance matrix for $x_i$. To this end, we define the population covariance analogue of $\Sigma= \frac{1}{n} X^\top X$ in \Cref{def:SNGM},
\begin{defn}[Population 
distribution]\label{eq:sigmaMdef}
For a distribution $\nu$ in $\R^d$ with compact support, we define the population covariance matrix $\Sigma^{\infty}$ by
    $$\Sigma^{\infty}:= \E_{\bfx \sim \nu}[\bfx\bfx^\top]\in \R^{d\times d}.$$
We denote the eigenvalues and eigenvectors of $\Sigma^{\infty}$ as $(\lambda_i, u_i)_{i=1}^d$, where $\lambda_1\geq \dots\geq \lambda_d$.
\end{defn}

In order to apply \Cref{thm:main} to $\cX$, we observe that for large $n$,
\begin{equation}
\label{eq:sigma_to_sigmaf}
    \Sigma = \frac{1}{n}X^\top X = \frac{1}{n}\sum_{i=1}^n x_ix_i^{\top} \approx \Sigma^{\infty}
\end{equation}
by the law of large numbers (the above display will be made precise in \Cref{lem: discrete trace limit}). Therefore, we can translate our conditions on $\Sigma$ in \Cref{thm:main} or \Cref{thm:lowrank} to those on $\Sigma^{\infty}$.

 We begin by considering the setting where $\Sigma^{\infty}$ is numerically low rank and an easily applicable Corollary  (\Cref{thm:numerically_low_rk_denoising} and \Cref{thm:manifold_denoising_approx_diag}). We then apply each of these variants to concrete examples that are commonplace in the manifold learning literature. The proofs can be found in \Cref{sec: manifold denoising}.

\begin{theorem}[Denoising with numerically low-rank $\Sigma^{\infty}$]\label{thm:numerically_low_rk_denoising}
    Let $n,d,p$ be as in \Cref{thm:main}.
    Let $X$ be the matrix whose rows are given by the point cloud $\cX = \{x_1, \ldots, x_n\}$ sampled independently from a 
    distribution $\nu$ in $\R^d$, with $\omega(\log^{16}n) \leq d < n$. 
    Moreover, suppose
    \begin{equation}
\label{assu:trace_num_low_rank}
        \Tr \Sigma^{\infty} \lesssim \frac{1}{\sqrt{d}\log^2 n} \qquad\text{and}\qquad \Tr \Sigma^{\infty} = o\bigg(\frac{np^3}{\log n}\bigg),
    \end{equation}
    \begin{equation}
    \label{assu:no_outliers_num_low_rank}
\esssup_{\bfx\sim\nu}\|\bfx\|_2^{2} \lesssim \Tr \Sigma^{\infty}
,
    \end{equation}
    and  
    that there exists $1\leq k \leq d$ 
    such that
    \begin{equation}
\label{assu:spec_gap_num_low_rank}
       \lambda_k= \omega(d^{-1})\qquad \text{and}\qquad \sum_{j=k+1}^d \lambda_j = o\left(\min(d^{-1}, p^3)\right).
    \end{equation}
    Let $G \sim \mathcal{G}_{n,d}(p, \cX)$. 
    For the spectral algorithm of Section \ref{sec:algorithm},  
    with probability
    $1-O(n^{-10})$
    :
\begin{align}\label{eqn:diag_manifold_error_rate_OP}
    \|XX^\top- \Xalg \Xalg^\top\|_\op
    \lessapprox&
    n\lambda_1\left(\frac{1}{\sqrt{d\lambda_1}} +\frac{\lambda_1}{\lambda_k}\sqrt{d}\Tr\Sigma^{\infty} + \frac{1}{\lambda_k\sqrt{d}}\left(\sqrt{\frac{1}{np}}+\frac{(\sum_{j=k+1}^d \lambda_j )^{1/4}}{p}\right)\right)
\end{align}
and 
\begin{equation}
    \begin{aligned}
        &\|XX^\top- \Xalg \Xalg^\top\|_F\\
    &\lessapprox
   \sqrt{k}n\lambda_1\left(\frac{1}{\sqrt{d\lambda_1}} +\frac{\lambda_1}{\lambda_k}\sqrt{d}\Tr\Sigma^{\infty} + \frac{1}{\lambda_k\sqrt{d}}\left(\sqrt{\frac{1}{np}}+\frac{(\sum_{j=k+1}^d \lambda_j )^{1/4}}{p}\right)\right) + \sqrt {\lambda_{k+1}}\sqrt{\sum_{j>k}\lambda_j}.
    \end{aligned}
\end{equation}
\end{theorem}

\begin{remark}
The second bound in \eqref{assu:trace_num_low_rank} 
follows automatically from the first bound when 
$d=o(n^2p^6\log^2 n)$. 
\end{remark}

In practice, verifying assumption \eqref{assu:spec_gap_num_low_rank} in \Cref{thm:numerically_low_rk_denoising} can be challenging. However, if $\Sigma^{\infty}$ is approximately diagonal, the diagonal entries of $\Sigma^{\infty}$ serve as proxies for eigenvalues of $\Sigma^{\infty}$ thanks to matrix perturbation inequalities.  The corollary below handles this case explicitly, allowing for an arbitrary perturbation of order $O(1/d)$. We will illustrate the utility of this corollary in an application to a high-dimensional sinusoid curve in \Cref{ex:cosines}.

\begin{cor}[Denoising with approximately diagonal $\Sigma^{\infty}$]
\label{thm:manifold_denoising_approx_diag}
    Let $X$ be the matrix whose rows are given by the point cloud $\cX = \{x_1, \ldots, x_n\}$ sampled from a compactly supported distribution $\nu$ in $\R^d$, with $\omega(\log^{16}n) \leq d < n$ and $\Sigma^{\infty} = \text{\emph{Diag}}(\Sigma^{\infty}) + E$ where $\norm{E}_{\text{\emph{op}}} \lesssim \frac{1}{d}$. Without loss of generality assume $\Sigma^{\infty}_{11} \geq \Sigma^{\infty}_{22} \geq \ldots \geq \Sigma^{\infty}_{dd}$. Moreover, suppose 
    \begin{equation}
    \label{assu:trace approx diag}
        \Tr \Sigma^{\infty} \lesssim \frac{1}{\sqrt{d}\log^2 n}\qquad\text{and}\qquad \Tr \Sigma^{\infty} = o\bigg(\frac{np^3}{\log n}\bigg),
    \end{equation}
    \begin{equation}
    \label{assu:no spikes approx diag}
        \esssup_{x\sim \nu}\norm{\bfx}_2^2 = O(\Tr(\Sigma^{\infty})),
    \end{equation}
    and there exists $k$ such that
    \begin{equation}
    \label{assu:spec gap assum approx diag}
        \Sigma^{\infty}_{kk} = \omega(d^{-1})\qquad \text{and}\qquad
        \sum_{j=k+1}^d \Sigma^{\infty}_{jj} = o\left(\min(d^{-1}, p^3)\right).
    \end{equation}
    Let $G \sim \mathcal{G}_{n,d}(p, \cX)$ with $pn \geq \omega(\log^{18}(n))$. If we apply the spectral algorithm (Section \ref{sec:algorithm}) with $k$ as chosen above then with probability at least $1-O(n^{-10})$:

\begin{align}
\label{eqn:denoising_approx_diag}
        \|XX^\top- \Xalg \Xalg^\top\|_\op &
    \lessapprox 
    n\Sigma^{\infty}_{11}\Bigg(\frac{1}{\sqrt{d\Sigma^{\infty}_{11}}} +\frac{\Sigma^{\infty}_{11}}{\Sigma^{\infty}_{kk}}\sqrt{d}\sum_{j=1}^d \Sigma^{\infty}_{jj}
    + \frac{1}{\Sigma^{\infty}_{kk}\sqrt{d}}\left(\sqrt{\frac{1}{np}}+\frac{(\sum_{j=k+1}^d \Sigma^{\infty}_{jj} )^{1/4}}{p}\right)\Bigg).
\end{align}
\end{cor}

\begin{remark}[Recovery for randomly sampled point clouds]
\label{rmk:sampling_recovery}
%\eqref{eqn:norm_recovery}, \eqref{eqn:inner_prod_recover} are deterministic consequences of the bounds \eqref{bd:main} and \eqref{eqn:frob_cor}.
Since these bounds are expressed through $\Sigma^{\infty}$ rather than $\Sigma$ and can be made more specific depending on their applications, it is more natural to state recovery guarantees on a case-to-case basis, such as in Propositions \ref{prop:sphere}, \ref{prop:nested_spheres}, \ref{prop:cosines}, and \ref{prop:cosines-inner}.
\end{remark}

\subsection{Examples}
\label{subsect:examples}

Here we apply the 
results of the previous subsection to signals sampled from a compact manifold, such as: spheres (\Cref{prop:sphere}) and their nested mixtures (\Cref{prop:nested_spheres}) and sinusoidal curves (\Cref{prop:cosines}).

\subsubsection{Sphere in $m$ dimensions}

\label{subsubsec:sphere}

Let $s_1, \ldots, s_n$ be i.i.d. samples from 
the sphere of radius $R$ in $\R^m$. Let $x_1, \ldots, x_n$ be the $d$-dimensional vectors obtained by extending each $s_i$ with $d-m$ zero entries.

\begin{prop}
\label{prop:sphere}
In the above setting, if $np = \omega(1)$, $\omega(\frac{m}{d})\le R^2 \lllog d^{-1/2}$, 
and $\omega(\log^{16}n) = d<n$, then 
for $\Xalg$ as in \eqref{def:embedding} with $k=m$,
with probability 
$1-O(n^{-10})$, 
\begin{equation}
\label{eq:sphere_err_bd}
    \frac{\norm{XX^{\top} - \Xalg   \Xalg^\top}_\op}{\norm{XX^{\top}}_\op} \lessapprox \frac{\sqrt{m}}{R\sqrt{d}}+ \sqrt{d}R^2+ \frac{m}{R^2\sqrt{npd}},
\end{equation}
and the same bound holds for $(\norm{XX^{\top} - \Xalg \Xalg^\top}_F)/\|XX^\top\|_F$.

Moreover, if additionally 
\begin{equation}\label{assu:radii}
    R^2 \sqrt{npd} = \tilde{\omega}(m) \qquad\text{and, }\qquad  \tilde{\omega}(\frac{m}{d}) = R^2 \lessapprox \frac{1}{\sqrt{d}},
\end{equation} then with probability at least $1-O(n^{-10})$ 
we have 
    \begin{equation}
    \label{eq:len preserved}
        |\{i: \| x_i^\alg \|_2 = (1+o(1))R\}| = (1-o(1))n\,,
    \end{equation}
and for any fixed $\varepsilon_0 \in (0,\frac{1}{8}]$,
    \begin{equation}
    \label{eq:inner prod preserved}
        \vert\{(i,j): i\neq j, \langle x_i^\alg, x_j^\alg \rangle = (1+o_{\varepsilon_0}(1))\langle x_i, x_j \rangle \}\vert \geq \bigg(1- 2\varepsilon_0 - \frac{1}{m} - o(1)\bigg)\bigg)n(n-1)\,.
    \end{equation}
\end{prop}
\begin{remark}
    In the additional assumptions e.g \eqref{assu:radii}, the implicit polylog factors are chosen according to the hidden polylogs in \Cref{prop:sphere}, and are made more precise in the proofs of the statements.
\end{remark}

\subsubsection{Mixture of nested spheres}

Here we generalize the previous example to a mixture of nested spheres. More precisely, for $1 \leq j \leq s$, let $\mathbb{P}_j \sim \text{Unif}(R_j\mathcal{S}^{m-1})$ for $R_j > 0 $ and $w_j > 0$ such that $\sum_{j=1}^s w_j = 1$. Then,
$$x_i \sim \sum_{j=1}^s w_j \mathbb{P}_j,$$
 i.e.  with probability $w_j$ we draw $x_i$ from $\mathbb{P}_j$. For convenience, without loss of generality 
 we assume that $R_1 < R_2 < \ldots < R_s$.

\begin{prop}
\label{prop:nested_spheres}
Suppose $\cX$ is sampled from a mixture of nested spheres as described above. Let $R^2 \coloneqq \sum_{j=1}^s R_j^2 w_j$. If $np = \omega(1)$, $\omega(\frac{m}{d}) = R^2 \lesssim \frac{1}{\sqrt{d}\log^2 n}$, $R_s^2=O(R^2)$ and $\omega(\log^{16}n) = d<n$, then with probability at least 
\eqref{eq:sphere_err_bd} holds, and the same bound holds for the relative Frobenius norm error. 
\end{prop}

A natural question 
is whether the vectors in the embedding associated to distinct spheres remain well separated. For an appropriate choice of model parameters, the answer is ``yes'', as stated below.

\begin{prop}
\label{prop:nested_sphere_analysis}
   Consider $\cX$ as in \Cref{prop:nested_spheres}. Additionally, assume that for some constant $C_0 > 0$, $\frac{R}{C_0} \leq R_j \leq C_0R$ and $\frac{1}{C_0} \leq w_j$ for all $j \in [s]$. Let $\alpha = \min_{j\in[s]} (R_j - R_{j-1})$ with $R_0 = 0$ and $x_i^\alg$ denote the $i$th row of $\Xalg$. Moreover, assume $\tilde{\omega}(\frac{m}{d}) = R^2$, $\tilde \omega(R^2)=d^{-1/2}$ and $R^2 \sqrt{npd} = \tilde{\omega}(m)$. Then there exists $\beta >0$ such that if for some $\varepsilon_0 > 0$,
   \begin{equation}
\label{eq:separation bd}
       \alpha = \omega_{C_0}\bigg(\frac{R}{\varepsilon_0 \log^{\beta} n}\bigg)
   \end{equation}
   then
   \begin{equation}\label{eqn:spheres_ip_norms+recovery}
       \vert \{i: \vert \norm{x_i}_2 - \norm{ x_i^\alg}_2 \vert = o(\alpha)\text{\emph{ and }} \vert \norm{x_i}_2 - \norm{ x_i^\alg}_2 \vert = o_{C_0}(\|x_i\|_2)\} \vert \geq (1-\varepsilon_0)n,
   \end{equation}
   with probability at least $1-O(n^{-10})$.
\end{prop}

\begin{remark}
    In words, the above result shows that under appropriate conditions, if $x_i$ is sampled from $R_j \mathcal{S}^{m-1}$ then $x_i^\alg$ lies much closer to $R_j \mathcal{S}^{m-1}$ than any other $R_{j'} \mathcal{S}^{m-1}$, meaning, the spectral embedding recovers the separation between distinct spheres by approximately preserving signal vector lengths.
\end{remark}

\subsubsection{Sinusoid curve}

\label{ex:cosines}
    Unlike the spheres and nested spheres, a sinusoid curve can be constructed such that it has a numerically low-rank population covariance matrix, rather than a strictly low rank one. Moreover, here we consider a convenient choice of curve parameters that make the population covariance matrix diagonal to leading order, which places this example in the setting
of \Cref{thm:manifold_denoising_approx_diag}.

\begin{prop}[Sinusoid  curve]\label{prop:cosines}
Suppose $\omega(\log^{16}n)= d<n$.
Let $\cX=\{x_1,\dots,x_n\}$ with $x_j=\big(a_i\cos(\omega_i T_j+\phi_i)\big)_{i=1}^d$, where
$T_1,\dots, T_n$ are independent and uniform in $[0,1]$, and the $a_i, \omega_i,\phi_i$ are deterministic. Suppose the amplitudes and frequencies satisfy
\begin{equation}\label{assu:cosines-amplitude}
  \sum_{i=1}^d a_i^2\lesssim \frac{1}{\sqrt d\log^2 n}, \qquad \sum_{i=1}^d a_i^2=o\bigg(\frac{np^3}{\log n}\bigg),\end{equation}
and, with $a_1\geq \dots\geq a_d\geq 0$, that there exists $1\leq k\leq d$ such that \begin{equation}\label{assu:cosines-tail}
      \sum_{j>k}^d a_j^2=o\bigg(\min(d^{-1}, p^3)\bigg),
  \end{equation}
  \begin{equation}\label{assu:cosines-gap}
  a_k^2
  =\omega(d^{-1}),
  \end{equation}
and, 
\begin{equation}\label{assu:cosines-freq}
 \Omega:=\min_{i\neq j}||\omega_i|-|\omega_j||\gtrsim  |a_1|^2d^2, \qquad
 \min_i |\omega_i|\gtrsim 1\,.
\end{equation}
Then for $G\sim\mathcal G_{n,d}(p,\cX)$ with $pn=\omega(\log^{18}n)$, applying the spectral algorithm with this choice of $k$
gives, with probability at least $1-O(n^{-10})$, 
\begin{equation}\label{eqn:cosines_ex_error}
            \|XX^\top- \Xalg \Xalg^\top\|_\op
    \lessapprox 
    n\frac{a_1}{\sqrt{d}} +n\frac{a_1^4}{a_k^2}\sqrt{d}\sum^d_{j=1}a_j^2 + n\frac{a_1^2}{a_k^2\sqrt{d}}\left(\sqrt{\frac{1}{np}}+\frac{(\sum_{j=k+1}^d a_j^2 )^{1/4}}{p}\right),
\end{equation}
and Frobenius norm error
\begin{equation}\label{eqn:cosines_ex_error_frob}
\begin{aligned}
    \|XX^\top- \Xalg \Xalg^\top\|_F
    &\lessapprox
\sqrt{k} \|XX^\top- \Xalg \Xalg^\top\|_\op +\frac{a_{k+1}}{a_k^2}\sqrt{\sum_{l>k}^d a_j^2}.
\end{aligned}
\end{equation}

Suppose additionally that
    \begin{equation}
        \label{assu:cosines-wellcond}
    a_1/a_k=O(1), \qquad k=\omega(1),
    \end{equation}
and that, 
\begin{equation}\label{assu:sinusoid_density_amp}
    \sum_{j=k+1}^d a_j^2=o\big(p^4/d^2\big),\qquad np\geq\omega(d).
\end{equation}
Then on the same event on which \eqref{eqn:cosines_ex_error} holds, we have that with probability at least $1-O(n^{-10})$,
\begin{equation}\label{eqn:norms-recovery-sinusoid-eq}
    \big|\big\{i\in[n]: \|x_i^\alg\|_2=(1+o(1))\|x_i\|_2\big\}\big|=(1-o(1))n.
\end{equation}
\end{prop}

\section{Numerical Illustrations} 
\label{sec:experiments}

We empirically illustrate the performance of the spectral algorithm on various curves and surfaces,
spanning the assumptions of \Cref{thm:lowrank} and \Cref{thm:numerically_low_rk_denoising}:
 three line segments meeting at the
origin, the Lissajous curve, and the unit circle. We specify $\Tr(\Sigma) $ as well as $p$ according to each curve, but we keep $n=3000$ and $d=500$. 
\begin{remark}[On the Numerical Implementation]\label{rmk:implicit_consts_logs}
    The numerical experiments presented in this section are intended primarily as qualitative proofs of concept to motivate the theoretical framework. Developing an optimized, practical implementation of this spectral algorithm falls outside the scope of the present work. As established in our main results, the theoretical error bounds contain several implicit constants and polylogarithmic factors that we haven't optimized for in our proofs. Consequently, identifying the optimal hyper-parameters (e.g., the optimal edge density $p$ or the exact ambient dimension scaling) to minimize the empirical error in finite-sample regimes is a non-trivial task, which depends on the data distribution. A comprehensive empirical study dedicated to fine-tuning these parameters for practical manifold recovery may require additional work.
\end{remark}

For each geometric object example we report three figures: a visualization of the clean
signal, the noisy observation, and the recovered points from our algorithm,  the
empirical spectral distribution (ESD) of the adjacency matrix with the
signal eigenvalues highlighted. 
While the spectral algorithm only recovers points up to rotation, we have chosen an optimal orientation in $\R^k$ for visual convenience.

\subsection{Three line segments}

\begin{figure}[h]
    \centering
     \includegraphics[width=0.75\linewidth]{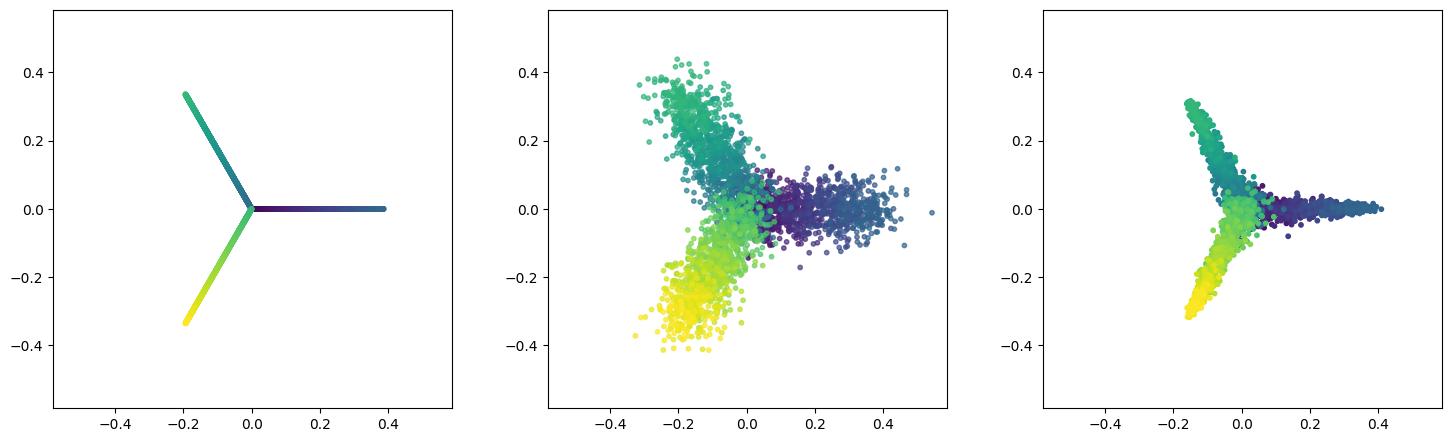}\\[2pt]
    \begin{subfigure}[H]{0.75\linewidth}
         \includegraphics[width=\linewidth]{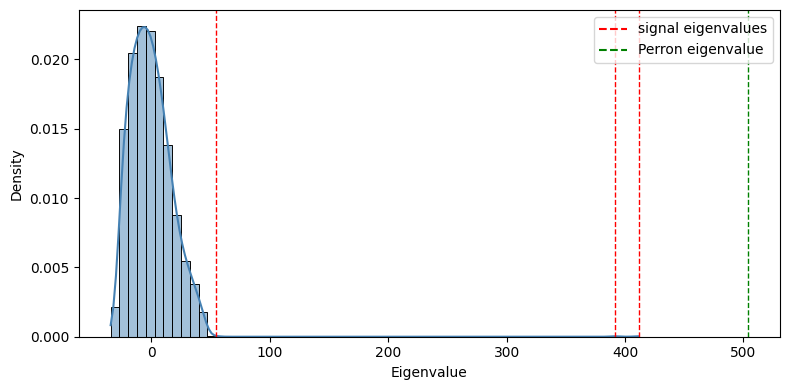}
    \end{subfigure}\\
    \caption{Three line segments meeting at the origin with $(n,d,p)=(3000, 500, 0.1)$, rescaled so that $\Tr\Sigma=0.05$. The spectral algorithm results in $\hat p=0.16$, giving $\hat\tau=\Psi^{-1}(\hat p)/\sqrt d\approx0.045$, with Relative Operator Norm Error:  0.242, 
Relative Frobenius Norm Error: 0.340. Top Left: The clean signal $X$ Top Center: noisy data
$Y = X + Z \in \mathbb{R}^{500}$ with independently sampled
$z_i \sim \mathcal{N}(0, I_d/d)$, projected onto the first 3 coordinates. Top Right: Embedding $\Xalg$ with $k=2$, recovered from Spectral Algorithm. Bottom: the ESD of the adjacency matrix (contains two outlier eigenvalues corresponding to the signal rank).}
    \label{fig:three_line_seg}
\end{figure}

We test the algorithm on data sampled independently and uniformly from three unit-length line
 segments joined at the
origin, in a common subspace of dimension 2, with angular separation $120^\circ$. 
These uniform samples represent our clean data $x_i$, to which we add Gaussian noise 
 of dimension $500$ and variance $1/500$. Though the union of these segments is not a smooth manifold (due to the singularity at the origin), the uniform measure on it has population covariance matrix is rank two, so we run the algorithm with $k=2$. See \Cref{fig:three_line_seg}.

\begin{remark}(On the collapse towards the origin)
    Because the SNGM only depends on whether $\langle y_i,y_j\rangle$ exceeding the threshold $\tau$, a data point $x_i$ with small enough norm relative to $\sqrt{\Tr\Sigma}$ may end up being `disconnected' from the rest of the graph. The graph will carry no directional information about such points, so the spectral algorithm will return a point close to the origin in the embedding. This is consistent with  \Cref{rmk:norms_recovery} (specifically \Cref{cor:relative_recovery}), as the bound is uninformative for small inner products near the origin.
\end{remark}

\subsection{Lissajous Curve}

\begin{figure}[h]
    \centering
     \includegraphics[width=0.75\linewidth]{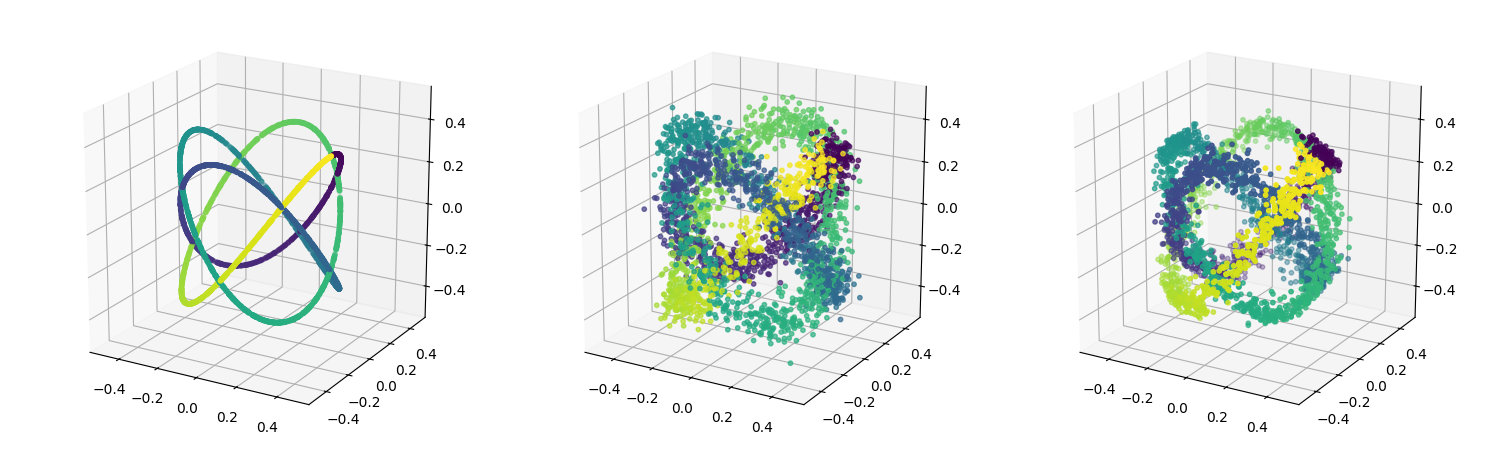}\\[2pt]
    \begin{subfigure}[H]{0.75\linewidth}
         \includegraphics[width=\linewidth]{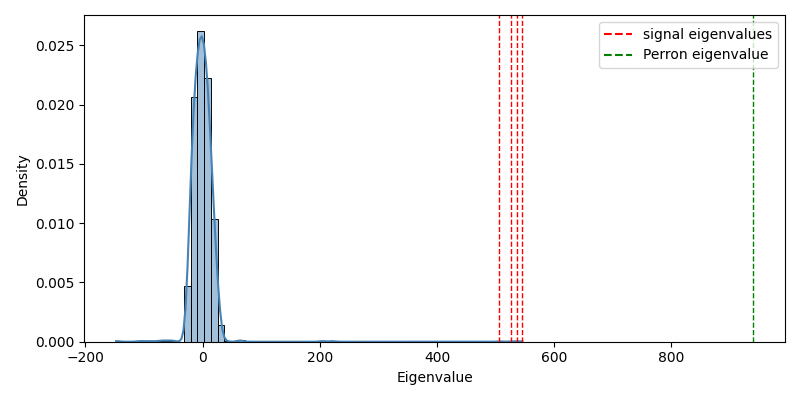}
    \end{subfigure}\\
    \caption{Lissajous curve in $\R^4$, with $(n,d,p)=(3000, 500, 0.08)$, rescaled so that $\Tr\Sigma=0.25$. The spectral algorithm results in $\hat p=0.22$ with $\hat\tau=\Psi^{-1}(\hat p)/\sqrt d\approx0.020$ relative Frobenius norm error=0.220, relative operator norm error=0.2182. Top Left: first 3 coordinates of the clean signal (colored by $t$). Bottom: the adjacency matrix ESD with four outlier eigenvalues.}
    \label{fig:lissajous}
\end{figure}

We sample uniformly from the so-called Lissajous curve, parametrized with $k=4$ by
$$F:[0,1]\to\R^{500},\qquad F(t)=(\cos(6\pi t), \sin(6\pi t), \cos(8\pi t), \sin(8\pi t), 0, \dots, 0) \,,$$
with data scaled so that $\Tr\Sigma=0.25$. For visual purposes, we project onto the first 3 coordinates, and we color it by $t\in[0,1]$ (the color of a data point $y_i$ corresponds to the position of the clean data point $x_i$). See \Cref{fig:lissajous}.

\subsection{Circle with Varying SNR}
\begin{figure}[h]
    \centering
    \includegraphics[width=1\linewidth]{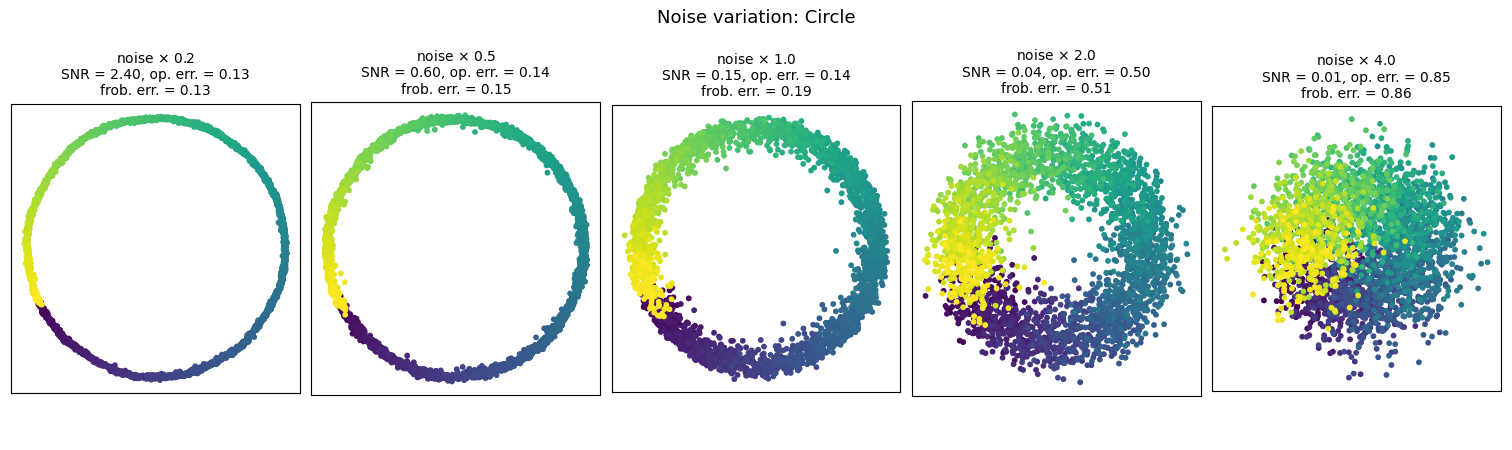}
    \caption{Recovered circle with respect to increasing (from left to right) noise scaling in units of $1/d$ for fixed $n=3000,d=500, p=0.1, \Tr\Sigma=0.15$.}
    \label{fig:circle_noise_seq}
\end{figure}

\Cref{fig:circle_noise_seq} shows the recovered circle as the standard deviation of the noise is scaled by a factor of $s$ i.e. $z_i\sim \cN(\frac{s^2}{d} I_d))$, while the other parameters $(n=3000,d=500, p=0.1, \Tr\Sigma=0.15)$ remain fixed. We define the signal-to-noise (SNR) to be $\text{SNR}:=\Tr\Sigma/s^2$.
The relative error value $0.13-0.19$ holds more or less steady while
$\mathrm{SNR}>0.15$ and then increases gradually toward $1$ from panel 4-5 from the left, once the noise is doubled and quadrupled. This shift illustrates
a Baik--Ben Arous--P\'ech\'e (BBP) phase transition \cite{Baiketal2005} where the signal
eigenvalues detach from the bulk (see for instance, the ESD panels of \Cref{fig:three_line_seg}--\ref{fig:lissajous}). We clearly see that the ``spike'' caused by the signal must exceed a critical value (in our case of order $1/d=1/500$) in order to separate from the bulk. The fourth and fifth panels from the left show the
transition: the circle is still visible at $\mathrm{SNR}=0.04$ but is difficult to tell apart from noise by $\mathrm{SNR}=0.01$, placing the critical SNR for this specific sequences between roughly $0.01$ and $0.04$. 
\section{Proof overviews}
\label{sec:overview}

For clarity of presentation, we will first establish \Cref{thm:main} in the special case that the clean data matrix $X$ is of rank $k$, i.e. $\mu_i=0$ for $k+1\le i\le d$. We state this result as \Cref{thm:lowrank} below, which we note makes a milder assumption on the density $p$. The extension to the full-rank setting is carried out in \Cref{sec: numerically_low_rk}.

\begin{theorem}
[Spectral recovery for low-rank point clouds]
\label{thm:lowrank}
Suppose that $n,d\in \Z_+$, $p\ge0$
satisfy
\begin{equation}    \label{assu:density-lowrank}
    \omega(\log^{16}n)\le d<n\,,\qquad
    \omega\bigg(\frac{\log^2  n}{n}\bigg)\le p\le \frac12-\epsilon 
    \,
\end{equation}
for some fixed $\epsilon>0$. 
Let $X\in \R^{n\times d}$ 
satisfy \eqref{assu:trace} and \eqref{assu:no-outliers}, and
assume 
\begin{align}
    \mu_k &\ge \omega(d^{-1/2})
    \qquad \text{and} \qquad
    \mu_{k+1}=0
    \label{assu:gap}
\end{align}
for some integer $k$.
If $G\sim \mathcal{G}_{n,d}(p,
\mathcal X)$, then for $\Xalg$ as in \eqref{def:embedding}, with 
probability $1-O(n^{-10})$,
\begin{equation}
    \label{eqn:lowrank_cor}
    \norm{XX^\top-\Xalg \Xalg^\top}_\op\lessapprox \frac{n\mu_1}{\sqrt{d}} +\frac{\mu_1^2}{\mu_k^2} \left(  n\sqrt{d}\mu_1^2\Tr\Sigma +  \sqrt{\frac{n}{dp}} \right)\,.
\end{equation}
\end{theorem}

The proof of \Cref{thm:lowrank} breaks into three comparisons, accomplished by \Cref{lem:p1}, \Cref{prop:XhatX} and \Cref{prop:YY_diff} below, the proofs of which are deferred to subsequent sections.

As a first step, it will be convenient for the analysis to replace the scaling factor $\alpha(\hat p)$ in \eqref{def:embedding} with a quantity that is based on the model parameter $p$ (or equivalently, the inner product threshold $\tau$ from \eqref{def:taup}) rather than the estimated edge density $\hat p$. We can do this because the latter concentrates around the former. 
Recalling $\tau$ from \eqref{def:taup}, set
\begin{equation}
\label{def:p1.overview}
         p_1 := \E u_1\1(u_1\ge \tau)\,,
     \end{equation}
for $u=(u_1,\dots,u_{d-k})$ uniform on $\cS^{d-k-1}$. 
(See \eqref{eqn:p1_computation} for a more cumbersome explicit formula.)
With $\mathbf W_k, D_k$ as in \Cref{sec:algorithm}, set
\begin{equation}
    \label{def:hatX}
    \hat X:= \frac1{\sqrt{(d-k)p_1}} \mathbf W_k D_k^{1/2}\,.
\end{equation}
Noting that the scaling parameter $\alpha(\hat p)$ from \eqref{def:scaling} can be expressed as
\begin{equation}
    \label{def:hatp1}
    \alpha(\hat p)= \frac{1}{(d-k)\hat p_1} \qquad \text{for }\quad 
    \hat p_1:= \frac{1}{\sqrt{2\pi (d-k)}}\,e^{-(d-k)\hat \tau^2/2}
\end{equation}
we see that
\begin{equation}    \label{XG-hatX}
    X_GX_G^\top = \frac{p_1}{\hat p_1} \hat X\hat X^\top.
\end{equation}
The following lemma thus allows us to pass from $X_G$ to $\hat X$. 

\begin{lemma}
    \label{lem:p1}
    Under the assumptions of \Cref{thm:main}, with probability $1-O(n^{-10})$,
\[
   \frac{\hat p_1}{p_1}  = 1+ O\Big( \sqrt{\log(1/p)} \Big( \cE_1 + \sqrt{\tfrac{\log n}{np}}\Big) \Big) 
\]
where $\cE_1 = \cE_1(X):=\sqrt{d}\Tr \Sigma + (\Tr \Sigma)^{1/2}\log n + \frac{\log^2n}{\sqrt{d}}$.
\end{lemma}

We prove \Cref{lem:p1} in \Cref{sec:p1}.

Having reduced to comparing $XX^\top$ with $\hat X\hat X^\top$, our next proposition controls the difference between these in terms of the error for a spectral approximation $\hat Y$ to the noisy data matrix $Y$. 
Recalling the labeling $\eta_0\ge \cdots\ge \eta_{n-1}$ for the eigenvalues of $A$, with $(\mathbf w_j)_{j=0}^{n-1}$ denoting a corresponding orthonormal basis of eigenvectors, we set\footnote{Note that we in fact only define the Gram matrix $\hat Y\hat Y^\top$ and do not define $\hat Y\in \R^{n\times d}$ itself. While it can be shown that $\eta_d>0$ with high probability for $d=o(n)$, so that $\hat Y$ could be defined similarly to $\hat X$ in \eqref{def:hatX}, we only ever work with $\hat Y\hat Y^\top$,  with slight abuse of notation to remind the reader that it is an approximation for $YY^\top$.}
\begin{equation}
    \label{def:hatY}
    \hat Y\hat Y^\top := \frac1{(d-k)p_1}\mathbf W_d D_d \mathbf W_d^\top
\end{equation}
where $\mathbf W_d=[\mathbf{w}_1,\dots, \mathbf{w}_d]$ and $D_d:= \diag(\eta_1,\dots, \eta_d)$. 

\begin{prop}
\label{prop:XhatX}
    Under the assumptions of \Cref{thm:lowrank}, with probability $1-O(e^{-\sqrt{d}})$,
\begin{equation}
    \|XX^\top-\hat X\hat X^\top\|_\op \lesssim \frac{n\mu_1}{\sqrt{d}} +
    \frac{\mu_1^2}{\mu_k^2}\|YY^\top- \hat Y \hat Y ^\top\|_\op.
\end{equation}
\end{prop}

We prove \Cref{prop:XhatX} in \Cref{sec:XhatX}, using concentration properties of the Gaussian noise and perturbation inequalities for eigenvalues and spectral projections.

\begin{prop}
\label{prop:YY_diff}
    Under the assumptions of \Cref{thm:lowrank}, with probability $1-O(n^{-10})$, 
    \begin{equation}    \label{bd:YY_diff}
        \big\|YY^\top-\hat{Y}\hat{Y}^\top\big\|_\op\lesssim  n\sqrt{d}\mu_1^2\Tr(\Sigma)\log^4(n)+\left(n\frac{\sqrt{\log(1/p)}}{d^{3/2}}+  \sqrt{\frac{n}{dp}}\right)\log^9 n +\log^4n.
    \end{equation}
\end{prop}

We prove this proposition in \Cref{sec:YY_diff}. 
\Cref{prop:YY_diff} extends Theorems 2.1 and 2.3 from \cite{LiSchramm2024}, which cover the case $k=1$. We closely follow their approach based on expanding the inner product threshold functions $A_{ij} = \1(\langle y_i,y_j\rangle\ge \tau)$ in the basis of Gegenbauer polynomials. See \Cref{sec:YY_diff} for a discussion of the proof ideas.

\begin{proof}
    [Proof of \Cref{thm:lowrank}]
    Combining \eqref{XG-hatX}, \Cref{lem:p1} and Propositions \ref{prop:XhatX} and \ref{prop:YY_diff}, we have
    \begin{align*}
        \|XX^\top - \Xalg\Xalg^\top\|_\op 
        &\le \|XX^\top - \hat X\hat X^\top\|_\op 
        + \Big|1-\frac{p_1}{\hat p_1}\Big| \|\hat X\hat X^\top\|_\op \\
        &\le \Big|1-\frac{p_1}{\hat p_1}\Big| \|XX^\top\|_\op + \Big( 1+ \Big|1-\frac{p_1}{\hat p_1}\Big|\Big) \|XX^\top - \hat X\hat X^\top\|_\op\\
        &\ls \sqrt{\log(1/p)} \Big( \cE_1 + \sqrt{\tfrac{\log n}{np}}\Big) \|XX^\top\|_\op
        + \|XX^\top - \hat X\hat X^\top\|_\op\\
        &\ls \sqrt{\log(1/p)} \Big( \cE_1 + \sqrt{\tfrac{\log n}{np}}\Big)n\mu_1^2 
        + \frac{n\mu_1}{\sqrt{d}} + \frac{\mu_1^2}{\mu_k^2} \| YY^\top - \hat Y\hat Y^\top\|_\op\\
        &\lessapprox n\mu_1^2\bigg(\sqrt{d}\Tr\Sigma  + \frac1{\sqrt{np} }\bigg) 
        + \frac{n\mu_1}{\sqrt{d}} + \frac{\mu_1^2}{\mu_k^2} \bigg( n\mu_1^2 \sqrt{d}\Tr \Sigma + \sqrt{\frac n{dp}} \bigg)\\
        &\lessapprox 
        \frac{n\mu_1}{\sqrt{d}} + \frac{\mu_1^2}{\mu_k^2} \bigg( n\mu_1^2 \sqrt{d}\Tr \Sigma + \sqrt{\frac n{dp}} \bigg)
    \end{align*}
    where in the penultimate line we noted that our assumptions \eqref{assu:density}--\eqref{assu:no-outliers} and $\mu_k =\omega(d^{-1/2})$ imply that the error expression $\cE_1$ from \Cref{lem:p1} is  $\wt O(\sqrt{d}\Tr\Sigma)$, and that the right hand side of \eqref{bd:YY_diff} is $\wt O( n\sqrt{d}\mu_1^2\Tr\Sigma + \sqrt{\frac{n}{dp}}))$, and in the final line we noted that  the term $\mu_1^2\sqrt{n/p}$ is dominated by the term $\sqrt{n/dp}$ since $\mu_1^2\le \Tr\Sigma \ls d^{-1/2}$, by \eqref{assu:trace}.
\end{proof}
\section{Proof of \Cref{prop:XhatX}}\label{sec:XhatX}
We will compare $\hat X\hat X^\top$ with the clean data matrix $XX^\top$ through an intermediary. 
Let
\begin{equation}
    Y_k=\Pi^{(k)}_{Y}Y
\end{equation}
denote the projection of the noisy data matrix $Y$ onto the subspace spanned by the eigenvectors corresponding to the largest $k$ eigenvalues of $YY^\top$.

\Cref{prop:XhatX} follows immediately from the next two lemmas.

\begin{lemma}
[Clean data to projected noisy data] 
\label{lem:projection-low} With hypotheses as in \Cref{thm:lowrank}, we have with probability $1-O(e^{-d})$ that
\[\norm{XX^\top-Y_kY_k^\top}_\op\lesssim n\mu_1\sqrt{\frac{1}{d}}.\]
\end{lemma}

\begin{lemma}[Projected noisy data to embedding] 
\label{lem:perturbation-low} 
With hypotheses as in \Cref{thm:lowrank}, we have that with probability $1-O(e^{-\sqrt{d}})$ that
\begin{equation}
\|Y_kY_k^\top-\hat X\hat X^\top\|_\op\ls 
    \frac{\mu_1^2}{\mu_k^2}\|YY^\top- \hat Y \hat Y ^\top\|_\op\,.
\end{equation}
\end{lemma}
For the proofs of both lemmas we will use the following lemma showing that the top eigenvalues of $YY^\top$ concentrate near those of $XX^\top$, up to a negligible shift.

\begin{lemma}
\label{lem:eval_conc}
    With probability at least $1-e^{-d}$,
    \begin{equation}    \label{YtoX}
        \|Y^\top Y - (X^\top X + \tfrac nd I_d)\|_\op \ls \sqrt{n}\mu_1\,.
    \end{equation}
    By Weyl's inequality, on the same event we have
    \begin{equation}\label{bd:YYlami}
        \lam_i(Y^\top Y) = n\mu_i^2 + \frac nd + O(\sqrt{n}\mu_1) \qquad \forall 1\le i\le d
    \end{equation}
    and in particular, by our assumptions $\mu_1^2\ge \mu_k^2 =\omega(1/d)$ and $d<n$,
    \begin{equation}    \label{bd:YYop}
        \lam_1(Y^\top Y)=\|YY^\top\|_\op \le (1+o(1))n\mu_1^2.
    \end{equation}
    Moreover, with probability at least $1-e^{-\sqrt{d}}$,
    \begin{equation}\label{eqn:interlacing_result}
    \begin{aligned}
           &\lambda_i(Y^\top Y)\ge n\mu_i^2 +\frac{n}{d}-O\left(\sqrt{n}\mu_1d^{-1/4}\right) \qquad \forall 1\leq i\leq k,\\
           &\lambda_{k+1}(Y^\top Y) \le \frac nd +O\left(\sqrt{\frac nd}\right)
 \end{aligned}
    \end{equation}
and consequently,
\begin{equation}\label{eqn:gap_conseq}
\lambda_k(Y^\top Y)-\lambda_{k+1}(Y^\top Y)\geq(1-o(1))n\mu_k^2.
\end{equation}
\end{lemma}

\begin{proof}
We start by proving \eqref{YtoX}. We have
    \begin{align} 
        \Big\|Y^\top Y - (X^\top X + \frac nd I_d)\Big\|_\op
    &=\Big\| X^\top Z + Z^\top X + ( Z^\top Z - \frac ndI_d)\Big\|_\op \notag\\
    & \le 2\|X^\top Z\|_\op + \big\| Z^\top Z - \frac ndI_d\big\|_\op\,.
    \label{spec1}
    \end{align}
    From \cite[Theorem 4.6.1]{Vershynin2018}, we have
    \begin{equation}
        \label{ZZnorm}
        \|Z^\top Z - \frac ndI_d\|_\op \le C\sqrt{n/d}
    \end{equation}
    with probability at least $1-e^{-10d}$, for a universal constant $C>0$.
    Moreover, with $X=UDV^\top$ the singular value decomposition of $X$, with $D,V$ of sizes $d\times d$ and $U$ having $d$ orthonormal columns in $\R^n$, we have
    \begin{equation}\label{eqn:X_svd_noise}
    \|X^\top Z \|_\op = \|D U^\top Z\|_\op \le \|D\|_\op \|U^\top Z\|_\op = \sqrt{n}\mu_1 \|U^\top Z\|_\op.
    \end{equation}
    Now note that by rotational invariance of the distribution of the Gaussian columns of $Z$, the $d\times d$ matrix $U^\top Z$ is identically distributed to a Gaussian matrix $G$ with independent centered entries of variance $1/d$. Again from \cite[Theorem 4.6.1]{Vershynin2018} (or classical bounds) we thus have
    \[
    \|U^\top Z\|_\op \ls 1
    \]
    with probability at least $1-e^{-10d}$.
    By the union bound, \eqref{YtoX} follows by substituting the above bound and \eqref{ZZnorm} into \eqref{spec1}. (Note that $\sqrt{n}\mu_1\ge\sqrt{n}\mu_k \ge \sqrt{n/d}$ by our assumption \eqref{assu:gap-gen}, so the first term in \eqref{spec1} dominates.) 

    Now we estimate $\lambda_i(Y^\top Y)$ using \eqref{YtoX} via Cauchy interlacing. First, recall the decomposition \eqref{eqn: data decomp} and let $a_i = \bar{x}_i+N_i \in \R^k$, where $N_i \sim \cN(0, I_k/d)$ are i.i.d. and $\bar{x}_i$'s are $x_i$'s with the last $d-k$ zeros removed. 
    Let $\bar Y= \bar X+N \in \R^{n \times k}$ where the rows of $\bar X$ are $\bar x_i$ and the rows of $N$ are $N_i$.
   Let $W\in \R^{n\times d-k}$ be formed by the last $d-k$ columns of $Y$, so that $Y=(\bar Y\; W)$.
    Notice that with this coordinate decomposition, one can write 
    \begin{equation}\label{eqn:block_matrix}
        Y^\top Y = \begin{pmatrix}
            \bar Y^\top \bar Y & \bar Y^\top W\\
             W^\top\bar Y & W^\top W
        \end{pmatrix}
    \end{equation}
Then, by the Cauchy Interlacing theorem, we can interlace the eigenvalues of $\lambda_i(YY^T)=\lambda_i(Y^\top Y)$ ($1\leq i\leq d$) with those of $\lambda_i(\bar Y^\top \bar Y)$:
\begin{equation}\label{eqn:interlacing}
    \lambda_i(Y^\top Y)\geq \lambda_i(\bar Y^\top \bar Y), \qquad \forall 1\leq i\leq k.
\end{equation}
Now we repeat the proof of \eqref{YtoX} but with $\bar Y^\top \bar Y$ in place of $Y^\top Y$. Instead of \eqref{spec1} we have that 
\begin{equation}\label{eqn: barY_conc}
\norm{\bar Y^\top \bar Y -(\bar X^\top\bar X+ \frac{n}{d}I_k)}_\op\leq 2\norm{\bar X^\top N} + \norm{N^\top N- \frac{n}{d}I_k}_\op. 
\end{equation}
Similarly to the reduction \eqref{eqn:X_svd_noise}, we use the truncated singular value decomposition of $X$: $\bar X= U_kD_kV_k^\top$ where $U_k,V_k\in \R^{n\times k}$ are the top$-k$ right and left singular vectors of $X$ (or equivalently, $\bar X$). Thus, by \cite[Theorem 4.6.1]{Vershynin2018} 
, with probability at least
$1-\exp( - 10d^{1/2})$,
\[\norm{U_k^\top N}_\op\le\frac{2\sqrt k+
O(d^{1/4})}{\sqrt d}.\]
This is due to the fact that $U_k^\top N$ is $k\times k$.
Moreover, analogously to \eqref{ZZnorm},
with probability $1-\exp( - 10d^{1/2})$,
\[\norm{N^\top N -\frac{n}{d}I_k}_\op\lesssim \frac{\sqrt{nk}+\sqrt n d^{1/4}+k+\sqrt d}{d}\lesssim  \sqrt n\mu_1 \frac{1}{d^{1/4}},\]
where the last step uses $\mu_1\geq\mu_k=\omega(d^{-1/2})$  and $k=o(\sqrt{d}/\log^{2}n)$ by \Cref{rmk:k}. 
Now we apply the upper bounds from previous two displays to \eqref{eqn: barY_conc} to get that
\begin{equation}
\label{eq:YbartoXbar}
    \norm{\bar Y^\top \bar Y -(\bar X^\top\bar X+ \frac{n}{d}I_k)}_\op \lesssim \frac{\sqrt{n}\mu_1}{d^{1/4}}.
\end{equation}

Now we combine \eqref{eq:YbartoXbar} with Weyl's inequality to estimate eigenvalues of $\bar Y^\top \bar Y$ and thus the eigenvalues of $Y^\top Y$ via \eqref{eqn:interlacing}. Thus we obtain the first line of \eqref{eqn:interlacing_result}.
To prove the second line of \eqref{eqn:interlacing_result} we again use the Cauchy Interlacing Theorem, now with $\lambda_{k+1}(Y^\top Y)\leq \lambda_1(W^\top W)$. By \cite[Theorem 4.4.5]{Vershynin2018} with probability at least $1-e^{-d}$
\[\lambda_1(W^\top W)=\norm{W}_\op^2\le\frac{(\sqrt n+2\sqrt d)^2}{d}\le\frac nd+C\sqrt{\frac nd}.\]
Finally, \eqref{eqn:gap_conseq} is proved using \eqref{eqn:interlacing_result} with $i=k$, and taking a difference:
\begin{align*}
    \lambda_k(Y^\top Y)-\lambda_{k+1} (Y^\top Y)
    &\geq n\mu_k^2 + \frac{n}{d} -O(\sqrt{n}\mu_1 d^{-1/4})+O\left(\sqrt{\tfrac{n}{d}}\right)\\
    &=n\mu_k^2\left(1+ \frac{1}{d\mu_k^2} + O\left(-\tfrac{\mu_1}{\sqrt{n}d^{1/4}\mu_k^2}+\tfrac{1}{\sqrt{nd}\mu_k^2}\right)\right).
\end{align*}
\eqref{assu:gap-gen} implies that the term $ \frac{1}{d\mu_k^2}=o(1)$. As for the error term, \eqref{assu:trace}, \eqref{assu:gap-gen} and $d<n$ imply that $\tfrac{\mu_1}{\sqrt{n}d^{1/4}\mu_k^2}\leq \tfrac{1}{\sqrt{dn\log n}\mu_k^2}=o(1)$, $\tfrac{1}{\sqrt{dn}\mu_k^2}=o(1)$.
\end{proof}

\begin{proof}[Proof of \Cref{lem:projection-low}]
By the triangle inequality and sub-multiplicativity of the operator norm,
\begin{equation}\label{eqn:YtoX_triangle_ineq}
\begin{aligned}
   \norm{XX^\top-Y_kY_k^\top}_\op &\leq \norm{XX^\top - XY_k^\top}_\op+ \norm{XY_k^\top- Y_kY_k^\top}_\op\\
   &\leq (\norm{X}_\op+ \norm{Y_k}_\op)\norm{X-Y_k}_\op \,.
\end{aligned}
\end{equation}
Then, combining a triangle inequality and the fact that $Y_k$ is the best rank-$k$ approximation to $Y$ (the Eckart--Young--Mirsky theorem):
\[
\norm{X-Y_k}_\op\leq \norm{X-Y}_\op + \norm{Y-Y_k}_\op\leq 2\norm{X-Y}_\op.
\]
From \eqref{bd:YYop} and \eqref{ZZnorm} we have that with probability $1-O(e^{-d})$,
$\|Y_k\|_\op=\|Y\|_\op \ls \sqrt{n}\mu_1=\|X\|_\op$ and $\|X-Y\|_\op = \|Z\|_\op \ls \sqrt{n/d}$. Substituting these bounds into \eqref{eqn:YtoX_triangle_ineq} yields the claim.
\end{proof}
For the proof of \Cref{lem:perturbation-low} we use the following classical result (see e.g.\ \cite[Theorem 2]{yu2015useful}):

\begin{lemma}[Davis--Kahan inequality for spectral projections]\label{DavisKahanForProjections}
Consider two symmetric matrices A and B with spectral decompositions $A= \sum_{i=1}^n\lambda_i u_iu_i^\top$ and $B= \sum_{i=1}^n \mu_i v_iv_i^\top$,
in which the eigenvalues are (weakly) decreasing. Then the spectral projections $
\Pi^{(k)}_{A}=\sum_{i=1}^k u_iu_i^\top, \Pi^{(k)}_{B}=\sum_{i=1}^k v_iv_i^\top$ satisfy
\[
\norm{\Pi^{(k)}_{A}-\Pi^{(k)}_{B}}_\op \leq 2\frac{\norm{A-B}_\op}{\lambda_k-\lambda_{k+1}}.
\]
\end{lemma}

\begin{proof}[Proof of \Cref{lem:perturbation-low}]
Recall that $\Pi^{(k)}_{M}$ denotes the projection matrix for the span of the top $k$ eigenvectors of a symmetric matrix $M$. For brevity, in this part of the proof, we let $ \Pi= \Pi^{(k)}_{Y}$ and $\hat \Pi= \Pi^{(k)}_{\hat{Y}}$. Then we decompose $Y_kY_k^{\top}-\hat X\hat X^\top$ using projections as follows:
\begin{align}
Y_kY_k^
{\top} - \hat X\hat X^\top 
&= \Pi  YY^{\top } \Pi  - \hat\Pi \hat{Y}\hat{Y}^\top \hat\Pi \notag\\
&= \underbrace{(\Pi  YY^{\top } \Pi  - \hat\Pi YY^{\top }  \Pi )}_{(i)} + \underbrace{(\hat\Pi YY^{\top } \Pi  - \hat\Pi YY^{\top }  \hat\Pi)}_{(ii)} +
\underbrace{(\hat\Pi YY^{\top }  \hat\Pi - \hat\Pi \hat{Y}\hat{Y}^\top \hat\Pi)\,.}_{(iii)}
\label{eqn:proj_decomp_YY_hatYY}
\end{align}
We dedicate the rest of the proof to estimating the operator norms of $(i)-(iii)$
:
\begin{itemize}
    \item[$(i+ii)$]
    By sub-multiplicativity of the operator norm
    and \Cref{DavisKahanForProjections},
\begin{align*}
\norm{\Pi YY^
{\top}\Pi -\hat\Pi YY^
{\top}\Pi }_\op &\leq \|YY^
{\top}\Pi \|_\op \|\Pi -\hat\Pi\|_\op\\
&\leq 2\norm{YY^\top}_\op\frac{\norm{YY^
{\top}-\hat{Y}\hat{Y}^\top}_\op}{\lambda_k(YY^
{\top})-\lambda_{k+1}(YY^
{\top})}.
\end{align*} 
By similar lines we obtain the same bound for $(ii)$.
\item[(iii)] 
By sub-multiplicativity of the operator norm:
\begin{align*}
 \|\hat\Pi YY^
{\top}\hat\Pi - \hat\Pi\hat{Y}\hat{Y}^\top \hat\Pi\|_\op &\leq \norm{\hat\Pi}_\op^2\|YY^{\top} - \hat{Y}\hat{Y}^\top\|_\op
\le \|YY^{\top} - \hat{Y}\hat{Y}^\top\|_\op. 
\end{align*}
\end{itemize}
Combining the bounds for $(i)$--$(iii)$ with \eqref{eqn:proj_decomp_YY_hatYY} and the triangle inequality yields
\begin{equation*}
\norm{Y_kY_k^\top-\hat X\hat X^\top}_\op\lesssim \norm{YY^\top- \hat Y \hat Y ^\top}_\op\left(\frac{\norm{YY^\top}_\op}{\lambda_k(YY^\top)-\lambda_{k+1}(YY^\top)}+1\right)
\asymp \frac{\norm{YY^\top}_\op\norm{Y_kY_k^\top-\hat X\hat X^\top}_\op}{\lambda_k(YY^\top)-\lambda_{k+1}(YY^\top)}
\end{equation*}
Finally, from \Cref{lem:eval_conc} we have  $\|YY^\top\|_\op \le (1+o(1)) n\mu_1^2$ and
$\lambda_k(YY^\top)-\lambda_{k+1}(YY^\top )\geq (1-o(1))n\mu_{k}^2$ with probability $1-O(e^{-\sqrt{d}})$. Substituting these estimates in the above display yields the claim.
\end{proof}

\section{Proof of \Cref{prop:YY_diff}}
\label{sec:YY_diff}

\subsection{Setup and overview}

We broadly follow the approach of \cite{LiSchramm2024}, expressing the entries of $A$ as functions of inner products of unit vectors and expanding in the basis of Gegenbauer polynomials. 
Since the distribution of the Gaussian matrix $Z$ is invariant under multiplication on the left by fixed orthogonal matrices, and the SNGM construction only depends on the inner products, we may assume without loss of generality that the subspace spanned by rows of $X$ aligns with the first $k$ canonical coordinates. Thus, we write
\begin{equation}\label{eqn: data decomp}
y_i=(a_i, w_i)\in \R^d,\qquad a_i\in \R^{k},\qquad w_i\sim \mathcal{N}(0,\frac{1}{d}I_{d-k}),\end{equation}
where $a_i \in \R^k$ contains the coordinates of $x_i$ plus the first $k$ coordinates of the noise, and $w_i \sim \mathcal{N}(0, \frac{1}{d}I_{d-k})$ represents the orthogonal noise component. We further decompose the Gaussian noise vector $w_i\in \R^{d-k}$ into a magnitude $l_i=\norm{w_i}_2$ and direction $v_i=\frac{w_i}{\norm{w_i}_2}$, thus
\begin{equation}
    \label{yalv}
    y_i=(a_i, l_iv_i)\,.
\end{equation}
Note that $a_i, l_i$ and $v_i$ are jointly independent. 
For each $i\neq j\in [n]$, the $(i,j)$-th entry of the adjacency matrix is thus given by
\begin{equation}\label{eq:indicator_shift}
A_{ij}=\1(\langle y_i,y_j\rangle \geq \tau)=\1(\ip{a_i}{a_j}+l_il_j\ip{v_i}{v_j}\geq \tau)=\1\left(\langle v_i, v_j \rangle \geq 
\tau^{ij} 
\right) ,
\end{equation}
where we define the (random) thresholds 
\[
\tau^{ij}:=\frac{\tau-\ip{a_i}{a_j}}{l_il_j}.
\] 
In \Cref{lemma:p_tauij_conc} we show that $\tau^{ij}$ are concentrated around $\tau$. Set
\begin{equation}    \label{def:vph}
    \vph(x):= \1(x \ge \sqrt{d-k}\tau)\,,\qquad 
    \vph_{ij}(x):= \1(x \ge \sqrt{d-k}\tau^{ij})\,.
\end{equation}
We expand these functions in the basis of Gegenbauer polynomials $(q^{(d-k)}_\ell)_{\ell\ge0}$. 
Recall that $(q^{(m)}_\ell)_{\ell\ge0}$ is an orthogonal basis for $L^2([-\sqrt{m},\sqrt{m}], \cD_m)$, where $\cD_m$ denotes the distribution of $\sqrt{m}\langle u_1,u_2\rangle$ for independent uniformly drawn elements of the unit sphere $\cS^{m-1}$ (thus, $\cD_m$ is the distribution of any fixed entry of $\sqrt{m}u_1$). 
The first three are given by
\begin{equation}
    q_0^{(m)}(x)= 1\,,\qquad 
    q_1^{(m)}(x) = x\,, \qquad
    q_2^{(m)}(x) = \frac1{\sqrt2} \sqrt{ \frac{m+2}{m-1}} (x^2-1)\,.
\end{equation}
We define expansion coefficients $p_\ell, p^{ij}_\ell$ with
\begin{equation}    \label{phi-expand}
    \vph(x) = \sum_{\ell=0}^\infty \sqrt{N_\ell^{(d-k)}} p_\ell q_\ell^{(d-k)}(x)\,,
    \qquad
    \vph_{ij}(x) = \sum_{\ell=0}^\infty \sqrt{N_\ell^{(d-k)}} p^{ij}_\ell q_\ell^{(d-k)}(x)
\end{equation}
where we have scaled the degree-$\ell$ parts by the square-roots of integers
\begin{equation}
    N_0^{(m)}=1\,, \qquad N_\ell^{(m)} = \frac{2\ell+m-2}{\ell}\binom{\ell+m-3}{\ell-1}\,,\quad \ell\ge1,
\end{equation}
giving the cardinality of the orthonormal degree-$\ell$ spherical harmonics on $\cS^{m-1}$.
(Note that for the case $\ell=1$ the above definition for $p_1$ agrees with our earlier definition \eqref{def:p1.overview}.)
For a review of spherical harmonics and Gegenbauer polynomials see for instance \cite{Hochstadt,dai2013approximation,EfFr}.

The orthogonal expansion \eqref{phi-expand} of $\vph_{ij}$ induces an expansion 
\begin{equation}    \label{Aexpand}
A= \sum_{\ell=0}^\infty A^{(\ell)} \,,\qquad \text{where}\quad
    A^{(\ell)}_{ij} = 
    \sqrt{N_\ell^{(d-k)}} p^{ij}_\ell q_\ell^{(d-k)}(\sqrt{d-k}\langle v_i,v_j\rangle)\,.
\end{equation}
In particular,
\begin{equation}    \label{Aexpand12}
    A^{(0)} = P_0 \,, \qquad
   A^{(1)} = (d-k) P_1\odot (VV^\top)
\end{equation}
where $V^\top = (v_1^\top,\dots, v_n^\top)$, 
\begin{equation}\label{def:Pell}
    P_\ell:= (p_\ell^{ij}1_{i\ne j})_{i,j=1}^n
\end{equation}
and we recall our notation $\odot $ for the entrywise product of matrices. We use the Gegenbauer expansion \eqref{Aexpand} of the adjacency matrix $A$ to bridge the gap between $YY^\top$ and $\hat Y \hat Y ^\top$.

\begin{defn}[Data, spherical and spectral approximations] \label{def:Bdata_Bsph}
We set
    \begin{equation}
        \boxed{
        \begin{aligned}
            B_\data &:= p_0\zeta\zeta^\top + d_kp_1 YY^\top, \\
            B_\sph &:= P_0 + d_kP_1\odot (VV^\top)\\
            B_\spec &:= \eta_0\mathbf{w}_0\mathbf{w}_0^\top + d_k p_1\hat{Y}\hat{Y}^\top
        \end{aligned}
        }
    \end{equation}
    where $\zeta\in\R^n$ has entries
    \begin{equation}
        \zeta_i := 1+\frac{(l_i-1)p_1d_k\tau}{p_0}
    \end{equation}
Here and in the sequel we write
\begin{equation}
    \label{def:dk}
    d_k:= d-k\,.
\end{equation}
\end{defn}
Recalling \eqref{Aexpand12}, we see that $B_\sph=A^{(0)}+A^{(1)}$ collects the order 0 and 1 terms of the Gegenbauer expansion of $A$.

For the proof of \Cref{prop:YY_diff} we note that 
\begin{align}
    d_kp_1 (YY^\top - \hat Y\hat Y^\top)
    &= B_\data - B_\spec + \eta_0\eigv_0\eigv_0^\top -p_0\zeta\zeta^\top   \notag\\
    &= (B_\data - A)+(A - B_\spec) + (\eta_0\eigv_0\eigv_0^\top -p_0\zeta\zeta^\top   ).\notag
\end{align}
Since $B_\spec$ is the best rank-$(d+1)$ approximation to $A$ under the operator norm, and $B_\data$ has rank at most $d+1$, we may thus bound
\begin{align}
    d_kp_1\norm{\hat Y \hat Y ^\top- YY^\top}_\op
    &\leq
    \|B_\spec-A\|_\op + \|A-B_\data
    \|_\op
    + \norm{\eta_0\mathbf{w}_0\mathbf{w}_0^\top- p_0\zeta\zeta^\top}_\op \notag\\
    &\leq 2\norm{A-B_\data}_\op + \norm{\eta_0\mathbf{w}_0\mathbf{w}_0^\top- p_0\zeta\zeta^\top}_\op \notag\\
    &\leq 2\norm{A-B_\sph}_\op+ 2\norm{B_\sph-B_\data}_\op +\norm{\eta_0\mathbf{w}_0\mathbf{w}_0^\top- p_0\zeta\zeta^\top}_\op  \label{eqn:prop_YY_diff_chain}
\end{align}
The three terms in the final line are bounded in Lemmas \ref{lemma:trace_method}, \ref{lemma:Bsph_Bdata_diff} and \ref{lem:eigenvector_diff}.
\subsection{Parameter Estimates}
\label{subsubsection:Parameter_estimation}

We require some preliminary parameter estimations to begin our analysis. Here, we only present the results, the proofs of these parameter estimation lemmas are standard concentration arguments, and can be found in \Cref{Appendix Section:parameter_estimation_proofs}. Although our parameter estimates are stated in terms of the original parameters, they are proved for the truncated versions of the relevant quantities (denoted $\tilde x_i,\tilde y_i, \tilde A, \tilde p$, etc.) -- which are introduced in \Cref{sec: numerically_low_rk}, the reader can ignore the tildes throughout the proofs in the appendix: under the exact low-rank assumption \eqref{assu:gap-gen}, truncation is trivial and $\tilde x_i=x_i,\tilde y_i= y_i,\tilde A=A, \tilde p=p$ etc. We adopt this convention so that the same proofs apply verbatim to the generalized numerically low-rank case treated in \Cref{sec: numerically_low_rk}.

\begin{lemma}
\label{lemma:ai_aj_li_lj_conc}
Under the same conditions as \Cref{thm:lowrank} then uniformly over all pairs $(i,j)\in [n]^2$ (including $i=j$) with probability at least $1-n^{-12}$:
\begin{align*}
\ip{a_i}{a_j}&=O\left(\Tr\left(\Sigma\right)\log n\right), \qquad l_i l_j = 1 + O\left(\sqrt{\frac{\log n}{d}}\right).
\end{align*}

\end{lemma}
See \Cref{proof of lemma:ai_aj_li_lj_conc} for the proof of the above Lemma. \newline
We need the following lemma estimating the value of the inner product threshold parameter $\tau$ defined in \eqref{def:taup}.

\begin{lemma}[Estimating Connectivity Threshold] \label{lemma:conc_tau}
Under the same conditions as in \Cref{lemma:ai_aj_li_lj_conc}
\[
\tau = (1+O\bigg(\frac1{\log^2n}\bigg))\sqrt{\frac2d\log (1/p)}.
\]
\end{lemma}
See \Cref{proof of lemma:conc_tau} for the proof. We will oftentimes opt to use the weaker result $\tau\asymp\sqrt{\log(1/p)/d}$ for brevity.
\begin{lemma}[Shifted and Original Gegenbauer Coefficient Estimates]
\label{lemma:p_tauij_conc}
    Under the assumptions as \Cref{lemma:conc_tau} and \Cref{lemma:ai_aj_li_lj_conc}, then uniformly with probability at least $1-n^{-12}$ over all pairs $(i,j)\in [n]^2$:
    \[p_0^{ij}=O(p),\quad p_1^{ij}= O(p\tau),\quad \tau^{ij}= \tau+O\left(\frac{1}{\sqrt{d}\log n}\right).\]
    Moreover,
    \[p_0=\Theta(p), \quad p_1=\Theta (p\tau).\]
    
\end{lemma}
See \ref{proof of lemma:p_tauij_conc} for the proof.

\subsubsection{
Parameter estimates in different regimes} 

In our proof of the ``Spherical to Data Error'' (\Cref{lemma:Bsph_Bdata_diff}), we will consider two regimes -- $\Tr\Sigma\leq \tau^2$ and $\Tr\Sigma>\tau^2$. 

\begin{remark}[On two regimes]
When $\Tr \Sigma \leq \tau^2$, the signal is quite weak so the data is very close to pure Gaussian noise, in which case $P_0 \approx p_0 \zeta\zeta^\top$ and $P_1 \approx p_1 \1\1^\top$. On the other hand, when $\Tr \Sigma > \tau^2$, the signal is a sufficiently strong perturbation of Gaussian noise to require a more refined analysis.
\end{remark}

\begin{lemma}\label{fact:spherical_matrix_norm}
    Let $v_1, \dots, v_n\in \mathcal{S}^{d-1}$ be uniform random vectors in $\mathcal{S}^{d-1}$ with $n>d$. Then their Gram matrix with the diagonal set to zero satisfies $\norm{VV^\top -I}_{\op}\leq O(\frac{n}{d})$, and hence $\norm{VV^\top-I}_{F}\leq O(\frac{n}{\sqrt{d}})$,  with probability at least $1-e^{-cd}$ for some $c>0$. 
\end{lemma}

\begin{proof}  
This follows from standard matrix concentration results, see e.g. \cite[Theorem 4.6.1]{Vershynin2018}.
\end{proof}
\begin{defn}\label{def:parameter def}
Set
\[
L_i:=l_i-1
\quad \qquad
p_0=\Prob_{\xi\sim \mathcal{D}_{d_k}}(\xi\geq \sqrt{d_k}\tau)
\qquad
p_1=\frac{1}{\sqrt{d_k}}\E_{\xi\sim \mathcal{D}_{d_k}}[\xi\1(\xi\geq \sqrt{d_k}\tau)].
\]
We further define the fluctuations of $\tau^{ij}, p_0^{ij}, p_1^{ij}$ as follows
\begin{equation}\label{def:Delta_ps}
\Delta \tau^{ij}:= \tau^{ij}-\tau
\qquad
\Delta p_0^{ij}:=p_0^{ij}-p_0
\qquad
\Delta p_1^{ij}:=p_1^{ij}-p_1.
\end{equation}
\end{defn}

\begin{lemma} [Regime Dependent Estimates of Parameter Fluctuations]\label{lemma:scalar_estimates_whp}
    Under the same assumptions as \Cref{thm:lowrank}, then with probability at least $1-O(n^{-12})$ over all $i,j$, the following universal estimates hold independently of the trace regime:
    \begin{equation*}
        L_i = O\left(\sqrt{\frac{\log n}{d}}\right), \quad
        l_i^{-1} = 1 - L_i + O\left(\frac{\log n}{d}\right),
    \end{equation*}
     and 
    \begin{equation}
        \forall i,j \ (i\neq j):
        \Delta p_1^{ij} = \tau \Delta p_0^{ij}(1+o(1))\quad \text{\emph{and}} \quad 
        \langle v_i, v_j \rangle = o\left(\frac{\log n}{\sqrt{d}}\right).
    \end{equation}
    Furthermore, the remaining approximations satisfy the following:
    
\renewcommand{\arraystretch}{2.2}
    \resizebox{\linewidth}{!}{
    \begin{tabular}{l | l}
        \toprule
        \textbf{Small Trace:} $\Tr(\Sigma) \leq \tau^2$ & \textbf{Large Trace:} $\Tr(\Sigma) > \tau^2$ \\
        \midrule
        $\displaystyle \frac{\langle a_i,a_j\rangle}{\tau}
            = O\big(\frac{\log n}{\sqrt{d}}\big)$ &
        $\displaystyle \langle a_i, a_j \rangle = O\big(\Tr(\Sigma)\log n\big)$ \\

        $\displaystyle \Delta\tau^{ij} = -\tau\left(\frac{\langle a_i,a_j\rangle}{\tau} + L_i + L_j + o\left(\frac{\log^{3/2} n}{d}\right)\right)$ &
        $\displaystyle \Delta\tau^{ij} = -\langle a_i,a_j\rangle\left(1 - L_i - L_j + O\left(\frac{\log n}{d}\right)\right) -\tau\left(L_i + L_j + O\left(\frac{\log n}{d}\right)\right)$ \\

        $\displaystyle \Delta p_0^{ij} = -p_1 d_k \Delta\tau^{ij} \big(1 + O(\tau^2\sqrt{d}\log n)\big)$ &
        $\displaystyle \Delta p_0^{ij} = -p_1 d_k \Delta\tau^{ij} \big(1 + O(\Delta \tau^{ij}\tau d)\big)$ \\
        \bottomrule
    \end{tabular}}
\end{lemma}

See \Cref{proof of lemma:scalar_estimates_whp} for the proof.

\begin{lemma}[Operator Norm Bounds for Threshold Matrix Fluctuations]\label{lemma: DeltaTau_norm_estimates}
Extending the entry-wise the definition of the parameters in \Cref{def:parameter def}, consider the matrix
\begin{equation}\label{eqn:Delta_Tau}
\Delta\tau = -aa^\top(1 + O\left(\frac{\log n}{d}\right)) + \left[ aa^\top - \tau\1\1^\top\right] \odot \left(\mathbf{1}L^\top + L\mathbf{1}^\top\right)-\tau O\left(\frac{\log n}{d}\right) \mathbf{1}\mathbf{1}^\top .
\end{equation}
We have defined $a=(a_1,\dots, a_n)^\top \in\R^{n\times d}$, $L=(L_1,\dots, L_n)^\top\in \R^n$ consistent with the decomposition \eqref{eqn: data decomp}. In addition, let $V\in \R^{n\times d}$ be the matrix with rows $v_i$ (consistent with \eqref{def:parameter def} and \Cref{fact:spherical_matrix_norm}). Then with probability at least $1-O(n^{-12})$
\begin{enumerate}
    \item[(a)] $\norm{\Delta\tau \odot VV^\top}_\op\lesssim \frac nd \left(\Tr\Sigma + \tau^2\right)\log n,$ 
    \item[(b)] $\norm{\Delta\tau^{\odot 2}\odot VV^\top}_\op\lesssim \frac nd \left((\Tr\Sigma \log n)^2+ \tau^2 \log (n)/\sqrt{d}\right)$, 
    \item[(c)] $\norm{\Delta\tau^{\odot 2}}_\op
            \lesssim n\mu_1^2\,\Tr\Sigma\,\log n + \dfrac{n\tau^2\log^3 n}{d}.$
\end{enumerate}
\end{lemma}

See \Cref{proof of lemma: DeltaTau_norm_estimates} in the appendix for these computations.

We will use the following extension of \cite[Proposition 4.1]{LiSchramm2024}, obtained by straightforward modifications of the proof in that work.

\begin{lemma}[Approximation to the Truncated Adjacency Matrix]
\label{lemma:trace_method}
    For $\Tr\left(\Sigma\right)\leq 1/(\sqrt{d}\log^2 n)$, $d\geq \omega(\log^4n)$, and $k=o(\sqrt{d}/\log^2 n)$,
    \[
    \lVert A- (P_0+ d_kP_1\odot VV^\top)\rVert_{op}\leq o\left( \max\{np\tau^2,\sqrt{np}\}\log^9 n\right),
    \]
    with probability at least $1-n^{-10}$. 
\end{lemma}

\begin{proof}
Given that the magnitude of $\tau^{ij}$ established in \Cref{lemma:conc_tau} matches that of the GMBM studied in \cite{LiSchramm2024}, the proof of this lemma follows verbatim the proof of \cite[Proposition 4.1]{LiSchramm2024}, replacing $d-1$ with $d-k$ where appropriate (one verifies that all steps are unchanged if $k\le cd$ for a sufficiently small constant $c>0$). 
\end{proof}

\subsection{Spherical to Data Error} In the following division of this section, we will focus on controlling the spherical to data error component $\norm{B_\sph -B_\data}_\op$. The proof is split into different regimes according to the size of the trace.

\begin{lemma}[Data Matrix is well approximated by Spherical Component]
\label{lemma:Bsph_Bdata_diff}
Under the conditions as \Cref{thm:lowrank},
with probability at least $1-O(n^{-12})$
:
\begin{equation*}
   \big\lVert B_\sph-B_\data \big\rVert_{op} \lesssim np\big(\tau^2+d\mu_1^2(\Tr\Sigma)\big) \log^4 (n).
\end{equation*}

\end{lemma} 
In the proof of \Cref{lemma:Bsph_Bdata_diff} we make frequent use of the following:
\begin{fact}[Infinity norm bound]\label{fact:infinity_norm_bound}
\begin{equation}
    \norm{M}_{op}\leq \norm{M}_F\leq n\norm{M}_\infty
    \end{equation}
    for an $n\times n$ matrix $M$, where $\|M\|_\infty:= \max_{i,j} |M_{ij}|$.
\end{fact}

    \begin{fact}[Hadamard Product Bound {\cite[7.5.P24]{horn2012matrix}}] \label{fact:hadamard_product_bound}
    If $A$ is positive semi-definite, then
    \[\norm{A\odot B}_\op\leq \norm{B}_\op 
    \|A\|_\infty.\]
    \end{fact}

\begin{proof}[Proof of \Cref{lemma:Bsph_Bdata_diff}]
We want an upper bound on the operator norm of the matrix
    \begin{equation}
        \label{Bsph-Bdata}
        B_\sph-B_\data=P_0 + d_k P_1\odot VV^\top-(p_0\zeta\zeta^\top + d_k p_1  YY^\top )
    \end{equation}
    
    \textbf{Adding an artificial diagonal: }We originally defined $P_0,P_1$ to have no diagonal, however we can introduce an artificial diagonal
    \begin{equation}\label{eqn:artif_diag_Bsph}
        \diag(B_\sph)_{ii}:=p_0^{ii} -d_kp_1^{ii} \norm{v_i}_2^2.
    \end{equation}
    The computation in \Cref{lemma:p_tauij_conc} holds for $i=j$, so the operator norm of this diagonal is of order:
    \begin{equation}\label{eqn:diag_Bdata_Bspher_diff}
    \diag(B_\sph)_{ii}=O(p)+O(dp\tau)=O(p\sqrt{d\log n}),
    \end{equation}
    where we used estimates from \Cref{lemma:p_tauij_conc}, \Cref{lemma:conc_tau} to estimate $p_1,\tau$ as well as the fact that $\norm{v_i}_2^2=1$. 
    Thus, \begin{equation}\label{eqn:diag_op_norm}
        \norm{\diag(B_\sph)}_\op\lesssim O(p\sqrt{d\log n}),
    \end{equation}
    with probability at least $1-O(n^{-12})$. \eqref{eqn:diag_op_norm} is much smaller than the resulting error we are looking for (see the theorem statement \Cref{lemma:Bsph_Bdata_diff}). We call the matrices that include $p^{ii}_0, p_1^{ii}$ on their diagonals $P_0^{\diag},P_1^{\diag}$ respectively.
    
We decompose the matrix entries in \eqref{Bsph-Bdata} into four components:
\begin{align*}
    &(B_\sph-B_\data)_{ij} \\
    &=p_0^{ij}+d_k p_1^{ij}\langle v_i,v_j\rangle-p_0(1+L_i p_1d_k\tau/p_0)(1+L_j p_1d_k\tau/p_0)-d_k p_1(\langle a_i,a_j\rangle+l_il_j\langle v_i,v_j\rangle)\\
    &= \Delta p_0^{ij} + d_k\langle v_i,v_j\rangle (p_1^{ij} -p_1 l_il_j) - \frac{L_iL_jp_1^2 d_k^2\tau^2}{p_0}-d_kp_1\langle a_i,a_j\rangle -(L_i+L_j)p_1d_k\tau\\
    &= \underbrace{(\Delta p_0^{ij}-d_kp_1\langle a_i,a_j\rangle -(L_i+L_j)p_1d_k\tau)}_{=:f_1^{ij}}\underbrace{- \frac{L_iL_jp_1^2d_k^2\tau^2}{p_0}}_{=:f_2^{ij}}+\underbrace{\Delta p_1^{ij}d_k\langle v_i,v_j\rangle}_{=:f_3^{ij}}\underbrace{-d_kp_1\langle v_i,v_j\rangle (L_i+L_j + L_iL_j)}_{=:f_4^{ij}}
\end{align*}
Then, define the matrices $f_1, f_2, f_3, f_4$ 
with entry $i,j$ equal to $f_1^{ij}, f_2^{ij}, f_3^{ij}, f_4^{ij}$. We control the operator norm of each of these matrices individually. We note that the operator norm of $f_2, f_4$ do not depend on the trace regime ($\Tr\Sigma\geq \tau^2$ or $\Tr\Sigma<\tau^2$). We begin the proof by estimating the regime-independent terms:
\begin{itemize}
    \item [$f_2$: ]For this term we use a straightforward entry-wise estimation and \Cref{fact:infinity_norm_bound}. Thanks to \Cref{lemma:p_tauij_conc}, \Cref{lemma:conc_tau} we know the approximate values of $p_0=\Theta(p), \tau=\Theta(\sqrt{\log(1/p)/d})$. Thus, plugging in estimate for every term in $\frac{L_iL_jp_1^2d_k^2\tau^2}{p_0}$ yields :
\begin{equation}\label{eqn:f2_error_simplified_case1}
\boxed{\lVert f_2 \rVert_{op} \leq O\left(\frac{\log n}{d}\right)\frac{n\tau^2p_1^2d_k^2}{p_0} = O(np\tau^4d\log n),}
\end{equation}
on the same event from \Cref{lemma:ai_aj_li_lj_conc} probability at least $1-O(n^{-12})$.
    \item[$f_4$:] For $f_4=-d_kp_1(L\1^\top+\1L^\top + LL^\top)\odot VV^\top$, we use \Cref{fact:hadamard_product_bound} since both $VV^\top$ and $\1\1^\top$ are PSD, then we use estimates on the parameter $p_1$, as well as on the norm of $VV^\top$, estimated in \Cref{fact:spherical_matrix_norm}.

\begin{align*}
\norm{f_4}_\op&\leq d_kp_1\bigg(\norm{VV^\top\odot\big(L\1^\top+ \1L^\top \big)}_\op + \norm{VV^\top\odot LL^\top}_\op\bigg)\\
&\leq d_kp_1\bigg( \norm{\diag(L_1,\dots, L_n)VV^\top}_\op + \max_{i,j}|L_iL_j|\norm{VV^\top}_\op\bigg)\\
&\lesssim d_k p_1\bigg(\underbrace{O\big(\tfrac{n}{d}\big)}_{\ref{fact:spherical_matrix_norm}}\cdot \underbrace{O\big(\sqrt{\tfrac{\log n}{d}}\big)}_{\ref{lemma:scalar_estimates_whp}} + \underbrace{O\big(\tfrac{\log n}{d}\big)}_{\ref{lemma:scalar_estimates_whp}}\cdot \underbrace{O\big(\tfrac{n}{d}\big)}_{\ref{fact:spherical_matrix_norm}}\bigg)\\
&\lesssim d_kp_1\bigg(O\big(\tfrac{n\sqrt{\log n}}{d\sqrt{d}}\big)+O\big(\tfrac{n\log n}{d^2}\big)\bigg)
\end{align*}
Now, simplifying this expression further by taking the dominant terms, we obtain
\begin{equation}\label{eqn:f4_error_simplified_case1}
\boxed{\lVert f_4\rVert_{\op} \leq O\left(np\tau \sqrt{\frac{\log n}{d}}\right),}
\end{equation}
which occurs on the same event as \Cref{lemma:ai_aj_li_lj_conc} and \Cref{fact:spherical_matrix_norm} with probability at least $1-O(n^{-12})-O(e^{-cd})\geq 1-O(n^{-12})$ by $\omega(\log^{16}n)\leq d$.
We now move on to estimating the norms of the error terms that depend on the trace regime, $f_1, f_3$.
\item[$f_1$: ]Substituting $\Delta p_0^{ij}=-p_1d_k\Delta\tau^{ij}(1+o(1))$
(\Cref{lemma:scalar_estimates_whp}) together with the entrywise form \eqref{eqn:Delta_Tau} of
$\Delta\tau$ into $f_1^{ij}$, the linear terms $d_kp_1\langle a_i,a_j\rangle$ and
$(L_i+L_j)p_1d_k\tau$ cancel, leaving in matrix form
\[
  f_1=-p_1d_k\Big[aa^\top\odot(\1L^\top+L\1^\top)+O\big(\tfrac{\log n}{d}\big)aa^\top
       -\tau O\big(\tfrac{\log n}{d}\big)\1\1^\top\Big]+O(pd^2\tau^2)\,\Delta\tau^{\odot2}.
\] 
Since $aa^\top\odot(\1L^\top+L\1^\top)=aa^\top\diag(L)+\diag(L)aa^\top$, and using part (c)
of \Cref{lemma: DeltaTau_norm_estimates} to control $\norm{\Delta\tau^{\odot 2}}_\op$, and $\norm{aa^\top}_\op\lesssim n\mu_1^2$,
$\norm{\diag(L)}_\op=O(\sqrt{\log n/d})$,
\begin{align*}
  \norm{f_1}_\op
  &\lesssim p_1d_k\Big[\norm{aa^\top}_\op\norm{\diag(L)}_\op+\tfrac{n\tau\log n}{d}\Big]
     +O(pd^2\tau^2)\,\norm{\Delta\tau^{\odot 2}}_\op\\
  &\lesssim np\tau\sqrt d\,\mu_1^2\log^{1/2}n+np\tau^2\log n
     +npd^2\tau^2 \bigg(\mu_1^2\,\Tr\Sigma\log n + \frac{\tau^2\log^3 n}{d}\bigg),
\end{align*}
where we used $d^2\tau^2=d O(\log n)$. Simplifying and taking the dominant terms, yields
\begin{equation}\label{eqn:f1_merged}
  \boxed{\ \norm{f_1}_\op\lesssim np\big(\tau^2+d\mu_1^2\Tr\Sigma\big)\log^4 n,\ }
\end{equation}
on the same event as \Cref{lemma:scalar_estimates_whp} and \Cref{lemma: DeltaTau_norm_estimates}, which both occur with probability at least $1-O(n^{-12})$.
The two summands correspond to the small and large trace regimes,  i.e. if $\Tr\Sigma\le\tau^2$ then
$\mu_1^2\leq\Tr\Sigma\leq\tau^2$, so $d\mu_1^2\Tr\Sigma\leq \tau^2(d\tau^2)=O(\tau^2\log n)$ and the
bound is $np\tau^2\log^3 n$. On the other hand, if $\Tr\Sigma>\tau^2$ then $d\mu_1^2\Tr\Sigma$ dominates (as
$d\mu_1^2 = \omega(1)$ by \eqref{assu:gap-gen}). 

\item[$f_3:$] By $\Delta p_1^{ij}=\tau\Delta p_0^{ij}(1+o(1))=-p_1d_k\tau\Delta\tau^{ij}(1+o(1))$
(\Cref{lemma:scalar_estimates_whp}), and considering that we have added the artificial diagonal to make $P_1^{\diag}$, we have that
\begin{align*}
    P_1-p_1\1\1^\top&=-p_1d_k\tau\Delta\tau+O(d^2p\tau^3)\Delta\tau^{\odot2},\\
    f_3&=-d_k^2p_1\tau\,(\Delta\tau\odot VV^\top)+d_k\,O(d^2p\tau^3)(\Delta\tau^{\odot2}\odot VV^\top).
\end{align*}
With $d_k^2p_1\tau=\Theta(d^2p\tau^2)$ and parts (a),(b) of \Cref{lemma: DeltaTau_norm_estimates},
\begin{align*}
  \norm{f_3}_\op
  &\lesssim d^2p\tau^2\,\norm{\Delta\tau\odot VV^\top}_\op+d^3p\tau^3\,\norm{\Delta\tau^{\odot2}\odot VV^\top}_\op\\
  &\lesssim d^2p\tau^2\cdot\tfrac nd\big(\Tr\Sigma+\tau^2\big)\log n
   +d^3p\tau^3\cdot\tfrac nd\big((\Tr\Sigma\log n)^2+\tfrac{\tau^2\log n}{\sqrt d}\big).
\end{align*}
Using $d\tau^2=O(\log n)$ and \eqref{assu:trace} (to make the simplification $\sqrt d\,\Tr\Sigma=O(1)$), then taking the dominant terms yields
\begin{equation}\label{eqn:f3_merged}
  \boxed{\norm{f_3}_\op\lesssim np (\Tr\Sigma+\tau^2)\log^2 n,\ }
\end{equation}
which is $\leq np(\tau^2+d\mu_1^2\Tr\Sigma)\log^4 n$, the right-hand side of \eqref{eqn:f1_merged}, since $d\mu_1^2\ge1$. Note that in the small trace regime, $\norm{f_3}_\op\lesssim np\tau^2\log^2 n$. Moreover, the bound in \eqref{eqn:f3_merged} happens on the same event as \Cref{lemma:scalar_estimates_whp}, \Cref{lemma: DeltaTau_norm_estimates} which both occur with probability at least $1-O(n^{-12})$.

\end{itemize}
In conclusion, by adding up all the boxed errors via the triangle inequality and taking the dominant terms:
\begin{align*}
\norm{B_\sph-B_\data}_\op&=\lVert  P_0^{\diag}+d_k P_1^{\diag}\odot VV^\top -(p_0\zeta\zeta^\top +d_kp_1  YY^\top ) \rVert_{\op}+\norm{\diag(B_\sph)}_\op \\
    &\leq \lVert f_1 \rVert_{\op} + \lVert f_2 \rVert_{\op} + \lVert f_3 \rVert_{\op} + \lVert f_4 \rVert_{\op} + O(p\sqrt{d\log n})\\
    &\lesssim np\big(\tau^2+d\mu_1^2\Tr\Sigma\big)\log^4 n,
\end{align*}
which probability at least $1-O(n^{-12})$ (since \eqref{eqn:diag_op_norm} also occurs with at least probability $1-O(n^{-12})$ as with all the $f_1,\dots, f_4$) which concludes the proof.
\end{proof}

\subsection{Rank-1 difference}\label{subsec:rank-1-diff}
This section establishes the following bound for the third term in \eqref{eqn:prop_YY_diff_chain}.

\begin{lemma}
\label{lem:eigenvector_diff}
    Under the conditions of \Cref{thm:lowrank}, with probability at least $1-O(n^{-10})$,
    \begin{equation*}
        \norm{\eta_0 \mathbf{w}_0\mathbf{w}_0^\top- p_0\zeta\zeta^\top}_\op \lesssim  npd\mu_1^2\Tr(\Sigma)\log^4(n)+\max\{np\tau^2,  \sqrt{np}\}\log^9 n+dp\tau\log^4n.
    \end{equation*}
\end{lemma}
Recall that the overall goal of \Cref{prop:YY_diff} is to control $\| YY^\top -\hat{Y}\hat{Y}^\top\|_\op$. The first term $ YY^\top$ is contained in $B_\data$, while the other, is contained in the spectral decomposition of $A$ (which we called $B_\spec$). $B_\spec$ is expressed in terms of orthogonal eigenvectors of $A$, however, for the components of $B_\data= p_0\zeta\zeta^\top + d_kp_1  YY^\top$, we don't necessarily know that $\zeta$ and $ Y$ lie in orthogonal vector spaces, so we must force orthogonality.
Consequently, the strategy for this part of the proof will be to project $ YY^\top $ away from the subspace spanned by $\zeta$. Then we use Davis--Kahan  to compare $\eta_0\mathbf{w}_0\mathbf{w}_0^\top$ and $p_0\zeta\zeta^\top$. 

 \begin{lemma}
 \label{lemma:projection_of_data_matrix}
     Under the same conditions as \Cref{thm:lowrank}, with probability at least $1-O(n^{-10})$:
    \begin{align*}
     \norm{ YY^\top -\Pi^{\perp}_{\zeta}( YY^\top )\Pi^{\perp}_{\zeta}}_{op}=O(\log^4 n)
     \end{align*}
     where we denote the projection $\Pi^{\perp}_{\zeta}= I-\frac{\zeta\zeta^\top}{\lVert\zeta \rVert_2^2}$.
 \end{lemma}

 \begin{proof} We have
 \[
 YY^\top  - \Pi^{\perp}_{\zeta}( YY^\top )\Pi^{\perp}_{\zeta} = -\frac{\zeta\zeta^\top}{\lVert\zeta \rVert^2} YY^\top  -  YY^\top \frac{\zeta\zeta^\top}{\lVert\zeta \rVert^2} + \frac{\lVert\zeta^\top  Y \rVert^2}{\lVert\zeta \rVert^4}\zeta\zeta^\top.
 \]
We use concentration to estimate the norm $\lVert\zeta^\top  Y \rVert_2^2$, first when $|(\zeta^\top Y)_j|$ where $j\leq k$, we note that $z_{ij}$ is independent from $\norm{w_i}_2-1=l_i-1$ (here, it is useful to visualize the decomposition $z_i=(N_i, w_i)$ from the proof of \eqref{eqn:interlacing_result}), so the term $L_i\frac{p_1d_k\tau}{p_0} = O\Big(\sqrt{\tfrac{\log n}{d}}\Big)\cdot O(\log n) = O\Big(\tfrac{\log^{3/2}n}{\sqrt d}\Big) = o(1),$ with high probability. 
Therefore, by using Gaussian concentration, with probability at least $1-n^{-10}$ we have that for all $j \leq k$,
    \begin{align*}
    |(\zeta^\top Y)_j|&=\Bigg|\sum^n_{i=1}(1+L_i\ p_1  d_k \tau/p_0)( x_{ij}+z_{ij})\Bigg|
	\\
	&\ls (1+\log n\sqrt{d}\tau^2)\sqrt{n\log n(\tfrac{1}{d}+\Tr\left(\Sigma\right))}\\
    &\ls \sqrt{n\log n(\tfrac{1}{d}+\Tr\left(\Sigma\right))} \,.
    \end{align*}
	For $j>k$, 
    \[|(\zeta^\top  Y)_j|=\Bigg| \sum^n_{i=1}(1+L_i\ p_1  d_k \tau/p_0)z_{ij}\Bigg|\le  \Bigg| \sum^n_{i=1}z_{ij}\Bigg|+p_1d_k\tau/p_0\Bigg|\sum_{i=1}^n (l_i-1)z_{ij}\Bigg|.\]
    By Gaussian concentration, the first term is of order $O(\sqrt{n\log n/d})$. To control the second term, we first condition on the event $\cL=\bigcap_{i=1}^n\{L_i\lesssim \sqrt{\frac{\log n} d}\}$ and note that 
    $z_{ij} = l_i v_{ij}$.
    Thus, applying Hoeffding's inequality for bounded random variables (e.g. see \cite[Theorem 2.2.6]{Vershynin2018}) to $v_{ij}$, which are independent for different $i$ given a fixed $j$ and centered, we get that 
    \[\Prob\left(\left|\sum_{i=1}^n l_i(l_i-1)v_{ij}\right|\geq t\big| \cL\right)\Prob(\cL)\leq 2\exp\left(-\frac{ct^2d_k}{O(n\log (n)/d)} \right).\]
    On the right hand side, we used the bound $\Prob(\cL)\leq 1$ and the fact that on the event $\cL$, each summand satisfies $l_i^2(l_i-1)^2=l_i^2L_i^2=O(\log n/d)$ deterministically, since $l_i=1+L_i=1+O(\sqrt{\log n/d})$ on $\cL$, so summing over $n$ gives $O(n\log(n)/d)$. Undoing the conditioning on $\cL$, we then have
     \begin{align*}
&\Prob\left(\left|\sum_{i=1}^n l_i(l_i-1)v_{ij}\right|\geq t\right)\\ 
     &\quad\leq \Prob(\cL)\Prob\left(\left|\sum_{i=1}^n l_i(l_i-1)v_{ij}\right|\geq t\big| \cL \right) +\Prob\left(\cL^c \right) \Prob\left(\left|\sum_{i=1}^n l_i(l_i-1)v_{ij}\right|\geq t\big| \cL^c \right)
     \\
     &\quad \leq  \Prob\left(\left|\sum_{i=1}^n l_i(l_i-1)v_{ij}\right|\geq t\big| \cL \right) +\Prob\left(\cL^c \right)\quad  \text{ by }\Prob(\cL), \Prob\left(\sum l_i(l_i-1)v_{ij}|\cL^c\right)\leq 1\\
     &\quad\leq 2\exp\left(-\frac{ct^2d_k}{O(n\log n/d)}\right) + n^{-12} \qquad\text{by \Cref{lemma:ai_aj_li_lj_conc}}.
     \end{align*}
     Finally, setting $t=C\frac {\sqrt{n}}{d}\log n$ for sufficiently large $C$ we get that with probability at least $1-n^{-12}$
     $$\bigg \vert \sum_{i=1}^n l_i(l_i - 1)v_{ij} \bigg \vert = O\left(\frac{\sqrt{n} \log n}{d}\right).$$

    Once again, the first term in the expression for $(\zeta^\top Y)_j$ is the dominant one, so
    \[(\zeta^\top Y)_j\lesssim \sqrt{n\log n/d} \qquad j>k.\]

	Taking the norm, and a union bound over the probabilities, we obtain that with probability at least $1-dn^{-12}\geq 1-n^{-10}$,
    \begin{align*}
	\lVert\zeta^\top  Y \rVert_2^2&=kO\left({n\log n(\tfrac{1}{d}+\Tr\left(\Sigma\right))}\right) + (d-k)O\left(\frac{n\log^4 n}{d}\right)
	=O(n\log ^4n).
	\end{align*}
     Since $\Tr\left(\Sigma\right)=O(1)$, with probability at least $1-n^{-10}$
     \begin{align*}\lVert \frac{\lVert\zeta^\top  Y \rVert^2}{\lVert\zeta \rVert^4}\zeta\zeta^\top\rVert_{\op} &= \frac{\lVert\zeta^\top  Y \rVert^2}{\lVert\zeta\zeta^\top\rVert^2}
	=O(\log^4n).
	\end{align*}
Similarly, with probability at least $1-n^{-10}$
\[
\big\lVert \frac{\zeta\zeta^\top}{\norm{\zeta}^2_2} YY^\top  \big\rVert_{\op}\leq \sum_{j=1}^d\big\lVert \frac{(\zeta^\top Y)_j^2}{\norm{\zeta}^2_2} \big\rVert_{\op}=O(\log^4 n).
\]
    This concludes the proof of \Cref{lemma:projection_of_data_matrix}.
\end{proof}

\begin{proof}[Proof of \Cref{lem:eigenvector_diff}]
We would like to apply Davis--Kahan to \Cref{lemma:projection_of_data_matrix}, so we need to find the spectral gap of 
\begin{equation}
    \label{def:Bproj}
    B_{\prj} :=p_0\zeta\zeta^\top +d_k p_1\Pi^{\perp}_{\zeta}( YY^\top)\Pi^{\perp}_{\zeta}.
\end{equation}
To this end, following \Cref{YtoX}
the first eigenvalue of the matrix $B_\prj$ can be estimated with \Cref{lemma:p_tauij_conc}, which tells us that $|\zeta_i|=O(1+\sqrt{\log n/d}\tau^2 d)$, thus;
\begin{equation}\label{eqn:lambda1_B_proj}
\lambda_1(B_\prj)=\Theta(pn),
\end{equation}
while the second eigenvalue
\begin{align}
\lambda_2(B_\prj)
&=d_kp_1\lambda_1(\Pi^{\perp}_{\zeta}( YY^\top )\Pi^{\perp}_{\zeta})\nonumber\\
&\leq d_kp_1 \norm{YY^\top}_\op \norm{\Pi^{\perp}_{\zeta}}^2_\op \quad\text{by sub-multiplicativity}\\
&\leq d_kp_1O\left(n\mu_1^2+\frac{n}{d}+ 
O\left(\sqrt{n}\mu_1 \right)
\right)\quad \text{ by }\Cref{lem:eval_conc}\nonumber\\
&\leq dp\sqrt{\tfrac{\log n}{d}}O\big(\tfrac{n}{\sqrt{d}\log^2 n}\big)\nonumber\\
&= o(pn)\label{eqn:lambd2a_Bprj}.
\end{align}

Therefore, by the Davis--Kahan theorem, there exists $\theta\in \{-1,1\}$ such that
\begin{equation}\label{eqn:DavisKahanEigenVectors}
    \norm{\theta \mathbf{w}_0-\frac{\zeta}{\norm{\zeta}_2}}_2
  \lesssim \frac{\norm{A-B_\prj}_\op}{\lambda_1(B_\prj)-\lambda_2(B_\prj)}
\end{equation}
From the triangle inequality and Lemmas \ref{lemma:projection_of_data_matrix}, \ref{lemma:trace_method} and \ref{lemma:Bsph_Bdata_diff},
\begin{align*}
\norm{A-B_\prj}_\op&= \norm{A-p_0\zeta\zeta^\top -d_k p_1\Pi^{\perp}_{\zeta}( YY^\top)\Pi^{\perp}_{\zeta}}_\op\\
&\leq d_kp_1\norm{ YY^\top-\Pi^{\perp}_{\zeta}( YY^\top)\Pi^{\perp}_{\zeta}}_\op+ \norm{A-B_\data}_\op\\
&\lesssim dp\tau\log^4(n) + npd\mu_1^2\Tr(\Sigma)\log^2(n)+\max\{np\tau^2,  \sqrt{np}\}\log^9 n\\
&\lesssim npd\mu_1^2\Tr(\Sigma)\log^2(n)+\max\{np\tau^2,  \sqrt{np}\}\log^9 n +dp\tau\log^4 n.
\end{align*}
The last inequality was obtained by taking dominant terms.
Now assume without loss of generality that $\theta=1$ in \eqref{eqn:DavisKahanEigenVectors}. Thus,
\begin{equation}
\label{eqn:eigenvector_diff}
\begin{aligned}
\norm{\eta_0\mathbf{w}_0\mathbf{w}_0^\top-p_0\zeta\zeta^\top}_\op&\leq \vert\eta_0-p_0\norm{\zeta}_2^2\vert+  p_0\norm{\zeta}_2^2\norm{\mathbf{w}_0-\frac{\zeta}{\norm{\zeta}}_2}_2 \nonumber\\
    &\leq \underbrace{\norm{A-B_\prj}_\op}_{\text{by Weyl's inequality}} + \underbrace{\Theta(pn)}_{
    =\lambda_1(B_\prj)}\cdot\underbrace{\frac{\norm{A-B_\prj}}{\Theta(pn) - o(pn)}}_{\text{by }\ref{eqn:DavisKahanEigenVectors} \text{ and }\ref{eqn:lambd2a_Bprj}} \\
    &\leq \Theta(1)\left(npd\mu_1^2\Tr(\Sigma)\log^4(n)+\max\{np\tau^2,  \sqrt{np}\}\log^9 n\right)
    + dp\tau \log^4 n\end{aligned}
\end{equation}
which completes the proof.
\end{proof}

\subsection{Proof of \texorpdfstring{\Cref*{prop:YY_diff}}{Proposition 4.2}}
Inserting the bounds from Lemmas \ref{lemma:trace_method}, \ref{lemma:Bsph_Bdata_diff} and \ref{lem:eigenvector_diff} into
\eqref{eqn:prop_YY_diff_chain}, we have 
\begin{equation}\label{eqn:scaled_YY_diff}
\begin{aligned}
    &\norm{d_k p_1  YY^\top -d_k p_1\hat Y \hat Y ^\top}_\op\\
    &\leq (\text{Adjacency to Spherical}) +(\text{Spherical to Data}) +(\text{Rank-1 difference})\\
    &\lesssim npd(\tau^2+\mu_1^2\Tr(\Sigma))\log^4(n)+\max\{np\tau^2,  \sqrt{np}\}\log^9 n +dp\tau\log^4n.
\end{aligned}
\end{equation}
From \Cref{lemma:p_tauij_conc} we have $p_1=\Theta(p\tau)$, so we may divide through by $d_kp_1\asymp dp\tau$ to obtain
\[
\norm{ YY^\top -\hat{Y}\hat{Y}^\top}_\op \lesssim n\sqrt{d}\mu_1^2\Tr(\Sigma)\log^4(n)+\left(n\frac{\tau}{d}+  \sqrt{\frac{n}{dp}}\right)\log^9 n + \log^4n.
\]
Noticing that $\sqrt{\frac{n}{dp}}\log^9 (n)\geq \log^4n$ and $np\tau^2\log^9 n\geq np\tau^2\log^4n$, gives the statement of the Proposition.\qed

\section{Generalization to Numerically Low-Rank Data}\label{sec: numerically_low_rk}

In this section we prove \Cref{thm:main}, which relaxes the strict rank assumption of \Cref{thm:lowrank}.
The general approach is to pass to a low-rank projection of the data matrix $X$. Hence, the statements and proofs of many lemmas and propositions mirror those of \Cref{thm:lowrank}, but additional arguments are needed to handle errors arising from the spectral projection. 

Recall our notation $\Pi^{(k)}_M$ for the matrix for projection to the span of the top $k$ eigenvectors of a symmetric matrix $M$.
Write $X_k:= \Pi^{(k)}_{X}X$, and set
\[
  \tilde Y := X_k + Z,
  \qquad
  \tilde A_{ij} := \1\!\left(\langle \tilde y_i,\tilde y_j\rangle \geq \tau\right),
\]
so that $\tilde Y$ is the noisy version of the spectrally truncated point cloud $X_k$ and
$\tilde A$ is the adjacency matrix of the exactly-rank-$k$ model on $\tilde Y$, whose rows are $\tilde y_i= \tilde x_i + z_i$.
Set $X_{k+1:d}:= X-X_k$, and, recalling $\Sigma=\frac1nX^\top X$, set 
\begin{equation}
    \label{def:Sigmak}
    \Sigma_k := \frac1nX_k^\top X_k\,,\qquad \Sigma_{k+1:d}:= \Sigma-\Sigma_k = \frac1n X_{k+1:d}^\top X_{k+1:d}.
\end{equation}
The spectral embedding $\Xalg$ is, as before, built from the top-$k$
eigenpairs of the \textit{observed} matrix $A$. Note that \Cref{lem:p1} was already stated under the assumptions of  \Cref{thm:main}, so we can carry on the rest of the proof with $\hat X$ in place of $\Xalg$.

The proof proceeds in three main steps, summarized as follows.
\begin{itemize}
    \item[(i)] \textbf{Parameter truncation:} Under \eqref{assu:gap-gen} the truncation does
not change the leading order model parameters: $\Tr\Sigma = (1+o(1))\Tr\Sigma_k$,
the threshold $\tau$ and first Gegenbauer coefficient $p_1$ are unchanged to
leading order, and $p=(1+o(1))\tilde p$ by \Cref{lem:ptp-same}.
Consequently every parameter estimate proved for the truncated model
(\Cref{Appendix Section:parameter_estimation_proofs}) transfers verbatim, up to
$(1+o(1))$ factors, to the observed model. In fact, the parameter estimate proofs in the appendix are done for this generalized case.
\item[(ii)]\textbf{One further projection:}
We pick up an additional error term in passing from the full-rank projected data $Y_kY_k^\top$ to the truncated projected data
$\tilde Y_k\tilde Y_k^\top$, defined as $\tilde Y_k= \Pi_{\tilde Y}^{(k)} \tilde Y$.
This step is spelled out  in \eqref{eqn:overview_triangle_ineq_split_gen}.
\item[(iii)]\textbf{One new comparison:} Because $\hat X$ is built from the observed adjacency matrix $A$, while the
rank-$k$ truncated data describes $\tilde A$, the comparison step of
\Cref{thm:lowrank} picks up an additional summand
$\|A-\tilde A\|_\op$. This new term is the source of the truncation term in
\eqref{bd:main}.
\end{itemize}
We bound
\begin{equation}\label{eqn:overview_triangle_ineq_split_gen}
\|XX^\top-
\hat X \hat X^\top\|_\op\leq\norm{XX^\top-Y_kY_k^\top}_\op
   +\norm{Y_kY_k^\top-\tilde Y_k \tilde Y_k^\top}_\op
   +\|\tilde Y_k \tilde Y_k^\top-\hat X\hat X^\top\|_\op\,.
\end{equation}
The second term is new and is controlled by concentration arguments
-- see \Cref{lemma: Y_k matrices difference_gen}. To bound the third term we now use the truncated adjacency matrix $\tilde A$, and the data matrix is redefined as $\tilde B_\data= p_0\zeta\zeta^\top + p_1d_k \tilde Y\tilde Y^\top$. Then to pass to the \textit{observed} adjacency matrix $A$ we need to estimate $\norm{A-\tilde A}_\op$. 
For additional clarity, \Cref{fig:proof comparison tree} is a diagram that illustrates the new steps, incorporated within the proof of the strictly rank--$k$ case.

\begin{figure}[t]
  \centering
  \resizebox{\textwidth}{!}{%
  \begin{tikzpicture}
  \node[neutral] (goal0)  at (-7,0)     {\textbf{Goal (\Cref{thm:main})}\\[1pt]$\|XX^\top-\Xalg\Xalg^\top\|_\op$};
  \node[neutral] (goal)  at (0,0)     {\textbf{ 
  \Cref{prop:XhatX}}\\[1pt]$\|XX^\top-\hat X\hat X^\top\|_\op$};
  \node[strict]  (step0) at (-10,-2.0)  {\textbf{Reduction to $\hat X$} (\Cref{lem:p1})\\ $\|\Xalg\Xalg^\top-\hat X\hat X^\top\|_\op$};
  \node[strict]  (step1) at (-3,-2.0)  {
  Lemma \ref{lem:projection-low}/\ref{lem:projection}\\
  $\|XX^\top- Y_kY_k^\top \|_\op$};
  \node[strict]  (step2) at (0,-3.9)  {
  Lemma \ref{lem:perturbation-low}/\ref{lem:perturbation}
  \\[1pt]Davis--Kahan};
  \node[strict]  (step3) at (-7,-6.9)  {
  Proposition \ref{prop:YY_diff}/\ref{prop:YY_diff_gen}\\$ \|\tilde Y\tilde Y^\top-\hat Y\hat Y^\top\|_\op$};
  \draw[arr] (goal) --(goal0);
    \draw[arr] (step0) -- (goal0);
  \draw[arr]  (step1)-- (goal);
\draw[arr] (step3) -- (goal0);
  \draw[arr] (step2) -- (goal);

  \node[new] (termii) at (7,0)
     {\textbf{Intermediate Low-rank Projection:}\\
     (\Cref{lemma: Y_k matrices difference_gen})\\[1pt]$\|Y_kY_k^\top-\tilde Y_k\tilde Y_k^\top\|_\op$}; 
  \coordinate (tIIb) at (,-2.95);
  \draw[arrnew] (termii) to (goal);

  \node[strict] (adjsph) at (5,-5.6)  {\textbf{Adjacency to Spherical} (\Cref{lemma:trace_method})\\[1pt]$\|A-B_\sph\|_\op$};
  \node[strict] (sphdat) at (5,-8.2)  {\textbf{Spherical to data} (\Cref{lemma:Bsph_Bdata_diff})\\[1pt]$\|B_\sph-B_\data\|_\op$};
  \node[strict] (rank1)  at (5,-10.6) {\textbf{Rank-1 difference} (\Cref{lem:eigenvector_diff})\\[1pt]$\|\eta_0\mathbf{w}_0\mathbf{w}_0^\top-p_0\zeta\zeta^\top\|_\op$};
  \draw[arr]    (adjsph) to (step3);
  \draw[arr]    (sphdat) to (step3);
  \draw[arr]    (rank1) to (step3);
  \draw[arr] (adjsph.east) to[bend left=45] (rank1);
  \draw[arr]    (sphdat) to  (rank1);

  \node[new] (pTrunc) at (13.2,-5.2)
     {\textbf{Truncation error} (\Cref{prop: Adjacency Matrix Truncation})\\[1pt]$\|A-\tilde A\|_{\op}$};
  
  \node[new] (pEq) at (13.2,-7.2)
     {\textbf{Average probability equivalence} (\Cref{lem:ptp-same})\\[1pt]$p=(1+o(1))\tilde p$};
  
  \node[param] (pTau) at (13.2,-9.0)
     {\textbf{Threshold Estimation} \Cref{lemma:conc_tau}\\[1pt]$\tau\asymp \sqrt{\log(1/p)/d}$};
  
  \node[param] (pRest) at (13.2,-11.0)
     {\textbf{Coefficient estimates} (\Cref{lemma:ai_aj_li_lj_conc}, \ref{lemma:p_tauij_conc}, \ref{lemma:scalar_estimates_whp})\\[1pt]
      $\langle a_i,a_j\rangle,\,p_0,\,p_1, \Delta\tau^{ij},\dots$};

  \draw[arrnew] (pTrunc) to [bend right=15](rank1);
  \draw[arrnew] (pTrunc) to (adjsph);

  \draw[arrnew]  (pTrunc) -- (pEq);
  \draw[arrparam] (pEq)    -- (pTau);
  \draw[arrparam] (pTau)   -- (pRest);

  \draw[arrparam] (pRest) to[bend left=10] (sphdat);
  \draw[arrparam] (pTau) to[bend left=20] (sphdat);

  \node[strict, text width=0.25cm, minimum width=0.5cm, minimum height=0.32cm, inner sep=1pt] 
     at (13,-0.55) {};
  \node[anchor=west, font=\footnotesize] at (13.45,-0.55) {strict rank-$k$};
  
  \node[new, text width=0.25cm, minimum width=0.5cm, minimum height=0.32cm, inner sep=1pt] 
     at (13,-1.45) {};
  \node[anchor=west, font=\footnotesize] at (13.45,-1.45) {numerically low-rank (new)};
  
  \node[param, text width=0.25cm, minimum width=0.5cm, minimum height=0.32cm, inner sep=1pt] 
     at (13,-2.35) {};
  \node[anchor=west, font=\footnotesize] at (13.45,-2.35) {parameter estimation };
  
  \end{tikzpicture}}

  \caption{Structure of the proof of \Cref{thm:main}. 
  Steps for
  the argument for \Cref{thm:lowrank} (strict rank-$k$ case) are shaded blue.
  For numerically low-rank data, two new components are added in red: the additional
  triangle inequality term passing from the full-rank projected data to the truncated
  projected data, and, inside the proof of \Cref{prop:YY_diff_gen}, a step estimating
  $\|\tilde A-A\|_\op$. 
}
  \label{fig:proof comparison tree}
\end{figure}
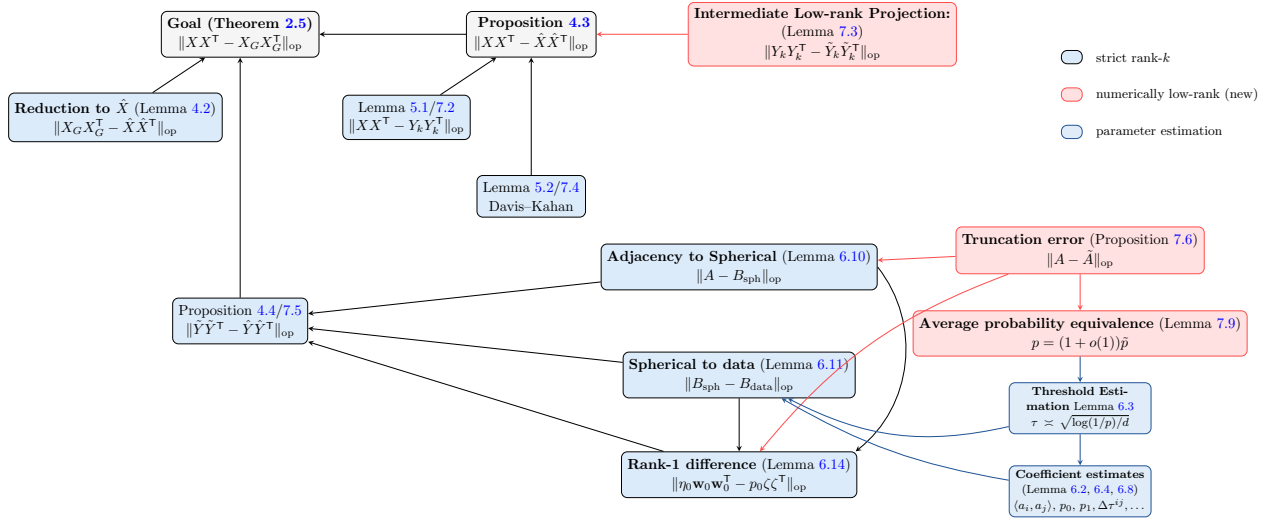

\subsection{Projection and Perturbation Step Analogues}

We begin by stating an extension of \Cref{lem:eval_conc}.

\begin{lemma}\label{eqn:kplus1_eval_UB} 
With probability at least $1-e^{-\sqrt{d}}$, 
$\lambda_{k+1}(YY^{\top}) \leq  \frac nd
 + O\Big(\sqrt{\frac nd}+\frac{n\mu_{k+1}}{\sqrt d}+n\mu_{k+1}^2\Big)$.
 The estimates of $\lambda_{i}(YY^{\top}), 1\leq i\leq k$ agree with \Cref{lem:eval_conc}.
\end{lemma}

\begin{proof}
The proof follows the Cauchy Interlacing part of the proof of \eqref{YtoX}, specifically the decomposition \eqref{eqn: data decomp}, but adding the ``leftover'' coordinates to the lower-right block of the matrix \ref{eqn:block_matrix}.
\end{proof}

\begin{lemma}[Clean Data and Noisy Data Projection Matrix Distance Generalized]
\label{lem:projection}
Under the same conditions as  \Cref{thm:main}, with probability $1-O(e^{-d})$,
\[
\norm{XX^\top-Y_kY_k^\top}_\op\lesssim n\mu_1\sqrt{\frac{1}{d}}.
\]
\end{lemma}

\begin{proof}
The proof is analogous to \Cref{lem:projection-low}, except for the first triangle inequality which separates the trailing eigenvalues of $XX^\top$ from the first $k$. 
\begin{equation}\label{eqn:xx_ykyk_diff_step1}
\begin{aligned}
   \norm{XX^\top-Y_kY_k^\top}_\op&\leq \norm{X_{k+1:d}X_{k+1:d}^\top}_\op+ \norm{X_kX_k^\top - X_kY_k^\top}_\op+ \norm{X_kY_k^\top- Y_kY_k^\top}_\op\\
   &\leq \norm{X_{k+1:d}X_{k+1:d}^\top}_\op + (\norm{X_k}_\op+ \norm{Y_k}_\op)\norm{X_k-Y_k}_\op\,.
\end{aligned}
\end{equation}
The first term on the right hand side of \eqref{eqn:xx_ykyk_diff_step1} is $n\mu_{k+1}^2$. For the second summand on the right hand side, we have that with probability at least $1-e^{-d}$,
\begin{equation}\label{eqn:xx_ykyk_diff_step2}
\norm{X_k - Y_k}_\op \leq 2 \norm{X_k - Y}_\op = 2 \norm{X_k - (X+Z)}_\op \leq 2( \sqrt{n}\mu_{k+1} + \norm{Z}_\op)\lesssim \sqrt{n}\mu_{k+1}+\sqrt{\frac nd}.
\end{equation}
Then, combining \eqref{eqn:xx_ykyk_diff_step1} and \eqref{eqn:xx_ykyk_diff_step2} with bounds from the proof of \Cref{lem:projection-low} yields
\begin{equation}
    \norm{XX^\top- Y_kY_k^\top}_\op\lesssim n\mu_{k+1}^2 +n\mu_1(\mu_{k+1}+ 1/\sqrt{d}).
\end{equation}
From the conditions \eqref{assu:gap-gen}
it follows $\mu_{k+1}\le d^{-1/2} \le \mu_k\le \mu_1$, so the right hand side is $O(n\mu_1/\sqrt{d})$ as claimed.
\end{proof}

\begin{lemma}[Original and Truncated Data Difference] \label{lemma: Y_k matrices difference_gen}
    Under the same conditions as \Cref{thm:main}, we have with probability at least $1-e^{-\sqrt{d}}$,
    \[\norm{Y_kY_k^\top-\tilde Y_k\tilde Y_k^\top}_\op\lesssim \frac{n}{\sqrt{d}}\frac{\mu_1^2}{\mu_k^2}\mu_{k+1}
    .\]
\end{lemma}

\begin{proof}
    Recall that $\tilde Y_k= \Pi_{\tilde Y}^{(k)}(X_k+Z), Y_k = \Pi_{Y}^{(k)}(X+Z)$, for short, we will use $\Pi= \Pi^{(k)}_{Y}$ as in the proof of \Cref{lem:perturbation-low} and $\tilde \Pi= \Pi^{(k)}_{\tilde Y}$
    Using the triangle inequality, we introduce an intermediate term $\Pi^{(k)}_{ Y}\tilde Y=\Pi \tilde Y$:
    \begin{equation}\label{eqn:triangle_ineq_truncated}
        \norm{Y_kY_k^\top-\tilde Y_k\tilde Y_k^\top}_\op\leq \norm{Y_kY_k^\top - \Pi\tilde Y (\Pi\tilde Y)^\top}_\op + \norm{\Pi\tilde Y (\Pi\tilde Y)^\top - \tilde Y\tilde Y^\top}_\op
    \end{equation}
    Then, estimating the first term on the right hand side via concentration of the operator norm of $Z$ (e.g. \cite[Theorem 4.4.5]{Vershynin2018}) yields, that with probability $1-O(e^{-d})$, 
    \begin{equation}\label{eqn:triangle_term_trunc1}
    \begin{aligned}
        \norm{\Pi\tilde Y (\Pi \tilde Y)^\top- Y_kY_k^\top}_\op&=\left\| \Pi\left[X_kX_k^\top- XX^\top+(X_k-X)Z^\top+Z(X_k-X)^\top\right] \Pi\right\|_\op\\
        &\le \| X_kX_k^\top- XX^\top+(X_k-X)Z^\top+Z(X_k-X)^\top\|_\op\\
        &\leq 
        \norm{X_kX_k^\top-XX^\top}_\op + 2 \norm{X-X_{k}}_\op 
        \norm{ Z}_\op\\
        & = n\mu_{k+1}^2  + 2\sqrt{n}\mu_{k+1} \|Z\|_\op \\
        &
        \ls n\mu_{k+1}^2+ n\mu_{k+1}\sqrt{\frac{1}{d}} \lesssim \frac{n}{\sqrt{d}}\mu_{k+1}
    \end{aligned}
    \end{equation}
where we used \eqref{ZZnorm} in the final line.
    As for the second expression in \eqref{eqn:triangle_ineq_truncated}, we write
    \begin{equation}\label{eqn:triangle_term_trunc2}
    \begin{aligned}
        \norm{\Pi\tilde Y (\Pi \tilde Y)^\top - \tilde Y\tilde Y^\top}_\op&= \norm{(\Pi-\tilde \Pi)\tilde Y\tilde Y^\top\tilde \Pi+ \Pi\tilde Y\tilde Y^\top(\Pi-\tilde \Pi)}_\op\\
        &\leq 2\norm{\Pi-\tilde \Pi}_\op\norm{\tilde Y\tilde Y}_\op\\
        &\lesssim \norm{\tilde Y\tilde Y^\top}_\op\frac{\norm{YY^\top-\tilde Y\tilde Y^\top}_\op}{\lambda_{k}(YY^\top)-\lambda_{k+1}(YY^\top)}\quad \text{ by \Cref{DavisKahanForProjections}}.
    \end{aligned}       
    \end{equation}
By definition of $\tilde Y$, we know that $\lambda_1(X_kX_k^\top)=n\mu_1^2=\lambda_1(XX^\top)$, so by the same argument as \Cref{lem:eval_conc}, $\norm{\tilde Y\tilde Y^\top}_\op=n\mu_1^2 +O(n/d)\lesssim n\mu_1^2,$  
with probability $1-O(e^{-d})$, and similarly combining \Cref{lem:eval_conc} and \Cref{eqn:kplus1_eval_UB} gives us an estimate on the eigenvalue gap 
\begin{equation}
\begin{aligned}
\lambda_{k}(YY^\top)-\lambda_{k+1}(YY^\top)&\geq  n\mu_k^2 +\frac{n}{d}-O\left(\sqrt{n}\mu_1d^{-1/4}\right)-\bigg(\frac nd
 + O\Big(\sqrt{\frac nd}+\frac{n\mu_{k+1}}{\sqrt d}+n\mu_{k+1}^2\Big)\bigg)\\
 &\gtrsim n\mu_k^2.
 \end{aligned}
 \end{equation}
 In the last line, we use \eqref{assu:gap-gen}, \eqref{assu:trace} and $d<n$ to obtain the dominant term.
 Then, by 
the same lines as in \eqref{eqn:triangle_term_trunc1}, we have that with probability $1-O(e^{-d})$:
 \begin{equation}   \label{YY-tYtY}
 \begin{aligned}
 \norm{YY^\top-\tilde Y\tilde Y^\top}_\op& 
\lesssim n\mu_{k+1}/\sqrt{d}\,.
 \end{aligned}
 \end{equation}
 Plugging these estimates into \eqref{eqn:triangle_term_trunc2}, then adding to \eqref{eqn:triangle_term_trunc1} gives us:
 \[\norm{Y_kY_k^\top -\tilde Y_k\tilde Y_k^\top}_\op\lesssim \frac{n}{\sqrt d}\mu_{k+1}+ n\frac{\mu_1^2}{\mu_k^2}(\mu_{k+1}^2+ \mu_{k+1}\frac{1}{\sqrt d})
 \lesssim \frac{\mu_1^2}{\mu_k^2} \frac{n\mu_{k+1}}{\sqrt{d}},\]
as claimed.
\end{proof}

\begin{lemma}
[Embedding Matrix and Noisy Data Projection Matrix Difference]
\label{lem:perturbation} 
Under the same conditions as \Cref{thm:main}, we have with probability $1-O(n^{-10})$,
\[
    \norm{\tilde Y_k\tilde Y^\top_k- \hat X\hat X^\top}_\op \lesssim \frac{\mu_1^2}{\mu_k^2} \big\|\tilde Y \tilde Y^\top-\hat{Y}\hat{Y}^\top\big\|_\op.
\]
\end{lemma}

\begin{proof}[Proof of \Cref{lem:perturbation}]
The argument coincides with that of \Cref{lem:perturbation-low}
up to the final substitution, so we only indicate the changes. Let us use the following shorthand notation, to avoid clutter:
$\tilde Y = X_k + Z$ is the truncated noisy data, and we abbreviate
$\tilde \Pi = \Pi^{(k)}_{\tilde Y}$ and $\hat\Pi = \Pi^{(k)}_{\hat Y}$ for the projections
onto the top $k$ eigenvectors of $\tilde Y\tilde Y^\top$ and
$\hat Y\hat Y^\top$ respectively.
 
The projection decomposition \eqref{eqn:proj_decomp_YY_hatYY} 
and the
term-by-term estimates $(i)$--$(iii)$ hold verbatim with $\tilde Y, \tilde \Pi$ in place of
$Y, \Pi$.
Terms $(i)$--$(ii)$ are controlled through
Davis--Kahan (\Cref{DavisKahanForProjections}) and $(iii)$ directly, yielding 
\[
\norm{\tilde Y_k\tilde Y_k^\top - \hat X\hat X^\top}_\op
  \lesssim
  \norm{\tilde Y\tilde Y^\top - \hat Y\hat Y^\top}_\op
  \left(\frac{\norm{\tilde Y\tilde Y^\top}_\op}
             {\lambda_k(\tilde Y\tilde Y^\top)-\lambda_{k+1}(\tilde Y\tilde Y^\top)}
        +1\right).
\]
$\tilde Y=X_k+Z$ has rank-$k$ signal, thus the eigenvalue estimates of \Cref{YtoX} apply
verbatim to $\tilde Y$, the spectral gap is then $\lambda_k(\tilde Y\tilde Y^\top)-\lambda_{k+1}(\tilde Y\tilde Y^\top)\gtrsim n\mu_k^2$. Substituting then gives \Cref{lem:perturbation}.
\end{proof}

We have the following analogue of \Cref{prop:YY_diff}:

\begin{prop}
\label{prop:YY_diff_gen}
Under the conditions of \Cref{thm:main},
with probability at least $1-O(n^{-10}$),
\[
    \big\|\tilde Y \tilde Y^\top-\hat{Y}\hat{Y}^\top\big\|_\op\lesssim  n\sqrt{d}\mu_1^2\Tr(\Sigma)\log^4(n)+\left(n\frac{\sqrt{\log(1/p)}}{d^{3/2}}+  \sqrt{\frac{n}{dp}}\right)\log^9 n+ \frac{n}{dp\tau}\left(\Tr\Sigma_{k+1:d}\right)^{1/4}.
\]
\end{prop}
We defer the proof of this Proposition to \Cref{sec:compare-gen}.

Now that we have all the components from \eqref{eqn:overview_triangle_ineq_split_gen}, we can complete the proof of the theorem.
\begin{proof}[Proof of \Cref{thm:main}]
The passage from $\Xalg$ to $\hat X$ follows from the same argument as in the proof of \Cref{thm:lowrank} using \Cref{lem:p1}.
    The proof is concluded by combining  Lemmas \ref{lem:projection}, \ref{lemma: Y_k matrices difference_gen},  \ref{lem:perturbation} and \Cref{prop:YY_diff_gen} with \eqref{eqn:overview_triangle_ineq_split_gen}.
\end{proof}

\subsection{Comparison Step} 
\label{sec:compare-gen}

Recall that in the proof of \Cref{thm:main}, we used \Cref{prop:YY_diff_gen}, a quantitative control on $\|\tilde Y\tilde Y^\top-\hat Y\hat Y^\top \|_\op$. In this subsection we explain how \Cref{prop:YY_diff_gen} follows from combining the arguments of \Cref{sec:YY_diff} (replacing occurrences of $Y$ with $\tilde Y$) with  a good enough quantitative bound on $\|A-\tilde A\|_\op$ (see \Cref{prop: Adjacency Matrix Truncation}). 

\subsubsection{Truncating the Adjacency Matrix}\label{subsec:truncating_adj_matrix}
In this section, we will show that the observed adjacency matrix $A_{ij}=\1\left(y_i, y_j\rangle\geq \tau\right)$ is well-approximated by adjacency matrix $\tilde A=\1\left(\langle \tilde y_i, \tilde y_j\rangle\geq \tau\right)$, induced by the strictly low-rank data. This is the first component that we need to find to prove \Cref{prop:YY_diff_gen}.

\begin{prop}[Truncation Error]
\label{prop: Adjacency Matrix Truncation}
Under the same assumptions as \Cref{thm:main}, and $\Tr\Sigma_{k+1:d}>0$ (otherwise we reduce to the strictly rank$-k$ case, as we have already done in \Cref{thm:lowrank}). Then, with probability $1- O(n^{-10})$:

\begin{equation}\label{eqn: adjacency truncation bound}
\|A-\tilde A\|_\op \lesssim n\left(\Tr\Sigma_{k+1:d}\right)^{1/4}+ \log n.
\end{equation}

\end{prop}

To prove this proposition, we first define a continuous approximation to the indicator function. For $\delta>0$, let:
\[
I_\delta(x)= \begin{cases} 0 \quad &\text{ if } x<\tau-\delta,\\
\frac{1}{\delta} (x - \tau + \delta) \quad &\text{ if } \tau-\delta \leq x < \tau,\\
1 \quad &\text{ if } x\geq \tau.
\end{cases}
\]
Note that $I_\delta$ is $(1/\delta)$-Lipschitz and bounded between 0 and 1. Let $A_\delta, \tilde A_\delta$ be matrices with entries $I_\delta(\langle y_i,y_j\rangle)$, $I_\delta(\langle \tilde y_i,\tilde y_j\rangle)$, respectively,
where we recall that $\tilde y_i = \Pi^{(k)}_{X}x_i +z_i$. We split the error as follows:
\begin{equation}\label{AtA-split}
\|A-\tilde A\|_\op\leq \|A_\delta-\tilde A_\delta\|_\op+\|A-A_\delta+\tilde A_\delta-\tilde A\|_\op.
\end{equation}

\begin{lemma}[Lipschitz Approximation Error]
\label{lem:ALip}
Under the same conditions as \Cref{thm:main}, with probability at least $1-O(n^{-10})$:
\begin{equation}
    \|A_\delta -\tilde A_\delta\|_\op \lesssim \frac{n}{\delta \sqrt d} (\Tr \Sigma_{k+1:d})^{1/2} 
    \qquad 
    \forall \delta>0. 
\end{equation}
\end{lemma}
\begin{proof}
We will use the Frobenius norm to control the operator norm since $\|A_\delta -\tilde A_\delta\|_\op\leq \|A_\delta -\tilde A_\delta\|_F$. Then, explicitly writing out the component-wise difference, we obtain:
\begin{equation}\label{eqn:Mij}
\begin{aligned}
|I_\delta(\langle y_i,y_j\rangle)- I_\delta (\langle \tilde y_i, \tilde y_j\rangle)| &\leq \frac{1}{ \delta}|\langle y_i,y_j\rangle-\langle \tilde y_i,  \tilde y_j\rangle| =:\frac1{\delta} M_{ij}, 
\end{aligned}
\end{equation}
and thus $\|A_\delta-\tilde A_\delta\|_F \leq \frac1{\delta}\|M\|_F$. We can express the matrix $M$ as a function of the noise vectors $M=M(Z)=\frac{1}{ \delta}\left(X_{k+1:d} X_{k+1:d}^\top+ X_{k+1:d}Z^\top + ZX_{k+1:d}^\top\right)$, where $Z$ contains the noise vectors as rows and $X_{k+1:d} = X-X_k=(I-\Pi^{(k)}_{X})X$ is the projection onto the tail eigenspace.

Note that for different inputs $W,W'\in \R^{n\times d}$, the reverse triangle inequality and the non-commutative H\"older inequality yield:
\begin{align*}
    \big|\|M(W)\|_F - \|M(W')\|_F\big| &\leq \|M(W)-M(W')\|_F\\
    &=  \|X_{k+1: d}(W-W')^\top + (W-W')X_{k+1:d}^\top\|_F\\
    &\leq
    2\sqrt{n}\mu_{k+1}
    \|W-W'\|_F.
\end{align*}
Therefore $\|M(Z)\|_F$ is a $2\sqrt{n}\mu_{k+1}$-Lipschitz function (with respect to the Euclidean Frobenius metric) of Gaussian noise with variance $\frac{1}{d}$.
From Gaussian concentration (see \cite[Theorem 5.2.3 ]{Vershynin2018}) it follows
\begin{equation*}
  \Prob(\|M\|_F - \mathbb{E}\|M\|_F\geq t ) \leq \exp\left(-\frac{c t^2 d}{n\mu^2_{k+1}}\right)
\end{equation*}
for any $t>0$.
Moreover, it is straightforward to control the expectation:
\begin{align}
\mathbb{E}\|M\|_F &\leq 
\| X_{k+1:d}X_{k+1:d}^\top\|_F + 2 \E\|X_{k+1:d} Z^\top \|_F  
\notag\\
&\leq 
\bigg(\sqrt{n}\mu_{k+1} + 2\sqrt{\frac{n}{d}}\bigg) \|X_{k+1:d}\|_F \notag\\
&\lesssim \frac{n}
{\sqrt{d}}
\sqrt{
\sum_{l > k}\mu_l^2}.\label{eqn:expectation_lipschitz_part}
\end{align}
The lemma follows by setting $t=C\sqrt{n/d}\mu_{k+1}\log n$ for a sufficiently large constant $C$, 
and noting that $\mu_{k+1}\log n\le \sqrt{n\Tr\Sigma_{k+1:d}}$.
\end{proof}

Now, we want to control the other part of the triangle inequality in \eqref{AtA-split}, which captures the discontinuity of the hard indicator. Let $g(x)= \1(x\geq \tau)-I_\delta(x)$, $Q_{ij}=\langle y_i,y_j\rangle$, and $\tilde Q_{ij}=\langle \tilde y_i,\tilde y_j\rangle$. 
Let $H$ be the symmetric matrix with entries $(A-A_\delta+\tilde A_\delta-\tilde A)_{ij} = g(Q_{ij})- g(\tilde Q_{ij})$.
By the Schur test,
\begin{equation}    \label{bd:Schur}
    \|H\|_\op \le \max_{1\le i\le n} \sum_{j=1}^n |H_{ij}|
\end{equation}
Note that $0\le |g(x)|\le \1(x\in [\tau-\delta,\tau])$. 
Thus, with $D_{ij}:= \1(Q_{ij}\in [\tau-\delta,\tau])$ and $\t D_{ij}:= \1(\t Q_{ij}\in [\tau-\delta,\tau])$, we have
\begin{equation}
    \label{bd:Hij}
    |H_{ij}|\le 2 (D_{ij} + \t D_{ij})\qquad \forall i, j\in[n].
\end{equation}

\begin{claim}   \label{claim:Dij}
    For any $\delta>0$, with probability $1-O(e^{-cd}) + O(ne^{-c\delta \sqrt{d}n})$,
    \begin{equation}    \label{bd:Dij}
        \max_{1\le i\le n} \sum_{j=1}^n D_{ij} \ls \delta\sqrt{d}n\,.
    \end{equation}
    The same holds with $\tilde D$ in place of $D$.
\end{claim}

We first conclude the proof of \Cref{prop: Adjacency Matrix Truncation} on the claim. 

\begin{proof}[Proof of \Cref{prop: Adjacency Matrix Truncation}]
Let $\delta>0$ to be chosen.
From \eqref{AtA-split}, \Cref{lem:ALip}, \eqref{bd:Hij} and \Cref{claim:Dij}, 
\[
\|A-\t A\|_\op \ls \frac n{\delta\sqrt{d}} (\Tr \Sigma_{k+1:d} )^{1/2}+ \delta\sqrt{d}n
\]
with probability $1-O(n^{-10} + e^{-cd}+ e^{-c\delta\sqrt{d}n})$. Taking $\delta  = d^{-1/2}(\Tr \Sigma_{k+1:d})^{1/4}+ C\frac{\log n}{\sqrt{d}n}$ for a sufficiently large constant $C$ concludes the proof. 
\end{proof}

\begin{proof}[Proof of \Cref{claim:Dij}]
We only give the argument for $D$; the argument for $\tilde D$ being almost identical. 

Fixing an arbitrary $i\in[n]$, note that conditional on $z_i$, the $n-1$ variables $(Q_{ij})_{j\ne i}$ are independent and Gaussian, with $Q_{ij}$ having mean $\langle y_i,x_j\rangle$ and variance $\|y_i\|_2^2/d$. 
In particular, the conditional density of $Q_{ij}$ given $z_i$ is bounded by $O(\sqrt{d}/\|y_i\|_2)$. Letting $\cE_i=\{ \|y_i\|_2\ge \frac14\}$, we have that on $\cE_i$ the variables $(D_{ij})_{j\ne i}$ are independent Bernoulli variables with expectations $O(\delta\sqrt{d})$. From Bernstein's inequality it follows $\1(\cE_i)\sum_{j=1}^n D_{ij} \ls \delta\sqrt{d}n$ with probability $1-O( \exp( - c\delta \sqrt{d}n))$. 
Furthermore, we have $\P( \cE_i) \ge 1-e^{-cd}$ for all $i\in[n]$.
Indeed, from \eqref{assu:no-outliers} we have $\|x_i\|_2 \ls \sqrt{\Tr \Sigma} = o(1)$, while with probability at least $1-e^{-cd}$, $\|z_i\|_2 \in [\frac12,2]$, and hence $\|y_i\|_2 \ge \|z_i\|_2 - \|x_i\|_2 \ge \frac12-o(1)$.
By the union bound, for $C$ sufficiently large we have
\begin{align*}
    \P \bigg( \max_{1\le i\le n} \sum_{j=1}^n D_{ij} > C\delta\sqrt{d}n\bigg) 
    &\le \sum_{i=1}^n \P( \cE_i^c) + \P \bigg( \1(\cE_i)\sum_{j=1}^n D_{ij} > C\delta\sqrt{d}n\bigg)\\
    &\ls n (e^{-cd} + e^{-c\delta\sqrt{d}n}) \\
    &\ls e^{-cd} + ne^{-c\delta\sqrt{d}n}
\end{align*}
where in the last bound we used that $d=\omega(\log n)$. 
\end{proof}

In some of the parameter estimates for the truncated case we will approximate $p$ with $\tilde p$, which we justify below.
\subsubsection{Estimation for Truncated Parameters}
\label{subsec:trunc_param_est}

In this part of the subsection, we show that there is asymptotically (up to leading order) no difference between the parameters of $A$ and $\tilde A$. We begin with $\tilde p$:
\begin{lemma}
\label{lem:ptp-same}
    Under the same conditions as \Cref{prop: Adjacency Matrix Truncation}, we have that
    \[(1+o(1))p=\tilde p.\]
\end{lemma}
\begin{proof}
    We use the same notation as in the proof of \Cref{prop: Adjacency Matrix Truncation}. We will use that for any $s>0$, 
    \begin{equation}\label{eqn: entrywise_p_tildep_UB}
    |\1(Q_{ij}\geq \tau)- \1(\tilde Q_{ij}\geq \tau)|\leq \1[|Q_{ij}-\tilde Q_{ij}|\geq s ] + \1[|Q_{ij}-\tau|<s]+ \1[|\tilde{Q}_{ij}-\tau|<s]\,.    
    \end{equation}
    Fixing $s>0$ to be optimized later, we have
    \begin{equation}\label{eqn:p_trunc_computation}
\begin{aligned}
    |p-\tilde p|&= \bigg|\binom{n}{2}^{-1}\sum_{i<j}^n \Prob\left(\langle y_i,y_j\rangle \geq \tau\right)-\Prob\left(\langle \tilde y_i,\tilde y_j\rangle \geq \tau\right)\bigg|\\
    &\leq \binom{n}{2}^{-1}\sum_{i<j}^n \E\big|\1\left(\langle y_i,y_j\rangle \geq \tau\right)-\1\left(\langle \tilde y_i,\tilde y_j\rangle \geq \tau\right)\big|\\
    &\leq \binom{n}{2}^{-1}\sum_{i<j}^n  \Prob(|Q_{ij}-\tau|\leq s) +\Prob(|Q_{ij}-\tilde Q_{ij}|\geq s)+ \Prob(|\tilde Q_{ij}-\tau|<s)\qquad \text{by }\eqref{eqn: entrywise_p_tildep_UB}.
\end{aligned}
\end{equation}
All that is left is to estimate the probabilities on the right hand side above. Recall from the proof of \Cref{claim:Dij} that 
\begin{equation} \label{eqn: small_ball_expectation}
\Prob(|Q_{ij}-\tau|\leq s), \Prob(|\tilde Q_{ij}-\tau|\leq s) \lesssim s\sqrt{d} + e^{-cd}.\end{equation}
Then, by Markov's inequality,
\begin{align*}
    \sum_{i<j}^n\Prob(|Q_{ij}-\tilde Q_{ij}|\geq s)&\leq\frac{1}{s^2}\sum_{i<j}^n\E(|Q_{ij}-\tilde Q_{ij}|^2) = \frac{1}{s^2}\sum_{i<j}^n\E\big(M_{ij}^2\big),
\end{align*}
with $M=Q-\t Q = X_{k+1:d}X_{k+1:d}^\top + X_{k+1:d}Z^\top + ZX_{k+1:d}^\top$ as in 
\eqref{eqn:Mij}.
We have
\begin{equation}\label{eq:markov_trunc_inner_prod}
\begin{aligned}
\E\norm{M}_F^2
    &\ls \norm{X_{k+1:d}X_{k+1:d}^\top}_F^2 + \E \|X_{k+1:d}Z^\top\|_F^2\\
    &=\norm{X_{k+1:d}X_{k+1:d}^\top}_F^2 +  \frac nd\norm{X_{k+1:d}}_F^2 \\
&\le (\|X_{k+1:d}\|_F^2 + \frac nd) \|X_{k+1:d}\|_F^2
\ls \frac nd \|X_{k+1:d}\|_F^2
\end{aligned}    
\end{equation}
where in the final bound we used our assumption \eqref{assu:gap-gen} that $\frac1n \|X_{k+1:d}\|_F^2 = \Tr\Sigma_{k+1:d}=o(1/d)$.
Combining with 
\eqref{eqn:p_trunc_computation} yields
\[
|p-\tilde p|\lesssim s\sqrt{d}+ e^{-cd}+ \frac{1}{s^2d}
\Tr \Sigma_{k+1:d}
\,.\]
Optimizing over $s$ 
gives 
$s= \frac1{\sqrt{d}}(\Tr\Sigma_{k+1:d})^{1/3}$, which by plugging into the right hand side of our bound, and taking dominant terms yields 
\[|p-\tilde p|\lesssim (\Tr\Sigma_{k+1:d})^{1/3}
+e^{-cd}=o(p), \text{ by }\eqref{assu:gap-gen}\]
as desired.

\end{proof}

\subsubsection{Proof of \texorpdfstring{\Cref*{prop:YY_diff_gen}}{Proposition 4.2}}

Recall that in \Cref{prop:YY_diff}, we controlled the quantity $\|YY^\top-\hat Y\hat Y^\top\|_\op$ by comparing the adjacency matrix to another matrix $B_\data$. To prove \Cref{prop:YY_diff_gen}, we must instead control $\|\tilde Y\tilde Y^\top-\hat Y\hat Y^\top\|_{\op}$, we do this by comparing $\tilde Y\tilde Y^\top$ to the truncated versions 
of \Cref{def:Bdata_Bsph},
\begin{equation}\label{eqn:B_data_gen}
    \boxed{
    \begin{aligned}
    \tilde{B}_\data &= p_0\zeta\zeta^\top + d_kp_1 \tilde Y\tilde Y^\top\\
    \tilde B_\prj &= p_0\zeta\zeta^\top+d_k p_1\Pi^\perp_{{\zeta}}(\tilde Y\tilde Y^\top)\Pi^\perp_{{\zeta}}.
    \end{aligned}}
\end{equation}

\begin{proof}[Proof of \Cref{prop:YY_diff_gen}]
Recall that in this context $B_\spec = \eta_0\mathbf{w}_0\mathbf{w}_0+d_kp_1\hat Y\hat Y^\top$ is the rank--$d$ spectral decomposition of the \textit{observed} adjacency matrix $A$. Moreover, if $\tilde A,\tilde Y,\{\tilde x_i\}_{i=1}^n,\tilde p$ are the truncated analogues of
$A,Y,\{x_i\}_{i=1}^n,p$, then \Cref{lemma:trace_method,lemma:Bsph_Bdata_diff,lem:eigenvector_diff}
hold verbatim with these substitutions. In particular, note that $B_\sph$ in \eqref{def:Bdata_Bsph}, is now defined in terms of the decomposition of $\tilde y=(a_i,w_i)$ from \eqref{eqn: data decomp}.

Since $\{\tilde x_i\}_{i=1}^n=\{\Pi_X^{(k)}x_i\}_{i=1}^n$ is exactly rank--$k$, the noisy vectors $\tilde y_i=\tilde x_i+z_i$ satisfy the assumptions of \Cref{thm:lowrank} and the cited proofs depend on the signal only through the parameter estimates in \Cref{subsubsection:Parameter_estimation}, which remain unchanged after truncation due to \Cref{lem:ptp-same} and $\tau$ remaining fixed; in particular, we note that coefficients $p_i$ in the Gegenbauer expansion $\1(x\geq \tau)=p_0+ p_1x+ \dots$ remain the same. Thus,

\begin{align}
    d_kp_1\|\hat Y \hat Y ^\top-\tilde Y\tilde Y^\top \|_\op
    &\leq 
    \|B_\spec-A\|_\op
    + \|\tilde A-A\|_\op+\|\eta_0\mathbf{w}_0\mathbf{w}_0^\top- p_0\zeta\zeta^\top\|_\op +  \|\tilde A-
    \tilde{B}_\data
    \|_\op  \notag\\
    &\leq 2\|A-\tilde{B}_\data\|_\op + 2\|\tilde A-A\|_\op+\|\eta_0\mathbf{w}_0\mathbf{w}_0^\top- p_0\zeta\zeta^\top\|_\op   \notag\\
    &\leq  2 \|\tilde A-\tilde{B}_\sph\|_\op+2\|\tilde{B}_\sph-\tilde{B}_\data\|_\op +\|\eta_0\mathbf{w}_0\mathbf{w}_0^\top- p_0\zeta\zeta^\top\|_\op+ 4\|\tilde A-A\|_\op.   \notag\\
    &\leq  \left(2\|\tilde A-\tilde{B}_\sph\|_\op+2\|\tilde{B}_\sph-\tilde{B}_\data\|_\op +\|\tilde{\eta}_0\tilde{\mathbf{w}}_0\tilde{\mathbf{w}}_0^\top- p_0\zeta\zeta^\top \|_\op\right)   \notag\\
    & \qquad + \| \eta_0 \mathbf{w}_0 \mathbf{w}_0^\top - \tilde{\eta}_0 \tilde{\mathbf{w}}_0\tilde{\mathbf{w}}_0^\top\|_\op + 4 \|A-\tilde{A} \|_\op.
    \label{eqn:prop_YY_diff_chain_gen}
\end{align}

Note that the terms inside the parentheses can be controlled via \Cref{lemma:trace_method}, \Cref{lemma:Bsph_Bdata_diff} and \Cref{lem:eigenvector_diff}, which combined together give the conclusion of \Cref{prop:YY_diff}.
We already have an estimate on $\|A-\tilde A\|_\op$, while $\| \eta_0 \mathbf{w}_0 \mathbf{w}_0^\top - \tilde{\eta}_0 \tilde{\mathbf{w}}_0\tilde{\mathbf{w}}_0^\top\|_\op$ can be controlled via Davis--Kahan and Weyl's inequality:
\begin{equation}\label{eqn:DavisKahan_tilde_A}
\| \eta_0 \mathbf{w}_0 \mathbf{w}_0^\top - \tilde{\eta}_0 \tilde{\mathbf{w}}_0\tilde{\mathbf{w}}_0^\top\|_\op
\leq \rvert \eta_0-\tilde\eta_0\lvert+ \tilde \eta_0\|\mathbf{w}_0 -\tilde{\mathbf{w}}_0\|_2 \leq \|A-\tilde A\|_\op+ 2\tilde \eta_0\frac{\|A-\tilde A\|_\op}{\lambda_1(\tilde A)-\lambda_2(\tilde A)}.\end{equation}
We estimate the spectral gap $\lambda_1(\tilde A)-\lambda_2(\tilde A)$ by first applying Weyl's inequality to relate it to $\lambda_1(\tilde B_\data) - \lambda_2(\tilde B_\data)$. Next $\lambda_1(\tilde B_\data) - \lambda_2(\tilde B_\data)$ can be estimated by following the argument in \Cref{subsec:rank-1-diff}; to find the eigenvalues of $\tilde B_\data,$ we must project away from the first eigenspace, which is spanned by $\zeta$. To this end, we apply Weyl's inequality once again:
\begin{equation}\label{eqn:gap_tilde_A}
\begin{aligned}
    \lambda_1(\tilde A)-\lambda_2(\tilde A)&\geq
    \lambda_1(\tilde B_\data)-\lambda_2(\tilde B_\data)-2\|\tilde A-\tilde B_\data\|_\op\\
    &\geq \lambda_1(\tilde B_\prj)-\lambda_2(\tilde B_\prj)-2( \|\tilde B_\data-\tilde B_\prj\|_\op+\|\tilde A-\tilde B_\data\|_\op).
\end{aligned}
\end{equation}
By \eqref{eqn:lambda1_B_proj} 
, we have that $\lambda_1(\tilde B_\prj)=\Theta(\tilde pn)=\Theta(pn)$, where the last equality is given by \Cref{lem:ptp-same}. Moreover, the bounds from  
Lemmas \ref{lemma:trace_method}, \ref{lemma:Bsph_Bdata_diff} and
\ref{lemma:projection_of_data_matrix} 
give: $\|\tilde A-\tilde B_\data\|_\op+\|\tilde B_\data-\tilde B_\prj \|_\op=o(pn)$ by \eqref{assu:density}, \eqref{assu:trace} and \Cref{lem:ptp-same}. Then \Cref{eqn:kplus1_eval_UB} yields \begin{equation}
\lambda_2(\tilde B_\prj)=d_kp_1\lambda_1(\tilde Y\tilde Y^\top)=\Theta(dp\tau)\cdot \left(n\mu_1^2+ \tfrac{n}{d}+ O(\sqrt{n}\mu_1)\right)=o(n p).
\end{equation}
The last equality was obtained by \eqref{assu:trace}. Then, plugging our estimates on $\lambda_1(\tilde B_{\proj})$, $\lambda_2(\tilde B_{\proj})$ into \eqref{eqn:gap_tilde_A} yields $\lambda_1(\tilde A)-\lambda_2(\tilde A)\geq \Theta(np)-o(np)$. Applying this last inequality to \eqref{eqn:DavisKahan_tilde_A} gives:
\begin{equation}\label{eqn:eigenvector_diff_tilde_A}
    \| \eta_0 \mathbf{w}_0 \mathbf{w}_0^\top - \tilde{\eta}_0 \tilde{\mathbf{w}}_0\tilde{\mathbf{w}}_0^\top\|_\op\lesssim \|A-\tilde A\|_\op\left(1+ \frac{\Theta(np)}{\Theta(np)-o(pn)}\right)\lesssim \|A-\tilde A\|_\op.
\end{equation}
Plugging \eqref{eqn:eigenvector_diff_tilde_A} into \eqref{eqn:prop_YY_diff_chain_gen} and applying \Cref{prop:YY_diff} to the resultant expression, we obtain 
\[d_kp_1\|\hat Y\hat Y^\top-\tilde Y\tilde Y^\top\|_\op\lesssim np(\tau^2+d\mu_1^2\Tr\Sigma)\log^4 n +\max\{np\tau^2, \sqrt{np}\}\log^9n+ dp\tau\log^4n+\|A-\tilde A\|_\op. \]
Finally, applying \Cref{prop: Adjacency Matrix Truncation} to the above display, taking dominant terms and dividing by $d_kp_1=\Theta(dp\tau)$, produces the desired result.
\end{proof}

\section{Proof of Frobenius and Entrywise Error Bounds}

\begin{proof}[Proof of \Cref{cor: Frob cor1}]
     To obtain \eqref{eqn:frob_cor}, it suffices to decompose  $XX^\top$ in terms of its first $k$ eigenspaces. Recall that $X_k=\Pi_{X}^{(k)}X$ and $X_{k+1:d}=X-X_k$. Then, 
\begin{align}
\|XX^\top - \Xalg \Xalg^\top\|_F&\le \|X_kX_k^\top- \Xalg\Xalg^\top\|_F+ \|X_{k+1:d}X_{k+1:d}^\top\|_F  
\notag\\
&\leq \sqrt{2k} \|X_kX_k^\top - \Xalg\Xalg^\top\|_\op+ %n\mu_{k+1}\sqrt{\Tr\Sigma_{k+1:d}}  
n \Big( \sum_{i>k} \mu_i^4 \Big)^{1/2} \notag\\
&\leq \sqrt{2k}(\|XX^\top -\Xalg\Xalg^\top\|_\op+n\mu_{k+1}^2)+ %n\mu_{k+1}\sqrt{\Tr\Sigma_{k+1:d}}. 
n \Big( \sum_{i>k} \mu_i^4 \Big)^{1/2}.
\label{eqn:frob_cor_comp}
\end{align}
Then, one can use \eqref{bd:main} for the first term on the right hand side. Note that $n\mu_{k+1}^2=o(n\mu_1d^{-1/2})\le \cE_\op(X)$ due to \eqref{assu:gap-gen}, so we can drop the $n\mu_{k+1}^2$ term in \eqref{eqn:frob_cor_comp}.  This concludes the proof of \eqref{eqn:frob_cor}. 
\eqref{eqn:inner_prod_recover} is immediate from applying Markov's inequality and \eqref{eqn:frob_cor} to obtain the mean squared error of the inner products. To prove \eqref{eqn:norm_recovery}, we first define a ``signed'' matrix $S=\diag(s_1,\dots, s_n)$ where $s_i=\text{sign}( \|x_i\|_2^2- \|x_i^\alg\|_2^2).$ We can write the difference between norms as
\begin{equation}\label{eqn:Tr_computation_signed}
    \begin{aligned}
    \sum_{i=1}^n |\| x_i \|_2^2-\| x_i^\alg\|_2^2 |&= \sum_{i=1}^ns_i(\| x_i \|_2^2-\| x_i^\alg\|_2^2)=\Tr(S( X X^\top-\Xalg\Xalg^\top ))\\
    &= \Tr(S(X_kX_k^\top-\Xalg\Xalg^\top ))+ \Tr(S(X_{k+1:d}X_{k+1:d}^\top))\\
    &\leq 2k\|S(X_kX_k^\top-\Xalg\Xalg^\top)\|_\op+ n\Tr\Sigma_{k+1:d}\\
    &\leq 2k(\|S\|_\op\|XX^\top-\Xalg\Xalg^\top\|_\op+n\mu_{k+1}^2)+n\Tr\Sigma_{k+1:d}.
    \end{aligned}
    \end{equation}
The penultimate inequality was obtained by using the fact that $S(X_kX_k^\top-\Xalg\Xalg^\top)$ is at most rank-$2k$, the last inequality was obtained by a reverse triangle inequality, and $\|S\|_\op=1$.
Applying Markov's inequality to the above gives \eqref{eqn:norm_recovery}, then using the inequality $(|a|-|b|)^2\leq |a^2-b^2|$ and again noting that $n\mu_{k+1}^2=o(\cE_\op(X))$  
gives \eqref{eqn:norm_recovery2}.
\end{proof}

    Moreover, if one restricts the point cloud to the following good sets, we can quantify the relative recovery;
    \begin{equation} \label{eq:good_sets_lb}
    \cA(t):=\{(i,j): i\neq j,|\langle x_i,x_j\rangle| \geq  t\},\qquad \cA'(t):=\{i\in [n]: \|x_i\|_2\geq t\}.\end{equation}
\begin{cor}\label{cor:relative_recovery}
Under the same conditions as \Cref{thm:main}, we have that with probability at least $1-O(n^{-10})$, for any $\varepsilon,\delta \in(0,1]$, for some sufficiently large constant $C>0$ and $t= \log^C n \cE_F(X) n^{-1}\varepsilon^{-1/2}\delta^{-1}$
\begin{equation}\label{eq:inner_prod_recovery_relative}
    \lvert \{(i,j)\in \cA(t): \langle x_i^\alg,x_j^\alg\rangle =(1+O(\delta))\langle x_i,x_j\rangle)\}\rvert\geq \lvert \cA(t)\rvert -\varepsilon n^2;
\end{equation}
if $t=\log^{C/2}n\left(\frac{2k\cE_\op(X)+n\Tr\Sigma_{k+1:d}}{\varepsilon \delta n}\right)^{1/2}$,
\begin{equation}\label{eq:norm_recovery_relative}
    \lvert\{i\in \cA'(t): \| x_i^\alg\|_2 =(1+O(\delta))\| x_i\|_2)\}\rvert\geq |\cA'(t)|-\varepsilon n.
\end{equation}
\end{cor}
\begin{proof}[Proof of \Cref{cor:relative_recovery}]
    The proof of \eqref{eq:inner_prod_recovery_relative} follows immediately from the conclusion of \eqref{eqn:norm_recovery} and $|\langle x_i,x_j\rangle-\langle x_i^\alg,x_j^\alg\rangle|\leq \delta\log^C n\cdot  \cE_F(X) n^{-1}\varepsilon^{-1/2}\delta^{-1}\leq \delta |\langle x_i,x_j\rangle|$, on the good pairs $(i,j)\in\cA(\log^C n\cdot \cE_F(X) n^{-1}\varepsilon^{-1/2}\delta^{-1})$. As for \eqref{eq:norm_recovery_relative}, consider $i\in \cA'(\log^{C/2}n\cdot \left(\frac{2k\cE_\op(X)+n\Tr\Sigma_{k+1:d}}{\varepsilon \delta n}\right)^{1/2})$, and such that $i$ does \emph{not} belongs in the $\varepsilon n$ indices from the set \eqref{eqn:norm_recovery}. Then,
    \begin{equation}\label{eqn:good_index_bd}
    |\|x_i\|_2^2-\|x_i^\alg\|_2^2|\leq \log^{C}n\cdot\left(\frac{2k\cE_\op(X)+n\Tr\Sigma_{k+1:d}}{\varepsilon  n}\right)=\delta\log^{C}n\cdot\left(\frac{2k\cE_\op(X)+n\Tr\Sigma_{k+1:d}}{\varepsilon \delta n}\right) \leq \delta\|x_i\|_2^2.
    \end{equation}
    Then we recover the difference of the norms as follows:
    \begin{equation}
        \begin{aligned}
        |\|x_i\|_2-\|x_i^\alg\|_2|&=\frac{ |\|x_i\|^2_2-\|x_i^\alg\|^2_2|}{\|x_i\|_2+\|
        x_i^\alg\|_2}\\
        &\leq\frac{ |\|x_i\|^2_2-\|x_i^\alg\|^2_2|}{\|x_i\|_2}\\
        &\leq  \delta\|x_i\|_2^2/\|x_i\|_2= \delta \|x_i\|_2.
        \end{aligned}
    \end{equation}
The last inequality was obtained by applying \eqref{eqn:good_index_bd}. An application of Markov's inequality concludes the proof.
\end{proof}

\section{Proofs of Application to Random Sampling}
\label{sec: manifold denoising}

\subsection{Low-rank variant}

We begin by stating and proving a low-rank variant of \Cref{thm:numerically_low_rk_denoising}, meant as an analogue to \Cref{thm:lowrank}.
\begin{theorem}[Denoising with low-rank $\Sigma^{\infty}$]
\label{thm:manifold_denoising_low_rank}
    With $n,d,p$ as in \Cref{thm:lowrank}, let
    $X$ be the matrix whose rows are given by the point cloud $\cX = \{x_1, \ldots, x_n\}$ sampled independently from a 
    distribution $\nu$ in $\R^d$, with 
    $\Sigma^{\infty}$ as defined in \Cref{eq:sigmaMdef} of rank $k$. Moreover, suppose
    \begin{equation}
        \label{assu:trace_truecov}
        \Tr \Sigma^{\infty} \lesssim \frac{1}{\sqrt{d}\log^2 n},
    \end{equation}
        \begin{equation}
\label{assu:no_outliers_truecov}
\esssup_{\bfx\sim\nu}\norm{\bfx}_2^2 \lesssim  \Tr \Sigma^{\infty}, 
    \end{equation}
    \begin{equation}
    \label{assu:spec_gap_truecov}
    \lambda_k(\Sigma^{\infty}) = \omega(d^{-1}).
    \end{equation}
    Let $G \sim \mathcal{G}_{n,d}(p, \cX)$. For the spectral algorithm of Section \ref{sec:algorithm},  
    with probability at least $1-O(n^{-10})$:
    \begin{equation}\label{eqn:low-rank denoising}
    \norm{XX^\top-\hat X\hat X^\top}_\op \lessapprox n\lambda_1\left(\frac{1}{\sqrt{\lambda_1 d}} +\frac{\lambda_1}{\lambda_k}\sqrt{d}\Tr\Sigma^{\infty} + \frac{1}{\lambda_k}\sqrt{\frac{1}{dnp}} \right).
\end{equation}
\end{theorem}

We prove \Cref{thm:manifold_denoising_low_rank} by applying Theorem \ref{thm:lowrank} to $\cX$. Let  $\Sigma = \frac{1}{n} X^\top X$ and $\{\mu^2_i\}_{i=1}^n$ the singular values of $\Sigma$, as before. We must verify that $\Sigma$ is strictly low rank and that the conditions we impose on $\Sigma^{\infty}$ imply the analogous assumptions \eqref{assu:trace}, \eqref{assu:no-outliers}, \eqref{assu:gap-gen} on $\Sigma$ in the statement of \Cref{thm:lowrank}. To this end, we begin by making the informal statement \eqref{eq:sigma_to_sigmaf} precise.

\begin{lemma}[Concentration of $\Sigma$ around $\Sigma^{\infty}$]\label{lem: discrete trace limit}
If $\esssup_{\bf{x} \sim \nu} \norm{{\bf x}}_2^2 = O(\Tr \Sigma^{\infty})$, then there exists $C \geq 0 $ such that with probability at least $1-n^{-11}$ 

:
\begin{equation}
\label{eqn:single entry concentrates}
\big \vert \Sigma_{jj} - \Sigma^{\infty}_{jj} \big \vert \leq  \norm{\Sigma-\Sigma^{\infty}}_\op  \leq C\Tr \Sigma^{\infty}\sqrt{\frac{\log n}{n}} \ \ \ \ \forall j:  1 \leq j \leq d, 
\end{equation}
\begin{equation}
\label{eqn: tail_trace_concentrates}
\bigg\vert\sum_{j=\ell}^d\Sigma_{jj} - \sum_{j=\ell}^d \Sigma^{\infty}_{jj} \bigg\vert\leq C\sqrt{\Tr \Sigma^{\infty} \Tr\Sigma_{\infty,\ell:d}}\sqrt{\frac{\log n}n} \ \ \ \ \forall \ell: 1\leq \ell \leq d,\end{equation}
in particular, setting $\ell = 1$ in \eqref{eqn: tail_trace_concentrates} we obtain
\begin{equation}
\label{eqn: curve-trace of sigma concentrates}
\big\vert\Tr\Sigma - \Tr \Sigma^{\infty} \big\vert \leq C\Tr\Sigma^{\infty}\sqrt{\frac{\log n}n}.
\end{equation}

\end{lemma}

\begin{proof}
We begin by proving \eqref{eqn: tail_trace_concentrates}. We apply Bernstein's inequality to 
$$\sum_{j=\ell}^d \Sigma_{jj} - \sum_{j = \ell}^d \Sigma^{\infty}_{jj} = \frac{1}{n}\sum_{i=1}^n\bigg(\sum_{j=\ell}^d (x_i)_j^2 - \E\bigg[\sum_{j=\ell}^d (x_i)_j^2\bigg]\bigg),$$ where $x_i\sim \nu$ are i.i.d, and $(x_i)_j$ denotes the $j-$th coordinate of $x_i$. By \eqref{assu:no_outliers_truecov}
$$\Var\left[\sum_{j=\ell}^d (x_i)_j^2\right]\leq O(\Tr\Sigma^{\infty}) \E_{x_{i}\sim \nu } \left[\sum_{j=\ell}^d (x_i)_j^2\right] = O(\Tr\Sigma^{\infty})\left[\sum_{j=\ell}^d \Sigma^{\infty}_{jj}\right].$$
Thus, Bernstein inequality for bounded variables and \eqref{assu:no_outliers_truecov} yield \eqref{eqn: tail_trace_concentrates}.

Next, we prove \eqref{eqn:single entry concentrates} -- in particular, we focus on the second inequality in \eqref{eqn:single entry concentrates} as the first is immediate from the definition of $\|\cdot\|_\op$. Here we apply matrix Bernstein inequality (e.g.\cite[ Theorem 5.4.1]{Vershynin2018}) to 
$$ \Sigma - \Sigma^{\infty} = \sum_{i=1}^n\frac{1}{n}(x_i^\top x_i-\Sigma^{\infty}).$$ To do so, we must verify that $\|\frac{1}{n}(x_i^\top x_i-\Sigma^{\infty})\|_\op$ is a bounded random variable. Note that assumption \eqref{assu:no_outliers_truecov} implies that 
$$\|x_i^\top x_i-\Sigma^{\infty}\|_{\op}\leq\|x_i\|_2^2+\|\Sigma^{\infty}\|_{\op}\le O(\Tr\Sigma^{\infty})^2$$ where we used $\|\Sigma^{\infty}\|_{\op}\leq\E_{\bfx\sim\nu}\|\bfx\|_2^2=\Tr\Sigma^{\infty}$.
Moreover, the matrix analogue of variance satisfies
\begin{align*}
\norm{\frac{1}{n^2}\sum_{i=1}^n\E_{x_i\sim \nu}\big(x_i x_i^\top-\Sigma^{\infty}\big)^2}_\op&
  =\frac{1}{n}\norm{\E_{\bfx\sim\nu}[\|\bfx\|_2^2 \bfx \bfx^\top \big]-(\Sigma^{\infty})^2}_\op\\
&\leq\frac{1}{n}(O(\Tr\Sigma^{\infty})\|\Sigma^{\infty}\|_\op+\|\Sigma^{\infty}\|_\op^2\big)\\
&\leq\frac{O(\Tr\Sigma^{\infty})^2}{n}.
\end{align*}
Therefore, we obtain
\[\Prob\left(\norm{\Sigma-\Sigma^{\infty}}_\op\geq t\right)\leq 2d\exp\left(-c\min\bigg\{\frac{t^2n}{O(\Tr\Sigma^{\infty})^2},\frac{tn}{O(\Tr\Sigma^{\infty})}\bigg\}\right). \]
Setting $t=C(\Tr\Sigma^{\infty})\sqrt{\frac{\log n}{n}}$ yields the second inequality in \eqref{eqn:single entry concentrates}, as desired. 
\end{proof}

 We are now ready check the conditions required for the application of \Cref{thm:lowrank} and complete the proof of Theorem \ref{thm:manifold_denoising_low_rank}.
\begin{proof}[Proof of \Cref{thm:manifold_denoising_low_rank}]

First, we show that if $\Sigma^{\infty}$ is rank $k$, then $\Sigma$ is at most rank $k$ a.s. Let $u_1$, $\ldots$, $u_d$ denote the eigenbasis for $\Sigma^{\infty}$ with $u_{k+1}$, $\ldots$, $u_d$ corresponding to the 0 eigenvalue. Then, for $i \geq k+1$ $$0 = u_{i}^T \Sigma^{\infty} u_i = \mathbb{E}[u_i^{\top} \Sigma u_i].$$
    Since $\Sigma$ is a positive semidefinite matrix, $u_i^{\top} \Sigma u_i \geq 0$ and the above identity implies $u_i^{\top} \Sigma u_i = 0$ a.s. Thus, $\dim \ker \Sigma \geq d-k$ a.s. and the desired conclusion follows.

Next we verify the following conditions (1)--(3), corresponding to \eqref{assu:trace}--\eqref{assu:gap-gen}, respectively, hold with high probability.

\begin{enumerate} 

\item $\Tr \Sigma \lesssim \frac{1}{\sqrt{d} \log^2 n}$ : 
This follows immediately from \eqref{assu:trace_truecov} and \eqref{eqn: curve-trace of sigma concentrates}.

\item The point cloud $\cX$ has no outliers (\eqref{assu:no-outliers}) with probability  $1 -O(n^{-10})$,  i.e.  $\max_{i, j} \langle x _i, x_j \rangle = O(\Tr \Sigma)$ :
By \eqref{eqn: curve-trace of sigma concentrates} we get that there exists $C\geq 0$ such that with probability $ 1- O(n^{-10})$
\begin{equation}\label{eqn: continuous regularity computation}
\Tr\Sigma\geq \Tr \Sigma^{\infty}\left( 1-C\sqrt{\frac{\log n}{n}}\right).
\end{equation}
Therefore, letting $n$ be sufficiently large (so that the right side of the above display is positive), we obtain
\[\max_{i,j}\langle x_i,x_j\rangle\leq \esssup_{\mathbf{x} \sim \nu} \| \mathbf{x}\|_2^2 =O(\Tr\Sigma^{\infty}) = O(\Tr\Sigma) \]
where the second inequality follows from \eqref{assu:no_outliers_truecov}.
\item $\mu_k^2 = \omega(d^{-1})$ : 
By Weyl's inequality and \eqref{eqn:single entry concentrates} we have 
$\mu_k^2 \geq \lambda_k(\Sigma^{\infty}) - C\Tr \Sigma^{\infty} \sqrt{\frac{\log n}{n}} = \omega(d^{-1}),$
where the last relation follows by \eqref{assu:spec_gap_truecov} and \eqref{assu:trace_truecov}.
\end{enumerate}
Thus, with high probability, Theorem \ref{thm:lowrank} applies to $\cX$ and \eqref{eqn:low-rank denoising} follows.
\end{proof}

To prove \Cref{thm:numerically_low_rk_denoising} we will check that with high probability conditions of \Cref{thm:main}, in place of \Cref{thm:lowrank} as in the proof of \Cref{thm:manifold_denoising_low_rank}, are satisfied by the point cloud $\cX$.

\subsection{Proof of \Cref{thm:numerically_low_rk_denoising}}

% \begin{proof}
% We verify that the following conditions (1)--(3) hold with high probability.
With conditions  (1) and (2) as in the proof of Theorem \ref{thm:manifold_denoising_low_rank}, these properties hold with high probability by the same argument. 
To apply \Cref{thm:main} it only remains to verify condition \eqref{assu:gap-gen},
that $\Sigma$ is numerically low rank,  i.e.  for some $k = o(\sqrt{d}/\log^2 n)$ we have that $\mu_k \geq \omega(d^{-1/2})$ and $\Tr \Sigma_{k+1:d} \leq o\left(\min(p^3, d^{-1})\right)$.

Let $k$ be such that $\lambda_k(\Sigma^{\infty}) \geq \omega(d^{-1})$ and $\sum_{j=k+1}^d \lambda_j(\Sigma^{\infty}) = o\left(\min(p^3,d^{-1})\right)$. The fact that $\mu_k = \omega(d^{-1/2})$ follows from the same argument as in the proof of \Cref{thm:manifold_denoising_low_rank}.

The rest of this proof is dedicated to verifying that $\Tr \Sigma_{k+1:d} \leq o\left(\min(p^3, d^{-1})\right)$. For this argument, it is convenient to assume that $\Sigma^{\infty}$ is diagonal -- we can do so without loss of generality as the properties \eqref{assu:trace_num_low_rank}--\eqref{assu:spec_gap_num_low_rank}
remain unchanged if we diagonalize $\Sigma^{\infty}$. Additionally, without loss of generality we assume that $\Sigma^{\infty}_{11} \geq \ldots \geq \Sigma^{\infty}_{dd}.$
    
By the Horn--Schur Theorem (e.g. see \cite[Theorem 4.3.45 (Schur)]{horn2012matrix}) and \eqref{eqn: tail_trace_concentrates}, we have that with probability at least $1-n^{-11}$
    \begin{equation}
    \label{eqn:schur_horn_thm}
        \Tr \Sigma_{k+1: d} \leq \sum_{j=k+1}^d \Sigma_{\ell_j \ell_j} \leq \sum_{j=k+1}^d \Sigma^{\infty}_{\ell_j\ell_j} + C\sqrt{\Tr \Sigma^{\infty} \Tr \Sigma_{\infty, k+1:d}}\sqrt{ \frac{\log n}{n}}
    \end{equation}
    where $\Sigma_{\ell_1 \ell_1} \geq \Sigma_{\ell_2\ell_2} \geq \ldots \geq \Sigma_{\ell_d \ell_d}$. Next, we show that with high probability $\{\ell_{k+1}, \ldots, \ell_d\} = \{k+1, \ldots, d\}$.
\begin{align}
   \mathbb{P}\big(\{\ell_{k+1}, \ldots, \ell_d\}  \neq & \{k+1, \ldots, d\}\big)\nonumber \\
   \leq &\sum_{j=1}^k \mathbb{P}(j \in \{\ell_{k+1}, \ldots, \ell_d\})\nonumber \\
    = & \sum_{j=1}^k \mathbb{P}(\Sigma_{jj} < \Sigma_{\ell \ell} \text{ for some $\ell \geq k+1$})\nonumber \\
    \leq & \sum_{j=1}^k \bigg[ \mathbb{P}\bigg(\Sigma_{jj} < \Sigma^{\infty}_{k+1,k+1}+ C\Tr \Sigma^{\infty} \sqrt{\frac{\log n}{n}}\bigg) \nonumber\\
    & + \mathbb{P}\bigg(\max_{\ell \geq k+1} \Sigma_{\ell \ell} \geq \Sigma^{\infty}_{k+1,k+1}+ C\Tr \Sigma^{\infty} \sqrt{\frac{\log n}{n}} \bigg)\bigg] \nonumber\\
    \leq & \sum_{j=1}^k \bigg[ \mathbb{P}\bigg(\Sigma_{jj} < \Sigma^{\infty}_{jj} -  C\Tr \Sigma^{\infty} \sqrt{\frac{\log n}{n}}\bigg) \label{eq: intermediate simplification}\\
    & + \sum_{j=k+1}^d \mathbb{P}\bigg(\Sigma_{jj} \geq \Sigma^{\infty}_{jj}+ C\Tr \Sigma^{\infty}  \sqrt{\frac{\log n}{n}} \bigg)\bigg] \nonumber\\
    \leq & dn^{-11}, \label{eqn:no_crossings_whp}
\end{align}
where to obtain \eqref{eq: intermediate simplification} we observe that for all $j \leq k$
$$\left(\Sigma^{\infty}_{jj}- C \Tr \Sigma^{\infty} \sqrt{\frac{\log n}{n}} \right)-\left( \Sigma^{\infty}_{k+1, k+1}+ C \Tr \Sigma^{\infty} \sqrt{\frac{\log n}{n}} \right) = \omega(d^{-1}) > 0$$
since \eqref{assu:spec_gap_num_low_rank} implies $\Sigma^{\infty}_{jj} - \Sigma^{\infty}_{k+1, k+1} =\omega(d^{-1})$ and \eqref{assu:trace_num_low_rank} with $d \leq n$ imply 
$$\Tr \Sigma^{\infty}\sqrt{\frac{\log n}{n}} \lesssim \frac{1}{\sqrt{dn \log n}} = o(d^{-1}).$$
Finally, \eqref{eqn:no_crossings_whp}
follows from \eqref{eqn:single entry concentrates}.

Now we combine \eqref{eqn:schur_horn_thm}, \eqref{eqn:no_crossings_whp} and \eqref{assu:trace_num_low_rank} to get that with probability $1 - O(n^{-10})$
\begin{equation}
\label{eq:tail_sv_bd}
\begin{aligned}
    \Tr\Sigma_{k+1:d}&\leq \Tr \Sigma^{\infty}_{k+1:d} + C\sqrt{\Tr \Sigma^{\infty} \Tr \Sigma_{\infty, k+1:d}}\sqrt{\frac{\log n}{n}} = o\left(\min(p^3, d^{-1})\right).
\end{aligned}
\end{equation}
%\end{enumerate}

Since we showed that conditions on $\cX$ required for application of \Cref{thm:main} are satisfied, the stated result \eqref{eqn:diag_manifold_error_rate_OP} follows if we substitute $\mu_i^2=(1+o(1))\lambda_i$, $\Tr\Sigma=(1+o(1))\Tr\Sigma^{\infty}$ and $\Tr\Sigma_{k+1:d}=(1+o(1))\Tr\Sigma^\infty_{ k+1:d}$ into \eqref{bd:main}.
% \end{proof}
\qed

\subsection{Proof of \Cref{thm:manifold_denoising_approx_diag}}
%\begin{proof}
By Weyl's inequality,
\begin{equation}
\label{eqn:weyl_ineq_inter_approx_diag}
    \lambda_k \geq \Sigma^{\infty}_{kk} - \norm{E}_\op \gtrsim \Sigma^{\infty}_{kk}  = \omega(d^{-1}).
\end{equation}
By Horn--Schur Theorem,
\begin{equation}
\label{eqn:horn_schur_approxdiag}
    \sum_{j=k+1}^d \lambda_j \leq \sum_{j=k+1}^d \Sigma^{\infty}_{jj} = o\bigg(\min(p^3, d^{-1})\bigg).
\end{equation}
Thus, conditions \eqref{assu:trace_num_low_rank}--\eqref{assu:spec_gap_num_low_rank} are satisfied and we can apply \Cref{thm:numerically_low_rk_denoising} to $\cX$ to obtain 
\begin{align*}
\|XX^\top- \Xalg\Xalg^\top\|_\op
    \lessapprox&
    n\lambda_1\left(\frac{1}{\sqrt{d\lambda_1}} +\frac{\lambda_1}{\lambda_k}\sqrt{d}\Tr\Sigma^{\infty} + \frac{1}{\lambda_k\sqrt{d}}\left(\sqrt{\frac{1}{np}}+\frac{(\sum_{j=k+1}^d \lambda_j )^{1/4}}{p}\right)\right).
\end{align*}
We then substitute $\lambda_1$ and $\lambda_k$ with $\Sigma^{\infty}_{11}$ and $\Sigma^{\infty}_{kk}$, respectively by \eqref{eqn:weyl_ineq_inter_approx_diag} and because by Weyl's inequality and the fact that $\lambda_k = \omega(d^{-1})$ we have 
$$\lambda_1 \leq \Sigma^{\infty}_{11} + \norm{E}_\op \lesssim \Sigma^{\infty}_{11}.$$
Similarly, we can replace $\sum_{j=k+1}^d \lambda_j$ by $\sum_{j=k+1}^d \Sigma^{\infty}_{jj}$ due to \eqref{eqn:horn_schur_approxdiag}. 
%\end{proof}
\qed

\subsection{Examples}

\subsubsection{Sphere in $m$ dimensions}
\begin{lemma}
\label{lem:inner prod lower bd}
    Let $x_1$, $\ldots$, $x_n$ denote i.i.d. samples from $\Unif(R\mathcal{S}^{m-1})$. Then, for $0 \leq \varepsilon_0 \leq \frac{1}{8}$
    $$\bigg\vert \bigg\{(i,j): i \neq j, \vert \langle x_i, x_j \rangle \vert \geq \frac{\varepsilon_0 R^2}{\sqrt{m}}\bigg \} \bigg\vert \geq \bigg(1 - \sqrt{\frac{96 \log n}{(n-1)}}\bigg)\left(1-2\varepsilon_0 - \frac{1}{m}\right)n(n-1)$$
    with probability at least $1-n^{-11}$.
\end{lemma}

\begin{proof}
    
We begin by recalling that if $U$ is an $m \times m$ orthogonal matrix independent of $x_i$ then $U x_i \stackrel{d}{=} x_i$. Thus, if we $U$ satisfies $Ux_j = Re_1$ we get that
\begin{equation}
\label{eq:rotational invariance}
    \langle x_i, x_j \rangle \stackrel{d}{=} (x_i)_1.
\end{equation}
Let $g$ denote an $m$-dimensional standard Gaussian vector. Using \eqref{eq:rotational invariance} and the fact that $x_i \stackrel{d}{=} \frac{R g}{\|g\|_2}$, for all $\varepsilon \geq 0$ we have that
\begin{align*}
    \P\bigg(\vert \langle x_i, x_j \rangle \vert \leq \frac{\varepsilon R^2}{\sqrt{m}}\bigg)= & \P\bigg(\vert (x_i)_1\vert \leq \frac{\varepsilon R^2}{\sqrt{m}}\bigg) \\
    = & \P\bigg(\vert g_1 \vert \leq \frac{\varepsilon\| g\|_2}{\sqrt{m}}\bigg) \\
    = & \P\bigg(\vert g_1 \vert \leq \frac{\varepsilon \|g\|_2}{\sqrt{m}}, \|g\|_2 \leq (1+\delta)\sqrt{m}\bigg)\\
    & + \P\bigg(\vert g_1 \vert \leq \frac{\varepsilon \|g\|_2}{\sqrt{m}}, \|g\|_2 \geq (1+\delta)\sqrt{m}\bigg)\\
    \leq & \P(\vert g_1 \vert \leq \varepsilon (1+\delta)) + \exp(-\delta^2 m/2) \\
    \leq & \varepsilon (1+\delta) + \exp(-\delta^2 m/2),
\end{align*}
where the penultimate inequality follows by applying Gaussian concentration to $\| g\|_2$. Setting $\delta = \sqrt{\frac{2 \log m}{m}}$ in the above display we get that for $m \geq 2$
\begin{equation}
\label{eq:inner product anti conc -- upper bd}
    \P\bigg(\vert \langle x_i, x_j \rangle \vert \leq \frac{\varepsilon R^2}{\sqrt{m}}\bigg) \leq 2\varepsilon + \frac{1}{m}.
\end{equation}
Now, using Chernoff bound we get that for $0\leq t \leq 1$
\begin{align*}
\mathbb{P}\left(\sum_{j:j\neq i} \1\left(\vert \langle x_i, x_j \rangle \vert \geq \frac{\varepsilon R^2}{\sqrt{m}}\right) \leq (1-t)\left(1 - 2\varepsilon  -\frac{1}{m}\right)(n-1)\bigg\vert x_i\right) \leq \exp\left(-\frac{t^2}{2}\left(1-2\varepsilon - \frac{1}{m}\right)(n-1)\right).
\end{align*}

Next, by the Pigeonhole principle and the union bound we have 

\begin{align*}
\mathbb{P}\Bigg(\sum_{i,j:j\neq i} \1\left(\vert \langle x_i, x_j \rangle \vert \geq \frac{\varepsilon R^2}{\sqrt{m}}\right) \leq &(1-t)\left(1-2\varepsilon - \frac{1}{m} \Bigg)n(n-1)\right) \\
\leq  & \sum_{i=1}^n \mathbb{P}\left(\sum_{j:j\neq i} \1\left(\vert \langle x_i, x_j \rangle \vert \geq \frac{\varepsilon R^2}{\sqrt{m}}\right) \leq (1-t)\left(1-2\varepsilon -\frac{1}{m}\right)(n-1)\right)\\
\leq & n\exp\left(-\frac{t^2}{2}\left(1-2\varepsilon-\frac{1}{m}\right)(n-1)\right),
\end{align*}
and the desired result follows by setting $t = \sqrt{\frac{24 \log n}{(1-2\varepsilon - \frac{1}{m})(n-1)}}$ and noting that for $m \geq 2$ and $0 \leq \varepsilon \leq \frac{1}{8}$ we have that $ 1- 2\varepsilon - \frac{1}{m} \geq \frac{1}{4}.$
\end{proof}

\begin{proof}[Proof of Proposition \ref{prop:sphere}]
We begin by proving \eqref{eq:sphere_err_bd}.
We would like to apply Theorem \ref{thm:manifold_denoising_low_rank} with $k = m$. To do so, we must verify that $\Sigma^{\infty}$ is rank $m$ and assumptions \eqref{assu:trace_truecov}--\eqref{assu:spec_gap_truecov} hold. First, note that
$$\Sigma^{\infty} =\mathbb{E}_{\bfx \sim \Unif R\cS^{m-1}}\bigg[\bfx\bfx^\top \bigg] = \begin{pmatrix}
    \frac{R^2}{m}I_m & O_{d-m} \\
    O_{d-m} & O_{d-m}
\end{pmatrix},$$
where $I_m$ is the $m \times m$ identity matrix and $O_{d-m}$ is the $(d-m) \times (d-m)$ matrix of all 0s.
Thus, $\Sigma^{\infty}$ is rank $m$ and
$\Tr \Sigma^{\infty} =  R^2.$
Hence, \eqref{assu:trace_truecov} is satisfied by assumption that $R^2 \lesssim 1/(\sqrt{d} \log^2 n)$. Additionally, $\max_{i, j} \langle x_i, x_j \rangle \leq R^2 = \Tr \Sigma^{\infty}$. Thus, \eqref{assu:no_outliers_truecov} holds. Finally, $\lambda_m(\Sigma^{\infty}) = \frac{R^2}{m}$, thus $\lambda_m(\Sigma^{\infty}) = \omega(d^{-1})$ by assumption $R^2 = \omega(\frac{m}{d})$ and, therefore, \eqref{assu:spec_gap_truecov} is satisfied as well.
 Consequently, we may apply \Cref{thm:manifold_denoising_low_rank} to $\cX$.

Now we estimate the spectral quantities that appear in the bound we get from a direct application of Theorem \ref{thm:manifold_denoising_low_rank} to $\cX$. To this end, we relate $\lambda_1(XX^{\top})$, $\lambda_m(XX^{\top})$ and $\norm{X}_F^2$ to the corresponding spectral statistics of $n \Sigma^{\infty}$ using Lemma \ref{lem: discrete trace limit}. By Weyl's inequality with \eqref{eqn:single entry concentrates} and \eqref{eqn: curve-trace of sigma concentrates}
\begin{equation}
\label{eq:lambda1_sphere}
   \lambda_i(XX^{\top}) = \lambda_i(X^\top X) = \frac{nR^2}{m} + O\bigg(R^2\sqrt{n\log n}\bigg) =(1+o(1))\frac{nR^2}{m}, \text{ for all }1\leq i\leq m,
\end{equation}
since our assumptions imply $\omega(m) = \frac{\sqrt{d}}{\log n} < \sqrt{\frac{n}{\log n}}$. Specifically, this tells us that $\|XX^\top\|_\op=(1+o(1))nR^2/m$ and $\|XX^\top\|_F=(1+o(1))nR^2/\sqrt m$.
Now, setting $\lambda_1=\lambda_k=R^2/m$, $k=m$ and $\Tr\Sigma^{\infty} = R^2$ in \Cref{thm:manifold_denoising_low_rank} and using the fact that $\omega(m/d) = R^2$ we obtain \eqref{eq:sphere_err_bd}. The Frobenius bound is immediately obtained using the fact that $XX^\top-\Xalg\Xalg^\top$ is at most rank$-2m$.
We now prove \eqref{eq:len preserved}. By \eqref{eq:sphere_err_bd}, and \eqref{eq:lambda1_sphere} we have that for some $C>0$
\begin{equation}
    \|XX^\top-\Xalg\Xalg^\top\|_\op\leq\log^{C}n\cdot\frac{nR^2}{m} \left(\frac{\sqrt{m}}{R\sqrt{d}}+ \sqrt{d}R^2+ \frac{m}{R^2\sqrt{npd}}\right):=(\cE_\infty)_\op.
\end{equation}
Note that the constant $C$ in \eqref{eqn:inner_prod_recover} and \eqref{eqn:norm_recovery} are meant to account for the constant from the above display, so they are the same. Let
\begin{equation}\label{eqn:varepsilon_sphere}
\varepsilon_n=\log^{C/2}n\cdot(\cE_\infty)^{1/2}_{\op}\frac{\sqrt{2m}}{\sqrt{n}R}.
\end{equation}

Then set $\varepsilon_n=\varepsilon$ from \eqref{eqn:norm_recovery}. 
Combining the above display with our assumptions \eqref{assu:radii}, we get that $\varepsilon_n=o(1)$. Therefore all but $n\varepsilon_n$ indices satisfy $|\|x_i^\alg\|_2^2-R^2|<\varepsilon_nR^2$, and consequently $\|x_i^\alg\|_2=(1+O(\varepsilon_n))R$ for all but $o(n)$ indices.
We now prove \eqref{eq:inner prod preserved}. First, we use $\varepsilon=\varepsilon_n^2$  as an input to \eqref{eqn:inner_prod_recover}, the lower bound on the inner product difference in the set \eqref{eqn:inner_prod_recover} is then
\[\frac{\log^C n}{n\sqrt{\varepsilon_n^2}}\cdot\sqrt{2m}(\cE_\infty)_\op=\frac{\sqrt{m}R^2\varepsilon_n^2}{m\varepsilon_n}=\varepsilon_n\frac{R^2}{\sqrt m}.\]
Consequently, taking the complement of \eqref{eqn:inner_prod_recover} with the above lower bound, we obtain that
\begin{equation}\label{eqn:good_complement}
|\cG|:=\bigg\lvert\left\{(i,j):i\ne j, |\langle x_i^\alg,x_j^\alg\rangle-\langle x_i,x_j\rangle|<\varepsilon_n\frac{R^2}{\sqrt m}\right\}\Bigg\rvert\geq n(n-1)-\varepsilon_n^2 n^2.
\end{equation}
Now, let $\cA(\frac{\varepsilon_0 R^2}{\sqrt m})$ be defined as in \Cref{eq:good_sets_lb}, where $\varepsilon_0\in (0,\frac18]$. By \Cref{lem:inner prod lower bd}, we know that with probability at least $1-n^{-11}$,
\begin{equation}\label{eq:cA_size}
\left|\cA\left(\tfrac{\varepsilon_0 R^2}{\sqrt m}\right)\right|\geq\bigg(1-\sqrt{\frac{96 \log n}{n-1}}\bigg)\bigg(1-2\varepsilon_0-\frac1m\bigg)n(n-1).
\end{equation}
Thus, for a pair on indices in the intersection of $\cA(\frac{\varepsilon_0 R^2}{\sqrt m})$ and $\cG$,
\[
|\langle x_i^\alg,x_j^\alg\rangle-\langle x_i,x_j\rangle|<\varepsilon_n\frac{R^2}{\sqrt m}\leq\frac{\varepsilon_n}{\varepsilon_0}|\langle x_i,x_j\rangle|,
\]
so $\langle x_i^\alg,x_j^\alg\rangle=(1+o_{\varepsilon_0}(1))\langle x_i,x_j\rangle$, since $\varepsilon_0$ is fixed and $\varepsilon_n=o(1)$. Finally, combining \eqref{eqn:good_complement} and  \eqref{eq:cA_size}, we count the indices in the intersection:
\begin{equation}
\begin{aligned}
    |\cA(\tfrac{\varepsilon_0 R^2}{\sqrt m})\cap\cG|&\geq|\cA(\tfrac{\varepsilon_0 R^2}{\sqrt m})|-\big(n(n-1)-|\cG|\big)\\
    &\ge|\cA(\tfrac{\varepsilon_0 R^2}{\sqrt m})|-\varepsilon_n^2n^2\geq\bigg(1-2\varepsilon_0-\frac1m-o(1)\bigg)n(n-1),
\end{aligned}
\end{equation}
which gives us \eqref{eq:inner prod preserved}. 
\end{proof}

\subsubsection{Mixture of nested spheres}

\begin{proof}[Proof of \Cref{prop:nested_spheres}
]
    Notice that
    $$\Sigma^{\infty} =\mathbb{E}_{\bfx \sim \sum^s_{j=1}w_j\Prob_j}\bigg[\bfx\bfx^\top\bigg] = \begin{pmatrix}
    \frac{R^2}{m}I_m & O_{d-m} \\
    O_{d-m} & O_{d-m}
\end{pmatrix},$$
which is identical to the spherical setting. Thus, the conclusion follows immediately from Proposition \ref{prop:sphere}. 
\end{proof}

\begin{proof}[Proof of Proposition \ref{prop:nested_sphere_analysis}]

Let $I_j = \{i: x_i \text{ sampled from } R_j \mathcal{S}^{m-1}\}$. By Bernstein inequality, we have that with probability at least $ 1-n^{-10}$ for every $j$
$$\vert I_j \vert = nw_j\bigg(1+O\bigg(\sqrt{\frac{\log n}{nw_j(1-w_j)}}\bigg)\bigg).$$
Therefore, if $w_j\ge 1/C_0$, then $|I_j|\ge n/(2C_0)$ for some $C_0>0$. 
We apply \eqref{eqn:norm_recovery} with $k=m$, using the error bound from \Cref{prop:nested_spheres} in place of $\cE_\op(X)$ (see \Cref{rmk:sampling_recovery}). Since $\Sigma^{\infty}$ has rank $m$, we have $\Tr\Sigma_{m+1:d}^\infty=0$. 
Let $\varepsilon_n$ be as in \eqref{eqn:varepsilon_sphere}. By definition,
\begin{equation}\label{eq:nested_lb}
\frac{\log^C n}{n}\cdot 2m(\cE_\infty)_\op=\varepsilon_n^2R^2 .
\end{equation}
To be precise, let us name the implicit constants in the assumptions \eqref{assu:radii}; assume $R^2 \geq \frac{m}{d}\log^{C_1} n$, $R^2 \leq d^{-1/2}\log^{-C_1} n$, and $R^2\sqrt{npd} \ge m\log^{C_1} n$ for some $C_1 > 4C$. So our assumptions on the constants, combined with \eqref{assu:radii}, imply that for $\beta=C_1/2-2C>0$ such that $\varepsilon_n^2\lesssim \log^{-\beta} n$. Thus applying \eqref{eqn:norm_recovery}, with $\varepsilon_0/(2C_0)$ in place of $\varepsilon$, gives 
\begin{equation*}
\label{eq:markov_nested_sphere}
 \bigg\vert \bigg\{i: \vert \|x_i^\alg\|_2^2 - R_j^2 \vert \lesssim_{C_0} \frac{R^2}{\varepsilon_0 \log^\beta n}\bigg\}\bigg\vert = (1-\varepsilon_0)\vert I_j \vert.   \end{equation*} 
If $R_j \asymp R$ for all $j$, we conclude that on the event in the previous display all but $\varepsilon_0 \vert I_j \vert$ indices $i \in I_j$ satisfy
\begin{equation*}
    \| x_i^\alg \|_2 = R_j\bigg(1+O_{C_0}\bigg(\frac{1}{\varepsilon_0 \log^\beta n}\bigg)\bigg).
\end{equation*}
Moreover, since $\varepsilon_0$ is fixed, $(\varepsilon_0\log^\beta n)^{-1} = o(1)$, since the separation $\alpha \leq R_1 - R_0 = R_1 \le C_0 R$ and \eqref{eq:separation bd} yield $R(\varepsilon_0 \log^\beta n)^{-1} = o_{C_0}(\alpha) = o_{C_0}(R)$. Hence the above display implies $|\|x_i\|_2-\|x_i^\alg\|_2| = o(\alpha)$ and $= o_{C_0}(\|x_i\|_2)$, as desired.
\end{proof}

\subsubsection{Sinusoid curve}

\begin{proof}[Proof of \Cref{prop:cosines}]

Here we will use \Cref{thm:manifold_denoising_approx_diag}, as $\Sigma^{\infty}$ turns out to be a numerically low rank approximately diagonal matrix. We begin by proving \eqref{eqn:cosines_ex_error} by computing $\Sigma^{\infty}$. Using the product-to-sum trigonometric identity $\cos(a)\cos(b)= \frac{1}{2}(\cos(a-b)+ \cos(a+b))$, we obtain for $i \neq j$ (also noting that our assumption \eqref{assu:cosines-gap} implies that $\omega_i \neq \omega_j$ for $i \neq j$)
    \begin{equation}\label{eqn:integral_prod_to_sum}
    \begin{aligned}
        \Sigma^{\infty}_{ij}&= a_ia_j\int_{[0,1]}\cos(\omega_it+\phi_i )\cos(\omega_jt+\phi_j )dt\\
        &= \frac{a_ia_j}{2}\int_{[0,1]}\cos((\omega_i-\omega_j)t +\phi_i-\phi_j)dt + \frac{a_ia_j}{2}\int_{[0,1]}\cos((\omega_i+\omega_j)t +\phi_i+\phi_j)dt\\
        &=\frac{a_ia_j}{2}\left(\frac{\sin((\omega_i-\omega_j) +\phi_i-\phi_j)- \sin(\phi_i-\phi_j)}{\omega_i-\omega_j} + \frac{\sin((\omega_i+\omega_j) +\phi_i+\phi_j)- \sin(\phi_i+\phi_j)}{\omega_i+\omega_j}\right).
        \end{aligned}
    \end{equation}
    
    Similarly, the diagonal entries:
    \begin{equation}\label{eqn:sinusoid_diag_entries}
    \begin{aligned}
        \Sigma^{\infty}_{ii}&= a_i^2\int_{[0,1]}\cos^2(\omega_it+\phi_i )dt\\
        &= a_i^2\frac{1}{2}+ a_i^2\frac{\sin(2(\omega_i +\phi_i))-\sin(2\phi_i)}{4\omega_i}.
    \end{aligned}
    \end{equation}
    Notice that we can write 
    \begin{equation}\label{eqn:cosines_true_covar}
    \Sigma^{\infty}= \frac{1}{2}\diag\Bigg(a_1^2\left(1+ \frac{\sin(2(\omega_1 + \phi_1))-\sin(2\phi_1)}{2\omega_1}\right),\dots, a_d^2\left(1+\frac{\sin(2(\omega_d + \phi_d))-\sin(2\phi_d)}{2\omega_d}\right)\Bigg) + E,\end{equation} where $E$ is defined to have entries
\begin{align*}
    E_{ij}&=  \frac{a_ia_j}{2}\left(\frac{\sin((\omega_i-\omega_j) +\phi_i-\phi_j)- \sin(\phi_i-\phi_j)}{\omega_i-\omega_j} + \frac{\sin((\omega_i+\omega_j) +\phi_i+\phi_j)- \sin(\phi_i+\phi_j)}{\omega_i+\omega_j}\right) 
\end{align*}
for $i\ne j$ and $E_{ii}=0$ for $1\le i\le d$. 
Applying triangle inequality to the above display we get that,
    \[|E_{ij} |\leq \begin{cases}
        a_ia_j\left(\frac{1}{|\omega_i-\omega_j|} + \frac{1}{|\omega_i+\omega_j|}\right)\quad \text{ if }i\neq j,\\
        0\qquad\text{ if }i= j
    \end{cases}.\]
    Recall $\Omega= \min_{i\neq j}||\omega_i|-|\omega_j||$, and that by assumption $a_1=\max_i |a_i|$. By the Schur test (see \eqref{bd:Schur}), 
    \[\norm{E}_\op\leq \max_i
    \sum_j |E_{ij}|
    \lesssim 
    \frac{ a_1^2d}{\Omega}
    .\]
   
    Notice, that our assumption \eqref{assu:cosines-freq} 
    gives us $\norm{E}_\op \lesssim \tfrac{1}{d}$, satisfying the ``approximately diagonal'' condition in the statement of \Cref{thm:manifold_denoising_approx_diag}.
    
    We check the other conditions 
    of \Cref{thm:manifold_denoising_approx_diag}. For condition \eqref{assu:trace approx diag} to be satisfied, we note that is deduced from  the minimum assumption on the frequencies \eqref{assu:cosines-freq} :
    \begin{equation}\label{eqn:sinusoid_tr_UB}
            \Tr\Sigma^{\infty}= \sum_{i=1}^d \frac{1}{2}\Bigg(a_i^2\bigg(1+ \frac{\sin(2(\omega_i+\phi_i))-\sin(2\phi_i)}{2\omega_i}\bigg)\Bigg)
            \lesssim  \sum_{i=1}^d a_i^2\left(1+ \frac{1}{|\omega_i|}\right)
    \lesssim \sum_{i=1}^d a_i^2.
    \end{equation}
    The last step is obtained by \eqref{assu:cosines-freq}, since that assumption implies that $1+1/|\omega_i|\lesssim 1$. Note that the same steps can be repeated, by using the lower bound on $\sin(y)\geq -1$, to obtain $\Tr\Sigma^{\infty}\gtrsim \sum_{i=1}^d a_i^2\left(1- \frac{1}{|\omega_i|}\right)\gtrsim\sum_{i=1}^d a_i^2$. \eqref{eqn:sinusoid_tr_UB} and the lower bound computation imply \begin{equation}\label{eqn:sinusoid_tr_magnitude}\Tr\Sigma^{\infty} =\Theta\left(\sum_{i=1}^d a_i^2\right).\end{equation}
    Combining \eqref{eqn:sinusoid_tr_UB} with \eqref{assu:cosines-amplitude}, gives us the upper bound assumptions on $\Tr\Sigma^{\infty}$--\eqref{assu:trace approx diag}.

    Condition \eqref{assu:no spikes approx diag} 
    is now trivially satisfied, by $-1\leq \cos(y)\leq 1$ and \eqref{eqn:sinusoid_tr_magnitude}:
    $$\norm{\bfx }_2^2 =  \sum_{i=1}^d a_i^2 \cos^2(\omega_i T+\phi_i) \leq \sum_{i=1}^d a_i^2 \lesssim \Tr\Sigma^{\infty}.$$
    Condition \eqref{assu:spec gap assum approx diag} follows from our assumptions $k=o(\sqrt{d}/\log^2n)$ and \eqref{assu:cosines-gap}. Then, the assumption on the tail \eqref{assu:spec gap assum approx diag}, can be inferred by the fact that $\Tr E=0$ and \eqref{assu:cosines-tail}  Therefore, \Cref{thm:manifold_denoising_approx_diag} applies and the inequality \eqref{eqn:cosines_ex_error} follows. We defer the proof of \eqref{eqn:norms-recovery-sinusoid-eq} below, as we will need \Cref{lemma:sinusoid-norms}.
\end{proof}

    \begin{remark}
        One could also sample frequencies in the window of length $n^{10}d\Omega:=\Omega'$
        \[\omega_i\sim \Unif\big[1,\Omega' 
        \big],\]
         where recall $\Omega$ is the minimum frequency gap, of the same order as specified in \eqref{assu:cosines-freq}.
Using an elementary order statistics computation (see \Cref{lemma: random frequency gap} for more details), we can find the probability of the smallest gap explicitly. For $t(d-1)\leq d\cdot\Omega$, we have
\[\P\left(\min_{i< j}|\omega_i-\omega_j|\geq t\right)= \left(1-\frac{(d-1)t}{\Omega'}\right)^d.\]
In our example, we require the gap to be of order greater than $\Omega\gtrsim a^2d^2$, setting $t=\Omega$ gives us that with probability at least $1-n^{-10}$,
\[\min_{i< j}|\omega_i-\omega_j|\geq \Omega=\omega\left(a^2d^2\right).\]
    \end{remark}\begin{lemma}\label{lemma: random frequency gap}
    If $\omega_1,\dots\omega_d\stackrel{iid}{\sim}\Unif{[a,b]}$ then
    \[\P\left(\min_{i\neq j}|\omega_i-\omega_j|\geq t\right)=\begin{cases}
        \left(\frac{b-a-t(d-1)}{b-a}\right)^d\quad \text{ for   }t(d-1)\leq b-a,\\
        0\quad \text{otherwise}.
    \end{cases}\]
\end{lemma}
\begin{proof}
    Let $\omega_{(i)}$ be the $(i)$-th order statistic. Since $\omega_i$'s are identically distributed, we know that the joint probability density of $\omega_{(1)},\dots \omega_{(d)}$ is $d!\frac{1}{(b-a)^d}$. To compute our desired probability, it suffices to evaluate the following integral:
    \begin{align*}
    \P\left(\min_{i\neq j}|\omega_i-\omega_j|\geq t\right) &= \P\left(\omega_{(1)}<\omega_{(2)}-t,\dots , \omega_{(d-1)}<\omega_{(d)}-t \right)\\
    &= \frac{d!}{(b-a)^d}\int_{a+(d-1)t}^{b-t}\dots \int_{a+t}^{x_3-t}\int_{a}^{x_2-t}1 dx_1dx_2\dots dx_d
    \end{align*}
    Making the change of variables $y_i=x_i-(i-1)t$ gives the result.
\end{proof}

\begin{lemma}\label{lemma:sinusoid-norms}
    Let $\cX$ be as in \Cref{prop:cosines}, satisfying
\eqref{assu:cosines-amplitude}--\eqref{assu:cosines-freq}, and assume additionally that \eqref{assu:cosines-wellcond}.
   Then, on the same event on which \eqref{eqn:cosines_ex_error} and \eqref{eqn:cosines_ex_error_frob} hold, we have that for some constants $c,c_1>0$, with probability at least $1-O(n^{-10})$
   \begin{equation}
    \Big|\Big\{i\in[n]: \|x_i\|_2^2\geq c\sum_{l=1}^d a_l^2\Big\}\Big|\geq \bigg(1-\sqrt{\frac{44\log n}{n}}\bigg)\bigg(1-\frac{c_1}{k}\bigg)n.
   \end{equation}
\end{lemma}

\begin{proof}[Proof of \Cref{lemma:sinusoid-norms}]
    The proof will rely on an application of a Chernoff bound on the indicators of large enough $\|x_i\|_2$'s. To this end, we first estimate the probability of $\P(\|x_i\|_2^2\leq O(\sum_{l=1}^d a_i^2))$, by Tchebyshev's inequality, thus we estimate the variance of $\|x_i\|_2=\|(a_j\cos(\omega_jT_i+ \phi_j))_{j=1}^d\|_2$, for $T_i\sim \Unif{[0,1]}$. \eqref{assu:cosines-wellcond} and the simplification $a_1\geq \dots\geq a_d\geq 0$ imply that 
    \begin{equation}\label{eqn:cosines-ka1}
    \sum_{l=1}^d a_l^2\geq\sum_{j\leq k}a_j^2\geq k a_k^2\gtrsim k a_1^2,
    \qquad\text{thus }k=\omega(1)\text{ gives }\qquad a_1^2=o\Big(\sum_{l=1}^d a_l^2\Big).
\end{equation}
By the double angle identity
$\cos^2(\theta)=\frac12(1+\cos(2\theta))$, we can find an explicit expression for the norm 
\begin{equation}\label{eqn:cosines-norm-decomp}
    \|x_i\|_2^2=\frac12\sum_{j=1}^d a_j^2+\frac12\sum_{j=1}^d a_j^2\cos(2\omega_jT_i+2\phi_j),
\end{equation}
and therefore for any $1\leq i\leq n$, \begin{align*}\Var\big(\|x_i\|_2^2\big)&\leq\frac{1}{4}\E_{T_i\sim \Unif{[0,1]}}\bigg[\Big(\sum_{j=1}^d
    a_j^2\cos(2\omega_jT_i+2\phi_j)\Big)^2\bigg]\\
    &=\;\frac14\sum_{j,l=1}^d a_j^2a_l^2\int_{[0,1]}\cos(2\omega_jt+2\phi_j)\cos(2\omega_lt+2\phi_l)dt.
\end{align*}
Following the same computation as in \eqref{eqn:integral_prod_to_sum} yields for $j\neq l$
\[\bigg|\int_{[0,1]}\cos(2\omega_jt+2\phi_j)\cos(2\omega_lt+2\phi_l)\,dt\bigg|
    \;\leq\;\frac{1}{2|\omega_j-\omega_l|}+\frac{1}{2|\omega_j+\omega_l|}\;\leq\;\frac{1}{\Omega}.\]
For $j=l$, the integral is $\leq1$. Then using
\eqref{assu:cosines-freq} and $\sum_{j=1}^d a_j^2\leq da_1^2$,
\begin{equation}\label{eqn:sinusoid-variance}
    \Var\left(\|x_i\|_2^2\right)\lesssim\sum_{j=1}^d a_j^4+\frac{1}{\Omega}\left(\sum_{j=1}^d a_j^2\right)^2
    \leq a_1^2\sum_{j=1}^d a_j^2+\frac{1}{a_1^2d^2}\left(\sum_{j=1}^d a_j^2\right)^2
    \leq a_1^2\sum_{j=1}^d a_j^2+\frac{1}{d}\sum_{j=1}^d a_j^2\lesssim  a_1^2\sum_{j=1}^d a_j^2.
\end{equation}
The last inequality was obtained by using \eqref{assu:cosines-gap} to conclude that $a_1^2\geq a_k^2=\omega(1/d)$.
Then, applying \eqref{eqn:cosines-ka1} to the above, yields for all $1\leq i\leq n$
\begin{equation}\label{eqn:sinusoid_variance_final}
    \Var\big(\|x_i\|_2^2\big)\lesssim\frac1k\Big(\sum_{l=1}^d a_l^2\Big)^2.
\end{equation}
Due to \eqref{eqn:cosines-norm-decomp}, we know that
$\Tr\Sigma^{\infty}=\Theta\big(\sum_{i=1}^d a_i^2\big)$. This implies that there exists a constant $c_0>0$ (not dependent on the index $i$) such that
$\E\|x_i\|_2^2=\Tr\Sigma^{\infty}\geq c_0\sum_{l=1}^d a_l^2$). Then Tchebyshev's inequality combined with \eqref{eqn:sinusoid_variance_final} gives, for some $c_1>0$
\begin{equation}\label{eqn:tchebyshev}
    q:=\P\left(\|x_i\|_2^2\geq\frac{c_0}{2}\sum_{l=1}^d a_l^2\right)
    \geq1-\frac{\Var\left(\|x_i\|_2^2\right)}{\left(\frac{c_0}{2}\sum_{l=1}^d a_l^2\right)^2}
    \geq1-\frac{c_1}{k}=1-o(1),
\end{equation}
where the last inequality was obtained by \eqref{assu:cosines-wellcond}. By independence of $T_1,\dots,T_n$, we know that $\1\big(\|x_i\|_2^2< \frac{c_0}{2} \sum_{l=1}^d a_l^2\big)\sim \text{Bern}(q)$. Thus, by a Chernoff bound, we have that
\begin{equation}\label{eqn:chernoff}
    \P\left(\sum_{i=1}^n\1\left(\|x_i\|_2^2\geq\frac{c_0}{2}\sum_{l=1}^d a_l^2\right)
    \leq(1-t)nq\right)\leq\exp\left(-\frac{t^2}{2}nq\right).
\end{equation}
\eqref{eqn:tchebyshev} gives that for large enough $n$, $q\geq 1/2$, so setting $t=\sqrt{22\log n/(nq)}$, implies that the right-hand side of \eqref{eqn:chernoff} is 
$\leq n^{-11}$, so with probability at least $1-n^{-11}$,
\[\bigg|\bigg\{i\in[n]:\|x_i\|_2^2\geq\frac{c_0}{2}\sum_{l=1}^d a_l^2\bigg\}\bigg|
    \geq(1-t)nq\geq\bigg(1-\sqrt{\frac{44\log n}{n}}\bigg)\bigg(1-\frac{c_1}{k}\bigg)n,\]
where in the last step we used \eqref{eqn:tchebyshev}, \eqref{assu:cosines-wellcond} and
$\log n=o(n)$. Taking $c=c_0/2$ concludes the proof.
\end{proof}

\begin{proof}[Proof of \eqref{eqn:norms-recovery-sinusoid-eq}]
\label{proof:sinusoid-norms}
    Throughout this proof, we will condition on the intersection of the events on which  \Cref{eqn:cosines_ex_error_frob} and \Cref{lemma:sinusoid-norms} hold,
which by a union bound, occurs with probability at least $1-O(n^{-10})$.  In particular, by \Cref{lemma:sinusoid-norms}, all but $o(n)$ of the points satisfy
$\|x_i\|_2^2\geq c\sum_{j=1}^d a_j^2$, so by
\Cref{cor:relative_recovery} with $\varepsilon=\delta$ (for simplicity) and 
$t=\big(c\sum_{j=1}^d a_j^2\big)^{1/2}$, it suffices to show that
\begin{equation}\label{eqn:cosines-ratio}
  \frac{2k\|XX^\top-\Xalg\Xalg^\top\|_\op+n\Tr\Sigma_{k+1:d}}{\sum_{j=1}^d na_j^2}=o(1).
\end{equation}
\eqref{eqn:norms-recovery-sinusoid-eq} then holds for all but $o(n)$ of the indices with
$\|x_i\|_2^2\geq c\sum_j a_j^2$, hence for all but $o(n)$ indices in $[n]$. For the first summand being $o(1)$ is an immediate consequence of \eqref{assu:sinusoid_density_amp}, \eqref{assu:cosines-wellcond} and \eqref{eqn:cosines-ka1}. As for the second summand, \eqref{eqn: tail_trace_concentrates} implies that with probability at least $1-O(n^{-11})$
\begin{equation}\label{eq:sinsuoid_tail_conc}
    \Tr\Sigma_{k+1:d}\leq \sqrt{\Tr\Sigma_{k+1:d}^\infty\Tr\Sigma^{\infty}}\sqrt{\frac{\log n}{n}} + \Tr\Sigma_{k+1:d}^\infty
\end{equation}
The Horn--Schur theorem applied to \eqref{eqn:cosines_true_covar}, together with $\Sigma^{\infty}_{jj}=\tfrac12 a_j^2\big(1+O(|\omega_j|^{-1})\big)
\lesssim a_j^2$ and $\Tr\Sigma^{\infty}=\Theta\big(\sum_{j=1}^d a_j^2\big)$ from
\eqref{eqn:sinusoid_tr_UB}, implies that $\Tr\Sigma_{k+1:d}^\infty\lesssim\sum_{j>k}a_j^2$. Plugging the former into \eqref{eq:sinsuoid_tail_conc}, and dividing by $\sum_j^d a_j^2$ gives 
\[
  \frac{\Tr\Sigma_{k+1:d}}{\sum_{j=1}^d a_j^2}
  \lesssim\frac{\sum_{j>k}a_j^2}{\sum_{j=1}^d a_j^2}
  +\sqrt{\frac{\sum_{j>k}a_j^2}{\sum_{j=1}^d a_j^2}\frac{\log n}{n}} =o(1/k)+o(1) .
\]
The last equality was obtained by \eqref{assu:cosines-gap}, and \eqref{assu:cosines-tail}. 
\end{proof}

\begin{prop}\label{prop:cosines-inner}
    Under the same assumptions as \Cref{prop:cosines}, on the same event for which \eqref{eqn:cosines_ex_error_frob} holds, and for every $t=\omega\big(n^{-1}(\cE_F(X))_\infty\big)$,
\begin{equation}
  \bigl|\bigl\{(i,j)\in\cA(t): \langle x_i^\alg,x_j^\alg\rangle
  =(1+o(1))\langle x_i,x_j\rangle\bigr\}\bigr|\geq|\cA(t)|-o(n^2),
\end{equation}
where $(\cE_F)_\infty$ denotes the right hand side of \eqref{eqn:cosines_ex_error_frob} (by substituting \eqref{eqn:cosines_ex_error} into \eqref{eqn:cosines_ex_error_frob}), and $\cA$ is defined as in \eqref{eq:good_sets_lb}.
\end{prop}

\begin{remark}
    Note that \eqref{prop:cosines-inner} is only informative if $|\cA(t)|$ is a constant fraction of $n^2$.
\end{remark}
\begin{proof}[Proof of \Cref{prop:cosines-inner}]
We condition on the event of \Cref{prop:cosines}, and refer to the right hand side of \eqref{eqn:cosines_ex_error_frob} as $(\cE_F)_\infty$. In the setting of \Cref{cor:relative_recovery}, set $\varepsilon=\delta=((\cE_F)_\infty n^{-1}t^{-1})^{2/3}$ (which is $o(1)$ by \eqref{assu:cosines-amplitude}--\eqref{assu:cosines-freq}, \eqref{assu:cosines-wellcond}, and \eqref{assu:sinusoid_density_amp}), so that $t=(\cE_F)_\infty n^{-1}\varepsilon^{-1/2}\delta^{-1}=o(1)$ as in \eqref{eq:inner_prod_recovery_relative}-- with $\cE_F$ in place of $(\cE_F)_\infty$. Then, the proof follows from \Cref{rmk:sampling_recovery}.
\end{proof}
\appendix
\section{Parameter Estimation Proofs and Asymptotics}
\label{Appendix Section:parameter_estimation_proofs}
With the exception of the proof of \Cref{lem:p1} in \Cref{sec:p1}, the
parameter estimates below are stated and proved for the truncated (strictly rank-$k$) model,  i.e.  in terms of the truncated adjacency matrix $\tilde A$ and the associated edge density $\tilde p$. This entails no loss of generality for the numerically low-rank setting of \Cref{thm:main}. In the strictly rank-$k$ case $\Tr\Sigma_{k+1:d}=0$, so $\tilde A=A$ and $\tilde p=p$. In the numerically low-rank case the tail hypothesis $\Tr\Sigma_{k+1:d}=o\left( \min(p^3, d^{-1})\right)$ guarantees, via \Cref{lem:ptp-same}, that $p=(1+o(1))\tilde p$, so every estimate proved here for $\tilde p$ holds for $p$ up to a $(1+o(1))$ factor and the leading-order asymptotics are unchanged.

\subsection[Proof of \Cref{lemma:ai_aj_li_lj_conc}]{Proof of \Cref{lemma:ai_aj_li_lj_conc}}
\label{proof of lemma:ai_aj_li_lj_conc}

Recall the decomposition of the noisy vectors $\tilde y_i = (a_i, w_i)$ where $a_i \in \R^k$ represents the signal plus noise, and $w_i \in \R^{d_k}$ is the noise. We further decompose the signal vector as $a_i =\Pi^{(k)}_{X}x_i + N_i =\tilde x_i + N_i$, where $x_i$ is the fixed point cloud vector and $N_i \sim \mathcal{N}(0, \frac{1}{d}I_k)$. The inner product expands as:
\begin{equation}
\label{eq: aiaj inner prod}
    \ip{a_i}{a_j} = \ip{\tilde x_i}{\tilde x_j} + \ip{\tilde x_i}{N_j} + \ip{N_i}{\tilde x_j} + \ip{N_i}{N_j}.
\end{equation}
Using Bernstein's inequality (e.g. see \cite[Theorem 2.8.1]{Vershynin2018}), and the fact that for each entry of the vector $N_i\in \R^k$, we have that $\norm{(N_i)_l}_{\psi_2}=O(1/\sqrt{d})$, 
\[
\mathbb{P}\left(
{\exists (i,j)\in [n]^2, i\neq j:\langle N_i,N_j\rangle\geq t} \right)\leq e^{2\log n-c_1\min(td, t^2d^2/k)}.
\]

We require $2\log n< c_1\min(td,t^2d^2/k)$, so if $k<\log n$, set $t=C_1\frac{\log n}{d}$ , or if $k>\log n$ then $t=C_1\frac{\sqrt{k\log n}}{d}$ for some $C_1>0$. Thus, it follows that $\ip{N_i}{N_j} \leq C_1\frac{\sqrt{\max\{k,\log n\}\log n}}{d}$ with probability at least $1-n^{-12}$, for $C_1$ sufficiently large 
uniformly over all pairs $(i, j)\in [n]^2, i\neq j$. 

By \eqref{assu:no-outliers} and Cauchy-Schwartz, 
\[
|\ip{\tilde x_i}{\tilde x_j}| \leq \|\tilde x_i\|_2 \|\tilde x_j\|_2\leq \|x_i\|_2\|x_j\|_2 \leq \Tr\Sigma.
\]
When $i \neq j$, the cross-term $\ip{\tilde x_i}{N_j}$ is a mean zero, Gaussian variable with variance $\norm{\tilde x_i}^2/d \lesssim  \Tr\left(\Sigma\right)/d$. Therefore, with probability at least $1-n^{-12}$, $|\ip{\tilde x_i}{N_j}| \lesssim \sqrt{\Tr\left(\Sigma\right)\log n/d} $ and thus
\begin{equation}\label{eqn:ai_inner_product}
\ip{a_i}{a_j}=O\left(\Tr\left(\Sigma\right)+\sqrt{\Tr\left(\Sigma\right) \frac{\log n}{d}}+\frac{\sqrt{\max\{k,\log n\}
\log n}}{d}\right).
\end{equation}
While it is tempting to bound the right hand side by the simpler expression $O(\Tr \Sigma\cdot \log n) $ (using \eqref{assu:gap-gen}), in some of the arguments below we cannot afford to lose a logarithmic factor.

On the other hand, for the $i=j$ case, the expression \eqref{eq: aiaj inner prod} becomes
\[\norm{a_i}_2^2 = \norm{\tilde x_i}_2^2 + 2\langle \tilde x_i, N_i\rangle + \norm{N_i}_2^2.\]
By  \cite[Theorem 3.1.1]{Vershynin2018}, we know that with probability at least $1-n^{-12}$,
\begin{equation}\label{eqn:norm_of_Nij}
\norm{N_i}_2^2=O\left(\frac{k}{d}+ \frac{\sqrt{\max\{k,\log n\}\log n}}{d}\right).
\end{equation}
Applying the condition \eqref{assu:no-outliers} as in the case $i\neq j$, we obtain
\begin{equation}\label{eqn:ai_norm}
\norm{a_i}_2^2 = O\left(\Tr\Sigma + \sqrt{\Tr\Sigma\frac{\log n}{d}} +\frac{k}{d}+ \frac{\sqrt{\max\{k,\log n\}}}{d}\right)\leq O(\Tr\Sigma \log n).\end{equation}
For the lengths $l_i = \norm{w_i}_2$, note that $d \norm{w_i}_2^2 \sim \chi^2_{d_k}$. $\norm{\chi_{d_k}^2-d_k}_{\psi_1}=\sqrt{d_k}$, therefore by Bernstein's inequality,

\[\Prob\left(\exists i\in [n]:|l_i^2-d_k/d|>t\right)\leq e^{\log n-c\min\{\frac{t^2d^2}{d_k},\frac{td}{d_k}\}}.\]
Setting $t=C_3\sqrt{\frac{\log n}{d}}$ for $C_3>0$ large enough gives
\begin{equation}\label{eqn: li_concentration}
l_i=1+O\left(\frac{\sqrt{\log n}}{\sqrt{d}}\right).
\end{equation}
Thus, the product becomes \begin{equation}\label{eqn:li_lj_prod}
l_i l_j = 1 + O(\sqrt{\log n}/\sqrt{d}).
\end{equation}
\qed

\subsection{Proof of \Cref{lemma:conc_tau}}
\label{proof of lemma:conc_tau}

Let $\psi_X(t):= \frac1{\binom{n}{2}}\sum_{i< j} \P(\langle y_i,y_j\rangle \ge t/\sqrt{d})$, and let $\tau(p)=\tau_X(p)$ be the unique value of $\tau$ such that $p= \psi_X(\tau\sqrt{d})$.

\begin{lemma}\label{lemma:tau_gaussian_quant}
Let $\Psi(t):= \P( g\ge t)$ denote the complementary distribution function for a standard Gaussian variable on $\R$.
Under the hypotheses of \Cref{thm:main}, 
\begin{align*}
    \sqrt{d}\tau_X(p) 
&= \Psi^{-1}(p) + O\bigg( \sqrt{d}\Tr \Sigma + (\Tr \Sigma)^{1/2}\log n + \frac{\log^2n}{\sqrt{d}} 
\bigg)\\
&=\Psi^{-1}(p) + \wt O( \sqrt{d}\Tr \Sigma 
)\,.
\end{align*}
\end{lemma}

Note that $\Psi^{-1}(p)= \sqrt{2\log(1/p)} ( 1+ O( \frac1{\sqrt{\log (1/p)}}))$.

\begin{proof}
    Set $\cE:=\sqrt{d}\Tr \Sigma + (\Tr \Sigma)^{1/2}\log n + \frac{\log^2n}{\sqrt{d}}$.
    It will suffice to show that for all $i\ne j$ and $T\in [\frac1{100}, \log n]$, 
    \begin{equation}
        \label{taup-goal1}
        \P( \langle y_i, y_j\rangle \ge T/\sqrt{d}) = \Psi( T + O( \cE)) + O( e^{-c\log^2n})\,.
    \end{equation}
    Indeed, we have that $\psi_X$ is strictly decreasing, and from \eqref{taup-goal1},
    \[
    \psi_X(T) = \Psi( T+ O(\cE)) 
    \]
    for $T\in[\frac1{100}, \log n]$.
    From this and the monotonicity of $\psi_X$ and $\Psi$ (noting that $\Psi(\log n)< p < \Psi(\frac1{100})$) we verify $\sqrt{d}\tau_X(p) = \Psi^{-1}(p)  + O(\cE
    )$, as desired.
    
    It remains to establish \eqref{taup-goal1}. By relabeling indices we may assume $(i,j)=(1,2)$.
    Letting $g:= \langle z_1/\|z_1\|_2,z_2\rangle$, we have
    \[
    \langle y_1,y_2\rangle = \langle x_1, x_2\rangle + \langle x_1, z_2\rangle + \langle x_2, z_1\rangle + \|z_1\|_2 g. 
    \]
    From Cauchy--Schwarz and \eqref{assu:no-outliers} it follows 
    \[
    |\langle x_1,x_2\rangle| \le \|x_1\|_2\|x_2\|_2 \ls \Tr \Sigma\,.
    \]
    Next note that $\langle x_1,z_2\rangle\sim N(0, \|x_1\|_2^2/d)$, so again from \eqref{assu:no-outliers} it follows 
    \[
    \langle x_1,z_2\rangle \ls d^{-1/2}(\Tr \Sigma)^{1/2} \log n
    \]
    with probability $1- O(e^{-c\log^2n})$. 
    The same holds for $\langle x_2, z_1\rangle$. 
    Finally, since $d\|z_1\|_2^2$ is a sum of independent uniformly sub-exponential variables of expectation 1, we find
    \[
    \|z_1\|_2 = 1+ O( \frac{\log n}{\sqrt{d}})
    \]
    with probability $1-O(e^{-c\log^2 n})$.
    Combining these estimates, we have that 
    \[
    \langle y_1,y_2\rangle = ( 1+ O( \frac{\log n}{\sqrt{d}})) g + O( \Tr \Sigma + d^{-1/2}(\Tr \Sigma)^{1/2}\log n )
    \]
    with probability $1-O(e^{-c\log^2 n})$, from which \eqref{taup-goal1} follows. 
\end{proof}

\subsection{Proof of Lemma \ref{lemma:p_tauij_conc}}

\label{proof of lemma:p_tauij_conc}

Before commencing with the proof we recall some facts related to the Gegenbauer polynomial basis.

We will use the random variable $\xi\sim \mathcal{D}_d$, to denote the distribution of $\sqrt{d}\langle u, w \rangle$ where $u$ and $w$ are independent vectors distributed uniformly on the unit sphere $\mathcal{S}^{d-1}$. The scaling factor $\sqrt{d}$ ensures that the second moment is normalized $\mathbb{E}[\xi^2] = 1$. The probability density function of $\xi$ is given by
\begin{equation}\label{eqn:spherical_density}
f_d(\xi) = \frac{\Gamma(d/2)}{\sqrt{d\pi}\Gamma((d-1)/2)} \left(1 - \frac{\xi^2}{d}\right)^{\frac{d-3}{2}}.
\end{equation}
This density function serves as the weight function for the Gegenbauer polynomials $\{q_k^{(d)}\}_{k \in \mathbb{N}}$, which form an orthonormal basis for the space $L^2([-\sqrt{d}, \sqrt{d}], \mathcal{D}_d)$ such that \[
\mathbb{E}_{\xi \sim \mathcal{D}_d}[q_k^{(d)}(\xi)q_l^{(d)}(\xi)] = \delta_{kl}.
\]
Note that they can be obtained by applying the Gram--Schmidt orthogonalization process to the monomial basis.
The Gegenbauer polynomials can be expressed as a linear combination of 
spherical harmonics,  which results in the following identity:
for $u,w,v$ independent and uniform in $ \mathcal{S}^{d-1}$,
\[
\E_v[q_l^{(d)}(\sqrt{d}\langle u,v\rangle )q_j^{(d)}(\sqrt{d}\langle v,w\rangle)]=\frac{1}{\sqrt{N_l^{(d)}}}
1_{l=j}
q_l^{(d)}(\sqrt{d}\langle u,w\rangle).
\]
Moreover, from orthonormality, we can obtain that
\[
q_l^{(d)}(\sqrt{d})=\sqrt{N_l^{(d)}}.
\]

\begin{proof}

We first estimate the size of $\tau^{ij}$ using \Cref{lemma:conc_tau}.
Recall that the shifted threshold is defined as $\tau^{ij} = (\tau - \ip{a_i}{a_j})/(l_i l_j)$. Here however, to be more accurate, we use the full estimates on $\langle a_i,a_j\rangle$ without adding extra logarithmic factors, as in the proof of \Cref{lemma:p_tauij_conc}. Plugging in the first expression \eqref{eqn:ai_inner_product}, and \eqref{eqn:li_lj_prod}.
Plugging these estimates into the expression for $\tau^{ij}$ yields
\begin{equation}\label{eqn:tau_ij}
\begin{aligned}
    \tau^{ij} &= \left(\tau - O\left(\Tr\left(\Sigma\right)+\sqrt{\Tr\left(\Sigma\right) \frac{\log n}{d}}+\frac{\sqrt{\max\{k, \log n\}
\log n}}{d}\right)\right)\left(1 + O\left(\sqrt{\frac{\log n}{d}}\right)\right)^{-1} \\
    &= \left(\tau - O\left(\Tr\left(\Sigma\right)+\sqrt{\Tr\left(\Sigma\right) \frac{\log n}{d}}+\frac{\sqrt{\max\{k, \log n\} 
\log n}}{d}\right)\right)\left(1-O\left(\sqrt{\frac{\log n}{d}}\right)\right).
\end{aligned}
\end{equation}
The last line was obtained by a Taylor expansion. Then, using \eqref{assu:trace} and \Cref{lemma:conc_tau}, the dominant term is $\tau^{ij} = \tau + O(\frac{1}{\sqrt{d}\log n})$. Note that the same computation works for $i=j$ since both $\langle a_i,a_j\rangle =O(\Tr\Sigma \log n)$ and $\norm{a_i}_2^2=O(\Tr\Sigma \log n)$ from \Cref{lemma:ai_aj_li_lj_conc}.

Recall that $p_0^{ij}$ and $p_1^{ij}$ correspond to the 0-th and 1-st Gegenbauer coefficients of the expansion of the kernel function $\1(\sqrt{d_k}\langle v_i,v_j\rangle\geq \tau^{ij})$, since $\tau^{ij}$ is of the same order of magnitude as in \cite{LiSchramm2024}, the proof for $p_0, p_1 p_0^{ij}$ and $p_1^{ij}$ follows verbatim as in their Lemma 4.9 \cite{LiSchramm2024}, which is a direct computation.
However, for the benefit of subsequent proofs, we outline explicit expressions for $p_1$ below:

\begin{align}
    p_1 &= \frac{1}{\sqrt{d_k}} \E_{\xi \sim \mathcal{D}_{d_k}}[\xi \mathbf{1}(\xi \ge \sqrt{d_k}\tau)] = \frac{1}{\sqrt{d_k}} \int_{\sqrt{d_k}\tau}^{\sqrt{d_k}} \xi f(\xi) d\xi,  \label{formula-p1}
\end{align}
where $f=f_{d_k}$ is the density function of $\xi \sim \mathcal{D}_{d_k}$ (see \eqref{eqn:spherical_density}).
Note that we can evaluate the integral for $p_1$ explicitly:
\begin{align}
    p_1&=\frac{1}{\sqrt{d_k}}\E_{\xi\sim \mathcal{D}_{d_k}}[\xi\1(\xi\geq \sqrt{d_k}\tau)]=\frac{\Gamma(d_k/2)}{d_k\sqrt{\pi}\Gamma((d_k-1)/2)}\int_{\sqrt{d_k}\tau}^{\sqrt{d_k}} \left(1 - \frac{\xi^2}{d_k}\right)^{\frac{d_k-3}{2}}\xi d\xi\notag \\
    &= \frac{\Gamma(d_k/2)}{d_k\sqrt{\pi}\Gamma((d_k-1)/2)}\frac{d_k}{d_k-1}(1-\xi^2/d_k)^{\frac{d_k-1}{2}}\Bigg|_{\sqrt{d_k}\tau}^{\sqrt{d_k}} = \frac{\Gamma(d_k/2)}{(d_k-1)\sqrt{\pi}\Gamma((d_k-1)/2)}(1-\tau^2)^{\frac{d_k-1}{2}}.\label{eqn:p1_computation}
\end{align}
\end{proof}

\subsection[Proof of Lemma \ref{lemma:scalar_estimates_whp}]{Proof of Lemma \ref{lemma:scalar_estimates_whp}}
\label{proof of lemma:scalar_estimates_whp}

\begin{proof}[ Proof of \Cref{lemma:scalar_estimates_whp}]
Throughout this proof, we will condition on the event where the estimates on $l_i, \langle a_i,a_j\rangle$ in \Cref{lemma:ai_aj_li_lj_conc} hold. Recall that these events hold with probability at least $1-n^{-12}$ over all pairs $(i,j)\in [n]^2$, including the case when $i = j$. 

We begin the proof by estimating quantities that do not change regardless of the relationship between $\Tr\Sigma$ and $\tau^2$ ({we call such estimates ``regime-free''}). We already know \Cref{lemma:ai_aj_li_lj_conc} that $l_i=1+O(\sqrt{\log(n)/d})$ with probability at least $1-n^{-12}$, thus \begin{equation}\label{eqn:Li_estimate}
L_i = l_i - 1=O(\sqrt{\log(n)/d})\end{equation} on the same event. Then using the geometric series expansion $(1+x)^{-1} = 1-x+O(x^2)$ for small $x$, the inverse length is:
\begin{equation}
\label{eq:inverse len}
    l_i^{-1}= (1+L_i)^{-1}=1 - L_i+O(L_i^2)=1- L_i + O\left(\frac{\log n}{d}\right).
\end{equation}
Next, we consider $\ip{v_i}{v_j}$ for $i \neq j$. Note that the noise vectors $v_i=w_i/l_i$ are uniformly distributed on $\mathcal{S}^{d_k-1}$. Due to rotational invariance of the spherical distribution, $\langle v_i,v_j\rangle\stackrel{dist.}{=}\langle e_1,v_j\rangle$. We want to apply Levy's Lemma \cite[Theorem 5.1.4]{Vershynin2018}, so we define $g:\cS^{d_k-1}\to\R$ as $g(v)=\langle v,e_1\rangle$. It is easy to check that $g$ is $1-$Lipschitz, and that $\E g(v_j)=0$. Thus, Levy's Lemma and rotational invariance tell us that:
\begin{equation}
    \Prob\left(|\langle v_i,v_j\rangle |\geq t\right)=\Prob\left(|\langle v_j,e_1\rangle |\geq t\right)\leq \exp(-Cd_kt^2), \quad\text{for some } C>0.
\end{equation}
Setting $t=C'\sqrt{\log n/d})=o(\log n/\sqrt{d})$ for a sufficiently large constant $C'>0$, gives us that $\langle v_i,v_j\rangle=o(\log n/\sqrt{d})$ with probability at least $1-n^{-12}$. Then taking a union bound over all the pairs $(i,j)\in [n]^2$, gives us that $\langle v_i,v_j\rangle=o(\log n/\sqrt{d})$ uniformly with probability at least $1-n^{-10}$. \\

We will now find the regime-free approximation $\Delta p_0^{ij} = -p_1 d_k \Delta \tau^{ij}(1+O(\Delta \tau^{ij}\tau d))$, and consequently derive the expression $\Delta p_1^{ij} = \tau \Delta p_0^{ij}(1+o(1))$. However, in order to do this, we first have to estimate $\Delta\tau^{ij}$. Using the definition of $\Delta \tau^{ij}$, \eqref{eqn:tau_ij}, and \eqref{eqn:Li_estimate}:
\begin{equation}\label{eqn:general_Delta_tau}
\begin{aligned}
\Delta\tau^{ij}&= -\langle a_i,a_j\rangle (1+ (L_i+L_j)) -\tau(L_i+L_j) + \tau O\left(\frac{\log n}{d}\right)\\
    &= -\langle a_i,a_j\rangle (1+o(1)) -\tau\left((L_i+L_j) + O\left(\frac{\log n}{d}\right)\right).
\end{aligned}
\end{equation}
We can always uniformly bound $\tau d |\Delta\tau^{ij}|$ by applying \eqref{eqn:ai_inner_product}, \Cref{lemma:conc_tau} and \eqref{eqn:Li_estimate} to \eqref{eqn:general_Delta_tau} to obtain the first inequality and then \eqref{assu:trace} and $k=o(\sqrt d/\log^2 (n))$ to obtain the final equality in the display below:
\begin{align}\label{eqn:tau_d_Delta_tau_small}
\tau d |\Delta\tau^{ij}| = & \sqrt{d\log(1/p)} \left( \Tr\Sigma + \sqrt{\Tr\Sigma\frac{\log n}{d}} + \frac{\sqrt{\max{k,\log n}\log n}}{d} \right) + O\left(\frac{\log^{3/2}n}{\sqrt{d}}\right) \nonumber \\ 
= & O\left(\frac{1}{\sqrt{\log n}}\right)
\end{align}
with probability at least $1-n^{-10}$. 
Moreover, note that \eqref{eqn:tau_d_Delta_tau_small} also implies that by \Cref{lemma:conc_tau}
\begin{equation}\label{eq:DeltaTau_UB}
    |\Delta\tau^{ij}|=  O\left(\frac{1}{\sqrt{\log(n)}d\tau}\right) = O\left(\frac{1}{\sqrt{d\log(n)}}\right),
\end{equation} and more generally, this implies that
\begin{equation}\label{eqn:Deltatau_little_o}
   \Delta\tau^{ij}=o(\tau), 
\end{equation}\\
invoking \Cref{lemma:conc_tau} once more.

Now we are ready to prove the estimate $\Delta p_0^{ij} = -p_1 d_k \Delta \tau^{ij}(1+O(\Delta \tau^{ij} \tau d))$ in the regime-free case which will allow us to show that $\Delta p_1^{ij} = \tau \Delta p_0^{ij}(1+o(1))$. We know by \eqref{eqn:tau_ij} that  $\tau^{ij}<\tau$ with probability at least $1-n^{-12}$ over all $(i,j)\in [n]^2$, we have that:
\begin{equation}
    \Delta p_0^{ij} = \Prob_{\xi \sim \mathcal{D}_{d_k}} \left[ \xi \in (\sqrt{d_k}\tau^{ij}, \sqrt{d_k}\tau) \big\vert \tau^{ij} \right] = \frac{\Gamma(d_k/2)}{\sqrt{d_k\pi}\Gamma((d_k-1)/2)}\int_{\sqrt{d_k}\tau^{ij}}^{\sqrt{d_k}\tau} \left(1-\frac{\xi^2}{d_k}\right)^{\frac{d_k-3}{2}} d\xi. 
\end{equation}
As in \eqref{eqn:spherical_density}, let $f=f_{d_k}$ be the density function for $\cD_{d_k}$.

For $\xi>0$, $f$ is monotonically decreasing, so
\begin{equation}\label{eqn:Delta_p0_sandwich}
    -\sqrt{d_k}\Delta\tau^{ij}f(\sqrt{d_k}\tau)\leq \int_{\sqrt{d_k}\tau^{ij}}^{\sqrt{d_k}\tau}f(\xi)d\xi=\Delta p_0^{ij}\leq -\sqrt{d_k}\Delta\tau^{ij}f(\sqrt{d_k}\tau^{ij}).
\end{equation}
Comparing the upper bounds above,
\begin{equation}\label{eqn:ratio_spherical_density}
\begin{aligned}
    \frac{f(\sqrt{d_k}\tau^{ij})}{f(\sqrt{d_k}\tau^{})}&= \left(\frac{1-(\tau^{ij})^2}{1-\tau^2}\right)^{\frac{d_k-3}{2}}\\
    &=  \left(\frac{1-\tau^2 - 2\tau \Delta \tau^{ij} - (\Delta \tau^{ij})^2}{1 - \tau^2} \right)^{\frac{d_k-3}{2}} \\
    & = \left(1+O(\tau \Delta \tau^{ij}(1+O(\tau^2)))\right)^{\frac{d_k-3}{2}}\\
    &=\left(1 + O(-\tau \Delta \tau^{ij})\right)^{\frac{d_k-3}{2}} \\
    &= 1 + O(\tau d \Delta \tau^{ij}).
\end{aligned}
\end{equation}
To get the third equality we used \eqref{eqn:Deltatau_little_o}. In the fourth equality we applied \Cref{lemma:conc_tau} (in particular, $\tau=o(1)$). Finally, to obtain the last equality, we invoked
\eqref{eqn:tau_d_Delta_tau_small} and   to the binomial expansion of the expression in the second to last line. Then by mean value theorem for integrals, we have that for some $\xi\in [\sqrt{d_k}\tau^{ij}, \sqrt{d_k}\tau]$, then $\Delta p_0^{ij}=-\sqrt{d_k}\Delta\tau^{ij} f(\xi)$. Then combining this with
monotonicity of $f$ and \eqref{eqn:ratio_spherical_density}, gives us that  $f(\xi)/f(\sqrt{d_k}\tau)-1\leq(\sqrt{d_k}\tau^{ij})/f(\sqrt{d_k}\tau)-1\leq O(\tau d \Delta \tau^{ij})$. Thus 
\begin{equation}\label{eqn:Deltap0_conc}
    \Delta p_0^{ij}=-\sqrt d_k\Delta\tau^{ij}f(\sqrt{d_k}\tau)(1+O(\tau d\Delta\tau^{ij})).
\end{equation}
We have already computed $p_1$ explicitly in \eqref{eqn:p1_computation}, so we will use this previous expression to relate $p_1$ to the formula we have derived for $\Delta p_0^{ij}$ in \eqref{eqn:Deltap0_conc}. By \eqref{eqn:p1_computation} 
we have that \begin{equation}\label{eqn:p1_expantion_w_ratio}
p_1=\sqrt{d_k}(d_k-1)^{-1}f(\sqrt{d_k}\tau)(1-\tau^2).
\end{equation}
Then, re-arranging $f(\sqrt{d_k}\tau) = \frac{p_1(d_k-1)}{\sqrt{d_k}}(1+O(\tau^2))$ and plugging this into \eqref{eqn:Deltap0_conc}, gives 
\begin{equation}\label{eqn:Deltap0_conc2}
    \Delta p_0^{ij}=-(d_k-1)\Delta \tau^{ij}p_1(1+O(\tau^2))(1+O(\tau d\Delta\tau^{ij}))=-p_1d_k\Delta\tau^{ij}(1+O(\tau d\Delta\tau^{ij}))\,,
\end{equation}
where the second equality follows from \Cref{lemma:conc_tau} and \eqref{eqn:tau_d_Delta_tau_small}.

Now we are ready to compute $\Delta p_1^{ij}$. First, notice that $\Delta p_1^{ij}$ can be evaluated explicitly, since taking the derivative $$\frac{d}{d\xi}\big(1-\xi^2/d_k\big)^{(d_k-1)/2} = -\frac{d_k-1}{d_k}\xi\big(1-\xi^2/d_k\big)^{(d_k-3)/2},$$ 
we get that 
\begin{equation}
    \begin{aligned}
        \Delta p_1^{ij} &= \frac{1}{\sqrt{d_k}} \mathbb{E}_{\xi \sim \mathcal{D}_{d_k}} \left[ \xi \1\left(\xi \in (\sqrt{d_k}\tau^{ij}, \sqrt{d_k}\tau)\right) \big\vert \tau^{ij}\right] \\
    &= \frac{1}{\sqrt{d_k}}\int_{\sqrt{d_k}\tau^{ij}}^{\sqrt{d_k}\tau} \xi f(\xi) d\xi \\
    &=\frac{\Gamma(d_k/2)}{\sqrt{d_k\pi}\Gamma((d_k-1)/2)}\frac{d_k}{d_k-1}\left((1-(\tau^{ij})^2)^{(d_k-1)/2}-(1-\tau^2)^{(d_k-1)/2}\right).
    \end{aligned}
\end{equation}
 The computation we will perform below is analogous to the one for $\Delta p_0^{ij}$ in \eqref{eqn:Deltap0_conc} but keeps track of the errors more carefully. Recalling the explicit expression for $p_1$ in \eqref{eqn:p1_computation} results in \begin{equation}\label{eqn:Delta_p1_est}
\begin{aligned}
\Delta p_1^{ij}&= p_1\left[\left(\frac{1-(\tau^{ij})^2}{1-\tau^2}\right)^{(d_k-1)/2} - 1\right]=p_1\left[\left(1-\frac{2\tau\Delta\tau^{ij}+(\Delta\tau^{ij})^2}{1-\tau^2}\right)^{(d_k-1)/2}-1\right]\\
&=p_1\left[\left(1-(2\tau\Delta\tau^{ij}+(\Delta\tau^{ij})^2)(1+o(1))\right)^{(d_k-1)/2}-1\right], \quad \text{by \Cref{lemma:conc_tau}}\\
&=p_1\left[(1-2\tau\Delta\tau^{ij}(1+o(1)))^{(d_k-1)/2}-1\right]\quad \text{by }\eqref{eqn:Deltatau_little_o}\\
&=p_1\left(-d_k\tau \Delta\tau^{ij}(1+o(1))\right)\\
&= \tau \Delta p_0^{ij}\left(1+O(\tau d\Delta^{ij}\tau)\right)^{-1}(1+o(1))\quad\text{by }\eqref{eqn:Deltap0_conc2}\\
&=\tau \Delta p_0^{ij}(1+O(\tau d\Delta^{ij}))(1+o(1))\quad \text{by Taylor approximation,}\\
&= \tau \Delta p_0^{ij}(1+o(1))\quad \text{by }\eqref{eqn:tau_d_Delta_tau_small}.
\end{aligned}
\end{equation}
This gives the desired $\Delta p_1^{ij}$ estimate.

Now we proceed to specialize the above calculations to obtain regime-specific estimates. We begin with the small trace regime ($\Tr\Sigma\leq \tau^2$). We notice that in this case, $k/d = o\left(\Tr\Sigma\right)\leq \tau^2=\Theta(\log(1/p)/d)$. This implies that we must have $k=o(\log n)$, thus we replace every instance of $\max\{k,\log n\}$ with $\log n$. 
For $\langle a_i,a_j\rangle/\tau$, we use the largest estimate from \eqref{eqn:ai_inner_product} and \eqref{eqn:ai_norm},
\[\frac{\langle a_i,a_j\rangle}{\tau}\leq O\left(\frac{\Tr\Sigma}{\tau} + \frac{\sqrt{\Tr\Sigma}}{\tau}\sqrt{\frac{\log n}{d}} + \frac{\log n}{d\tau}\right)
\ls \tau +\sqrt{\frac{\log n}{d}}+\frac{\log n}{\sqrt{d\log(1/p)}}.
\]

In particular we use the fact that $\Tr\Sigma\leq \tau^2$ for the first two terms on the right hand side, which are of order $o(\frac{\log n}{\sqrt{d}})$, while
the third term is $O(\frac{ \log n}{\sqrt{d}})$. This simplifies to 
\begin{equation}
\label{eq:small snr inner prod conc}
\frac{\langle a_i,a_j\rangle}{\tau}=O\left(\frac{\log n}{\sqrt{d}}\right).
\end{equation}
Substituting $\tau^{ij} = (\tau-\langle a_i,a_j\rangle)/(l_il_j)$ into  the general expression \eqref{eqn:general_Delta_tau}, yields
\begin{equation}
    \Delta\tau^{ij}=-\tau\Big(\frac{\langle a_i,a_j\rangle}{\tau} + L_i+L_j + o\big(\log^{3/2}n/d\big)\Big),
\end{equation}
where we group the error terms $\langle a_i,a_j\rangle\tau^{-1} o(1)+O(\log(n)/d)=o(\log^{3/2}(n)/d)$ by \eqref{eq:small snr inner prod conc}. Note that in this regime, $\Delta\tau^{ij}=O(\tau \log n/\sqrt d)$, so we can refine our regime-free estimate in \eqref{eqn:Deltap0_conc2} by plugging this quantity into the error, thus yielding $\Delta\tau^{ij}=-p_1d_k\Delta\tau^{ij}(1+O(\tau^2\sqrt d\log n))$, as desired. 

Now we move to the large trace regime ($\Tr\Sigma> \tau^2$). In this case, we may use both the forms {of $\Delta \tau^{ij}$} given in \eqref{eqn:tau_d_Delta_tau_small}, but we choose to present the more general one in the statement of the lemma. 
Moreover, we keep the regime-free expression for $\Delta p_0^{ij}$ in \eqref{eqn:Deltap0_conc2}.

\end{proof}

\subsection[Proof of \Cref{lem:p1}]{{Proof of Lemma \ref{lem:p1}}}
\label{sec:p1}

We begin by approximating the realized edge density $\hat p$, in terms of the SNGM parameter $p$.

\begin{lemma}
\label{lem:pHat_to_p}
    Let $\hat p$ be as in \Cref{sec:algorithm}. Then under the conditions of \Cref{thm:main}, with probability $1-O(n^{-11})$ we have
$
|\hat p -p|\lesssim 
\sqrt{\frac{p\log n}{n}} = o(p).
$
\end{lemma}

Throughout this section, we will use the following notation for $1\leq i,j\leq n$
\begin{equation}\label{eq:deg_mean_row_sum}
    \deg(i)= \sum_{j\neq i}^nA_{ij},\qquad q_{j}(y)=\P
    \big(\langle y, x_j+z_j\rangle\geq \tau\big),
\qquad
r_i(y)=\sum_{j\neq i}^n q_j(y).
\end{equation}
Note that $q_i,r_i$ are deterministic functions of $y$ since $\cX$ is deterministic. Moreover, by independence of $z_1,\dots, z_n$, we have that for all $i\neq j$, $(A_{ij}|z_i)_{j\neq i}^n$ are independent Bernoulli random variables, with means $(q_j(y_i))_{j\neq i}^n$ respectively. In particular $\E[\deg(i)|z_i]=r_i(y_i)$. The proof will proceed in 2 steps, the first is an application of Bernstein's inequality for bounded random variables \cite[Theorem 2.8.4]{Vershynin2018} to $\deg(i)$, conditioned on $z_i$. Then, we ``uncondition'' via another application of Bernstein's inequality applied to $\sum_i^n r_i(y_i)$.

\begin{lemma}[Concentration of degrees]
    \label{lem:degree-bernstein}
With probability at least $1-2n^{-11}$, for all $i\in [n]$
\begin{equation}\label{bd:degree_conc}
\big|\deg(i) - r_i(y_i)\big| \leq 9\sqrt{r_i(y_i)\log n}+18\log n.
\end{equation}
\end{lemma}
\begin{proof}[Proof of \Cref{lem:degree-bernstein}]
    We work on an event conditioned on $z_i$, with $i\in [n]$, fixed. We seek to apply Bernstein's inequality to $\deg(i)$, which is a sum of independent Bernoulli variables $A_{ij}$ of respective conditional variances $\Var(A_{ij}|z_i) \le q_j(y_i)$. 
Then, applying \cite[Theorem 2.8.4]{Vershynin2018} with
$t_i=9\sqrt{r_i(y)\log n}+18\log n$ gives
\[
\P\big(|\deg(i)-r_i(y)|\geq t_i|z_i\big)\leq 2\exp\left(-\frac{t_i^2/2}{r_i(y_i)+t_i/3}\right)\leq 2n^{-12}.
\]
Since this bound holds for any realization of $z_i$, the claim follows from the union bound. 
\end{proof}

\begin{lemma}\label{lem:conc_row_sum}
    Let $R=\sum_{i=1}^n r_i(y_i)$. Then $\E R=n(n-1)p$, and with probability at least $1-2n^{-12}$, 
    \begin{equation}\label{eq:conc_row_sum}
        \big|R-\E R\big|\leq 9\sqrt{pn^3 \log n}+18n\log n.
    \end{equation}
In particular, on the same event we have $R\leq 2n^2 p.$
\end{lemma}
\begin{proof}[Proof of \Cref{lem:conc_row_sum}]
By definition of $r_i, q_j$, 
\[
\E_{z_i} r_i(y_i)=\sum_{j\neq i}^n \E_{z_i} \Prob_{z_j}(\langle y_i,y_j\rangle \geq \tau)=\sum_{j\neq i}^n \Prob_{z_j, z_i}(\langle y_i,y_j\rangle \geq \tau),
\]
and hence by the definition \eqref{def:taup},
\[\E R=\sum_{i\neq j}\P\left(\langle y_i,y_j\rangle\geq \tau\right)=n(n-1)p\,.
\]
By independence of $z_1,\dots, z_n$ and the uniform bound $r_i(y_i)\le n$, 
\[
\sum_{i=1}^n\Var(r_i(y_i))\leq \sum_{i=1}^n \E[r_i(y_i)^2]\leq n\E R.
\]
The claim \eqref{eq:conc_row_sum} now follows from an application of Bernstein's inequality as in the proof of \Cref{lem:degree-bernstein}. The bound $R\le 2n^2p$ then follows from \eqref{eq:conc_row_sum} and the assumption \eqref{assu:density}.
\end{proof}

\begin{proof}[Proof of \Cref{lem:pHat_to_p}]
    We work on the intersection of the events on which \eqref{bd:degree_conc} and \eqref{eq:conc_row_sum} hold, which occurs with probability $1-O(n^{-11})$. Recalling that $\hat p=|E(G)|\binom{n}{2}^{-1}$, we have
    \begin{equation}\label{eq:pHat_diff}
        |\hat p-p|
        = \frac1{n(n-1)} \bigg|  \sum_{i=1}^n \deg(i) - \E R\bigg|
        = \frac1{n(n-1)} \bigg|(R-\E R)+  \sum_{i=1}^n (\deg(i) - r_i(y_i)) \bigg|.
    \end{equation}
From \eqref{bd:degree_conc} and Cauchy--Schwarz,
\[
\bigg|\sum_{i=1}^n\big(\deg(i)-r_i(y_i)\big)\bigg|\lesssim n\log n+
\sqrt{\log n}\sum_{i=1}^n\sqrt{r_i(y_i)}
\lesssim \sqrt{Rn\log n}+ n\log n\lesssim\sqrt{pn^3\log n}
\]
where we used the bound $R\leq 2n^2p$ from \Cref{lem:conc_row_sum}, and noted that $n\log n = o(\sqrt{pn^3\log n})$ under \eqref{assu:density}. We get the same bound for $|R-\E R|$ from \eqref{eq:conc_row_sum}, and the claim follows.  
\end{proof}

\begin{proof}[Proof of \Cref{lem:p1}]
Define 
\begin{equation}
    \fat p_1:=\frac{\Gamma((d-k)/2)}{(d-k-1)\sqrt\pi\,\Gamma((d-k-1)/2)}\big(1-\hat\tau^2\big)^{\frac{d-k-1}2}.
\end{equation}
We will proceed by a triangle inequality, 
\begin{equation}\label{eq:triangle_ineq_fatP}
    |p_1-\hat p_1|\leq |p_1-\fat p_1|+|\fat p_1 -\hat p_1|.
\end{equation}
We start  by finding the first term on the right hand-side above. By definition of $\fat p_1, p_1$, write
\begin{equation}\label{eq:p1_p1Hat_diff_integral}
\begin{aligned}
    |p_1-\fat p_1|&= \frac{1}{\sqrt{d_k}}\bigg|\int^{\sqrt{d_k}\tau}_{\sqrt{d_k}\hat \tau}\xi f(\xi)d\xi\bigg|\le |\tau -\hat \tau|\sup_{\xi \in [\sqrt{d_k}
    \tau, \sqrt{d_k}\hat \tau] }|\xi f(\xi)|\\
    &\leq |\tau-\hat \tau|\sqrt{d_k}\max(\tau,\hat \tau)f(\sqrt{d_k}\min(\tau,\hat \tau)).
\end{aligned}
\end{equation}
The last equation was obtained by \eqref{eqn:spherical_density} which says that $f$ is decreasing on $[0,\sqrt{d_k})$. We first estimate $f(\sqrt{d_k}\min(\tau,\hat \tau))$ by comparing it with $f(\sqrt{d_k}\tau)$. If $\tau\leq \hat \tau $, then the two quantities are equal. If $\tau>\hat \tau$, then
\begin{equation}\label{eq:ratio_density}
    \frac{f(\sqrt{d_k}\hat\tau)}{f(\sqrt{d_k}\tau)}=  \left(\frac{1-\hat \tau ^2}{1-\tau^2}\right)^{\frac{d_k-3}{2}}= \left(1+\frac{\tau^2-\hat \tau^2}{1-\tau^2}\right)^{\frac{d_k-3}{2}}\leq \exp\left(d_k\left(\tau^2-\hat \tau^2\right)\right)\leq \exp\left(2d_k\tau|\tau-\hat \tau|\right).
\end{equation}
The above display is bounded by $2$ if we can show that $d\tau|\tau-\hat \tau|=o(1)$. Moreover, one also needs to find $|\tau-\hat \tau|$ as it is an expression on the right-hand side of \eqref{eq:p1_p1Hat_diff_integral}. To this end, we write by definition of $\hat \tau$ and $\tau$, and using the same $\Psi$ for the complementary Gaussian density as in \Cref{lemma:tau_gaussian_quant}:
\begin{equation}\label{eq:tau_diff}
    \sqrt{d}|\tau-\hat \tau|\leq |\sqrt{d}\tau -\Psi^{-1}(p)| +|\Psi^{-1}(p)-\Psi^{-1}(\hat p)|.
\end{equation}
By \Cref{lemma:tau_gaussian_quant}, the first term is $\cE=\sqrt{d}\Tr \Sigma + (\Tr \Sigma)^{1/2}\log n + \frac{\log^2n}{\sqrt{d}}$. The second term can be estimated by the mean value theorem (MVT). For $y$, in between $p$ and $\hat p$, by the MVT and the inverse function theorem (with $-\Psi'(t)=\frac{1}{\sqrt{2\pi}}\exp(-t^2/2)$),
\begin{equation}\label{eq:MVT_IFT}
    |\Psi^{-1}(p)-\Psi^{-1}(\hat p)|\lesssim \frac{|p-\hat p|}{|\Psi'(\Psi^{-1}(y))|}.
\end{equation}
Then, note that by a Chernoff bound, $\Psi(t)\lesssim |\Psi'(t)|$, for $t\ge 0$, so \eqref{eq:MVT_IFT} is bounded above by $|p-\hat p|/y$. By \Cref{lem:pHat_to_p}, we can write $y=p(1+O(\sqrt{\tfrac{\log n}{np}}+\tfrac{\log n}{np}))$, and similarly using the same result for the numerator, we have that
\begin{equation}\label{eq:Psi_diff}
    |\Psi^{-1}(p)-\Psi^{-1}(\hat p)|\lesssim O\left(\sqrt{\tfrac{\log n}{np}}+\tfrac{\log n}{np}\right).
\end{equation}
Combining \eqref{eq:Psi_diff}, \eqref{eq:tau_diff} with \Cref{lemma:conc_tau} gives that \begin{equation}\label{eq:tau_diff_scaled}
d\tau|\tau-\hat \tau|\lesssim \sqrt{\log(1/p)}\left(\cE + \sqrt{\tfrac{\log n}{np}}+\tfrac{\log n}{np}\right)\lesssim \log^{-3/2}(n)=o(1)\end{equation} by \eqref{assu:density},\eqref{assu:trace}. This implies that the ratio in \eqref{eq:ratio_density} is at most 2. Then by \eqref{eqn:p1_expantion_w_ratio} $f(\sqrt{d_k}\tau)\leq 2\sqrt{d_k}p_1$. Finally, using the aforementioned fact combined with \eqref{eq:tau_diff}, \Cref{lemma:conc_tau} we have that \eqref{eq:p1_p1Hat_diff_integral} becomes
\begin{equation}\label{eq:triangle1_pFat_diff}
    |p_1-\fat p_1|\lesssim  |\tau-\hat \tau|\max(\tau,\hat \tau) d_kp_1 \lesssim |\tau-\hat \tau| (\tau +|\tau -\hat \tau|)d_kp_1 
    \lesssim p_1\sqrt{\log(1/p)}\left(\cE+\sqrt{\tfrac{\log n}{np}}+\tfrac{\log n}{np}\right).
\end{equation}
Now it remains to estimate $|\hat p_1-\fat p_1|$. 
Note that from the quantitative form of Stirling's approximation
\[
\Gamma(u) = e^{O(1/u)}\sqrt{\frac{2\pi}u}(u/e)^u
\]
for real $u\ge1$, we have
    \[
    \frac{\Gamma((d-k)/2)}{(d-k-1)\sqrt\pi\,\Gamma((d-k-1)/2)} = (1+O(1/d)) \frac1{\sqrt{2\pi (d-k)}}\,.
    \]
    Furthermore
    \[
    (1-\hat \tau^2)^{(d-k-1)/2} = e^{O(d\hat \tau^4)} e^{-\hat \tau^2(d-k-1)/2}
    \]
    so for $\hat \tau\lesssim d^{-1/2}\sqrt{\log (1/\hat p)}$ by \eqref{eq:tau_diff} and \Cref{lemma:conc_tau},
    \begin{align*}
    \fat p_1 =& (1+ O(\frac{\log^2(1/p)}{d})) \frac1{\sqrt{2\pi (d-k)}} e^{-\hat \tau^2(d-k)/2}.
    \end{align*}
    Then using $\fat p_1\leq 2p_1$, \Cref{lemma:scalar_estimates_whp} and \Cref{rmk:k}
    $$|\fat p_1 -\hat p_1|\lesssim  \fat p_1 \frac{k\log^2(1/p)}{d}\leq 2p_1 \frac{k\log^2(1/p)}{d}\leq p_1 \frac{\log^{2}n}{\sqrt d}\lesssim p_1\sqrt{\log(1/p)}\cE,$$
    which combined with \eqref{eq:triangle1_pFat_diff} concludes the proof of the lemma.
\end{proof}

\subsection[Proof of Lemma \ref{lemma: DeltaTau_norm_estimates}]{Proof of Lemma \ref{lemma: DeltaTau_norm_estimates}}
\label{proof of lemma: DeltaTau_norm_estimates}

In what follows, assume we are on the high probability event on which estimates of \Cref{lemma:scalar_estimates_whp} hold.

(a): Here, we want to formalize the fact that $\Delta \tau =(\text{PSD term})+ (\text{small terms})$. We will see that, 
$(\text{PSD term})= aa^\top$, so the crux of the proof will be based on the following triangle inequality,
\begin{equation}\label{eqn: DeltaTau_odot_VV_norm}
\begin{aligned}
    \norm{\Delta\tau\odot VV^\top}_\op &\leq \norm{aa^\top\odot VV^\top}_\op+ \norm{\big[aa^\top -\tau\1\1^\top\big]\odot\big(\1L^\top + L\1^\top\big) \odot VV^\top}_\op \\
    &\quad + \norm{\tau O\left(\frac{\log n}{d}\1\1^\top\right)\odot VV^\top}_\op.
\end{aligned}
\end{equation}
For the first term on the right hand side, we use \Cref{fact:hadamard_product_bound} to obtain
\begin{equation}\label{eqn: PSD_part_DeltaTau_norm}
 \|aa^\top\odot VV^\top\|_{\mathrm{op}}\le\|VV^\top\|_{\mathrm{op}}\|aa^\top\|_\infty
        \lesssim \underbrace{\frac nd}_{\text{by \Cref{fact:spherical_matrix_norm}}} \underbrace{\Tr\Sigma\log n}_{\text{by \Cref{lemma:scalar_estimates_whp}}}.
\end{equation}
For the other two terms on the right hand side of \eqref{eqn: DeltaTau_odot_VV_norm}, we can bound 
\begin{align*}
\norm{(aa^\top\odot VV^\top)\diag(L)+\diag(L)(aa^\top\odot VV^\top)}_\op&\leq 2\norm{(aa^\top\odot VV^\top)}_\op\norm{\diag(L)}_\op  
\\
&\lesssim \frac nd \Tr\Sigma\log n\sqrt{\tfrac{\log n}{d}}\qquad \text{ by \eqref{eqn: PSD_part_DeltaTau_norm}, \Cref{lemma:scalar_estimates_whp}}, 
\end{align*}
 and similarly,
 \begin{align*}
\norm{\tau\big(VV^\top\diag(L)+\diag(L)VV^\top\big)}_\op\leq 2\tau\norm{VV^\top}_\op\norm{\diag(L)}_\op\lesssim \frac nd\tau\sqrt{\frac{\log n}{d}}\lesssim \frac{n}{d}\tau^2 \log n.
 \end{align*}
The last inequality was obtained by \Cref{lemma:conc_tau}. Finally, the third term on the right hand side of \eqref{eqn: DeltaTau_odot_VV_norm}, has norm \[\norm{\tau O\left(\frac{\log n}{d}\right)\1\1^\top\odot VV^\top}_\op\leq \tau O\left(\frac{\log n}{d}\right)\norm{\1\1^\top}_\infty \norm{VV^\top}_\op\lesssim \frac nd\frac{\tau\log n}{d}
\]
by \Cref{fact:hadamard_product_bound} and \Cref{lemma:scalar_estimates_whp}.
Taking the dominant terms these estimates to the right hand side of \eqref{eqn: PSD_part_DeltaTau_norm} gives the desired result. 
(b): We split the expression of $\Delta \tau^{\odot 2}$ into 2 parts,
\begin{equation}
    \norm{\Delta\tau^{\odot 2} \odot VV^\top }_\op\leq (1+ O(\tfrac{\log n}{d}))^2\norm{aa^{\top \odot 2}\odot VV^\top}_\op + \norm{R_3' \odot VV^\top}_\op.
\end{equation}
We set the remainder term equal to \[R_3'= \Delta\tau^{\odot 2}- aa^{\top\odot 2}(1+O\left(\frac{\log n}{d}\right))^2,\]
and estimate it in \Cref{appendix:R3_prime} below. \Cref{appendix:R3_prime} gives us that 
\[\norm{R_3'\odot VV^\top}_\op\lesssim O\left(\frac{n\log^3 n}{d^{3/2}}\tau^2\right).\]
To control the first term on the right hand side, we note that $aa^{\top \odot 2}$ is PSD, so by \Cref{fact:hadamard_product_bound}, we have that 
\[\|(aa^\top)^{\odot 2}\odot VV^\top\|_{\mathrm{op}}\le\|VV^\top\|_{\mathrm{op}}\|aa^{\odot 2}\|_\infty
\lesssim \frac nd(\Tr\Sigma\log n)^2. \]
Then taking the dominant terms, gives us the desired result. \newline
(c): Here, we also split $\Delta\tau^{\odot 2}$ exactly as in \Cref{appendix:R3_prime},
\begin{align*}
  \Delta\tau^{\odot 2}
  &= (aa^\top)^{\odot 2}\big(1+O(\tfrac{\log n}{d})\big)^2 + R_3',
  \\
  R_3'& = -2\big(1+O(\tfrac{\log n}{d})\big)\,(aa^\top)\odot R_3'' + R_3''^{\odot 2},
\end{align*}
where $$R_3''=\big[aa^\top-\tau\1\1^\top\big]\odot(\1L^\top+L\1^\top)-\tau O(\tfrac{\log n}{d})\1\1^\top,$$ as in \Cref{appendix:R3_prime}. By the triangle inequality,
\begin{equation}\label{eqn:DeltaTauSq_split}
  \norm{\Delta\tau^{\odot 2}}_\op
  \le \big(1+o(1)\big)\norm{(aa^\top)^{\odot 2}}_\op + \norm{R_3'}_\op .
\end{equation}
 Since $aa^\top$ is positive semi-definite, by \Cref{fact:hadamard_product_bound}:
\begin{equation}\label{eqn:aaPSDsquared_op}
  \norm{(aa^\top)^{\odot 2}}_\op
  = \norm{(aa^\top)\odot(aa^\top)}_\op
  \le \norm{aa^\top}_\op\,\norm{aa^\top}_\infty
  \lesssim n\mu_1^2\cdot \Tr\Sigma\log n,
\end{equation}
using $\norm{aa^\top}_\op\lesssim n\mu_1^2$ and $\norm{aa^\top}_\infty\lesssim \Tr\Sigma\log n$
from \Cref{lemma:scalar_estimates_whp}.

Now we move on to bounding $\|R_3'\|_\op$. For the first summand in $R'_3$, we will see that the dominant part of $(aa^\top)\odot R_3''$ is
$\big[(aa^\top)^{\odot 2}-\tau aa^\top\big]\odot(\1L^\top+L\1^\top)$. Pulling out $\diag(L)$ and using sub-multiplicativity together with \eqref{eqn:aaPSDsquared_op},
\begin{align*}
  \norm{(aa^\top)\odot R_3''}_\op
  &\le 2\,\norm{(aa^\top)^{\odot 2}-\tau aa^\top}_\op\,\norm{\diag(L)}_\op
     + \tau O(\tfrac{\log n}{d})\norm{aa^\top}_\op\\
  &\lesssim \norm{aa^\top}_\op\big(\norm{aa^\top}_\infty+\tau\big)\sqrt{\tfrac{\log n}{d}}
   \;\lesssim\; n\mu_1^2\big(\Tr\Sigma\log n+\tau\big)\sqrt{\tfrac{\log n}{d}},
\end{align*}
which is $\lesssim n\mu_1^2\,\Tr\Sigma\log n$. Due to the fact that $\Tr\Sigma=\omega(1/d)$, this gives us that
$\tau\sqrt{\log n/d}=o(\Tr\Sigma\log n)$. 

The second summand in $R_3'$, $R_3''^{\odot 2}$, is not PSD,
so we pass to the Frobenius norm, i.e.
\[
  \norm{R_3''^{\odot 2}}_\op
  \le \norm{R_3''^{\odot 2}}_F
  \le n\,\norm{R_3''}_\infty^2
  \lesssim n(\Tr\Sigma \log n+\tau)^2\frac{\log n}{d}
  \lesssim \frac{n((\Tr \Sigma)^2 \log^2n + \tau^2)\log n}{d},
\]
using $\norm{R_3''}_\infty\lesssim(\Tr\Sigma \log n+\tau)\sqrt{\log n/d}$ from
\eqref{eqn:infty_norm_of_R3''} and $\Tr\Sigma=o(\tau)$. Hence
$\norm{R_3'}_\op\lesssim n\mu_1^2\Tr\Sigma\log n+\tfrac{n\tau^2\log n}{d}$, and substituting this
together with \eqref{eqn:aaPSDsquared_op} into \eqref{eqn:DeltaTauSq_split} gives
\[
  \norm{\Delta\tau^{\odot 2}}_\op
  \lesssim n\mu_1^2\,\Tr\Sigma\log n + \frac{n\tau^2\log^3 n}{d},
\]
as claimed.
\qed

\begin{lemma}\label{appendix:R3_prime} Following the definitions in the proof of \Cref{lemma:Bsph_Bdata_diff}, if \[R_3'= \Delta\tau^{\odot 2}- aa^{\top\odot 2}\left(1+O\left(\frac{\log n}{d}\right)\right)^2,\] then with probability at least $1-O(n^{-12})$,
\[\norm{R_3'\odot VV^\top}_\op\lesssim O\left(\frac{n\log^3 n}{d^{3/2}}\tau^2\right)\]
\end{lemma}
\begin{proof}
As in the proof of \Cref{lemma: DeltaTau_norm_estimates}, we assume that we are on the high probability event on which estimates of \Cref{lemma:scalar_estimates_whp} hold. Now, expanding $\Delta \tau^{\odot2}$, we can write the difference we are seeking {to control} as
\begin{equation}\label{eqn: r3_prime_expanded}R_3'=\Delta\tau^{\odot2}-(aa^{\top})^{\odot2}\big(1+O(\frac{\log n}{d})\big)^2  = -2\big(1+O(\frac{\log n}{d})\big) (aa^{\top})\odot R_3'' + R_3''^{\odot2},\end{equation}
    where we define
    \[R_3''= \big[aa^{\top}-\tau\1\1^{\top}\big]\odot(\1L^{\top}+L\1^{\top}) - \tau O(\frac{\log n}{d})\1\1^{\top}.\]
   To bound the contribution of the first summand on the right hand side of \eqref{eqn: r3_prime_expanded} to $\| R_3' \odot VV^\top \|_\op$, we pull out $\diag(L_1,\dots, L_n)$ and apply \Cref{fact:hadamard_product_bound} to the components containing the PSD matrices $(aa^{\top })^{\odot 2}$ and $aa^{\top }$:
\begin{equation}\label{eqn:DeltaTauHadSquared1}
\begin{aligned}
    &\norm{(aa^{\top }\odot R_3'')\odot VV^{\top }}_\op =\norm{aa^\top \odot  \left(\big[aa^{\top}-\tau\1\1^{\top}\big]\odot(\1L^{\top}+L\1^{\top}) - \tau O(\frac{\log n}{d})\1\1^{\top}\right)\odot VV^\top}_\op\\
    &= \big\|\Big(\big[(aa^{\top})^{\odot2}-\tau aa^{\top}\big]\odot VV^\top\Big)\diag(L)
\\
&\qquad +\diag(L)\Big(\big[(aa^{\top})^{\odot2}-\tau aa^{\top}\big]\odot VV^\top\Big)
- \Big(aa^\top\odot \tau O\big(\tfrac{\log n}{d}\big)\1\1^\top\Big)\odot VV^\top \|_\op\\
    &\lesssim \left(\norm{aa^{\top \odot2}}_\infty  + \tau \norm{aa^\top}_\infty\right)\cdot O(\sqrt{\tfrac{\log n}{d}}) \norm{VV^\top}_\op \qquad \\
    &\qquad \qquad \text{ by \Cref{fact:hadamard_product_bound} and taking dominant terms}\\
    &\lesssim 
    \sqrt{\frac{\log n}{d}} \big[(\Tr\Sigma \log n)^2 + \tau \Tr\Sigma \log n\big] \left(\frac{n}{d}\right),
\end{aligned}
 \end{equation}
 where $\diag(L)= \diag(L_1,\dots, L_n)$ and we invoked bounds of \Cref{lemma:scalar_estimates_whp} where needed.
 
Next we bound $ \norm{R_3''^{\odot2}\odot VV^\top}_\op$. We do so by passing to the Frobenius norm, which we control by $\|R_3''\|_\infty$ (see \eqref{eqn:DeltaTauHadSquared2} below). To this end, we must estimate $\|R_3''\|_\infty$. Once again plugging in the bounds from \Cref{lemma:scalar_estimates_whp} we get that 
\begin{equation}\label{eqn:infty_norm_of_R3''}
\begin{aligned}
    \norm{R_3''}_\infty&\leq \norm{aa^{\top}\odot(\1L^{\top}+L\1^{\top})}_\infty +\norm{\tau\1\1^{\top}\odot(\1L^{\top}+L\1^{\top})}_\infty + \norm{\tau O(\frac{\log n}{d})\1\1^{\top}}_\infty\\
    &\lesssim O\left(\left(\Tr\Sigma\log n + \tau\right)\sqrt{\frac{\log n} {d}}\right) + \tau O\left(\frac{\log n}{d}\right)\\
    &\lesssim O\left((\Tr\Sigma \log n + \tau)\sqrt{\frac{\log n}{d}}\right).
\end{aligned}
\end{equation}
Therefore, 
 \begin{equation}\label{eqn:DeltaTauHadSquared2}
 \begin{aligned}
     \norm{R_3''^{\odot 2}\odot VV^\top}_\op&\leq \norm{R_3''^{\odot 2}\odot VV^\top}_F\\
     &\leq \norm{R_3''}^2_\infty\norm{VV^\top}_F\\
     &\lesssim \underbrace{\left((\Tr\Sigma +\tau)^2\frac{\log^3 n}{d}\right)}_{\eqref{eqn:infty_norm_of_R3''}}\underbrace{\left(\frac{n}{\sqrt{d}}\right)}_{\text{\Cref{fact:spherical_matrix_norm}}}
\end{aligned}
 \end{equation}

 Finally, we combine our estimates \eqref{eqn:DeltaTauHadSquared1} and \eqref{eqn:DeltaTauHadSquared2} to obtain 
 \begin{align*}
     \norm{R_3'\odot VV^\top}_\op&\leq 2(1+O\left(\frac{\log n}{d}\right))\norm{aa^\top \odot R_3''\odot VV^\top}_\op + \norm{ R_3''^{\odot 2}\odot VV^\top}_\op\\
     &\lesssim (1+O\left(\frac{\log n}{d}\right))\big[(\Tr\Sigma \log n)^2 + \tau \Tr\Sigma \log n\big]O\left(\frac{n\sqrt{\log n}}{d\sqrt{d}}\right) + \left( (\Tr\Sigma +\tau)^2\right)\left(\frac{n{\log^3 n}}{d\sqrt{d}}\right)\\
     &\lesssim \frac{n}{d^{3/2}}\tau^2\log^3n.
 \end{align*}
The last line was obtained using \Cref{lemma:conc_tau} in conjunction with \eqref{assu:trace}.
\end{proof}

\section{Table of Notation}\label{app:notation}

\renewcommand{\arraystretch}{1.3}
\begin{longtable}{lp{10.5cm}}
%\caption{Summary of frequently used notation and variables.}
\label{tab:notation} \\

\toprule
\textbf{Symbol} & \textbf{Description} \\
\midrule
\endfirsthead

\multicolumn{2}{c}{{\bfseries \tablename\ \thetable{} -- continued from previous page}} \\
\toprule
\textbf{Symbol} & \textbf{Description} \\
\midrule
\endhead

\midrule \multicolumn{2}{r}{{Continued on next page}} \\ \bottomrule
\endfoot

\bottomrule
\endlastfoot

\multicolumn{2}{l}{\textbf{Asymptotics \& Norms}} \\
\midrule
$\omega(a_n) \leq b_n$, $a_n = o(b_n)$ & $\lim_{n\rightarrow\infty}|a_n/b_n| = 0$ \\
$a_n = \Theta(b_n)$, $a_n \asymp b_n$ & Asymptotic equivalence up to constant factors (holds for large enough $n$ for $\Theta$, and holds for all $n$ for $\asymp$) \\
$\lambda_i(B)$, $\sigma_i(B)$ & For matrix $B$, the $i$-th eigenvalue (when $B$ is symmetric) and $i$-th singular value, respectively, labelled in non-increasing order. \\
$\|B\|_{op}$, $\|B\|_F$ & The spectral operator norm and Frobenius norm of matrix $B$, respectively \\
$\|\cdot\|_{\psi_1}$, $\|\cdot\|_{\psi_2}$ & Subexponential and subgaussian norms, respectively, for random variables \\
$\Pi^{(k)}_{B}$ & Orthogonal projection matrix onto the subspace spanned by the top $k$ eigenvectors of $BB^\top$. \\

\midrule
\multicolumn{2}{l}{\textbf{Data Setup \& Graph Model}} \\
\midrule
$n$, $d$& Number of samples, ambient dimension of samples, respectively. \\
$A$ & The $n \times n$ adjacency matrix of the generated graph $\mathcal{G}_{n,d}(p,\cX)$. \\
$p$ & Average edge probability in the random SNGM. \\
$\cX$ & Deterministic point cloud, $\cX = \{x_1, \dots, x_n\} \subset \mathbb{R}^d$.
\\
$X \in \mathbb{R}^{n \times d}$ & Clean data matrix, with rows $x_i$. \\
$\Sigma$ & Sample covariance matrix $\Sigma := \frac{1}{n}X^\top X$. \\
$\mu_i^2$ & The $i$-th largest eigenvalue of $\Sigma$.
\\
$Z$ & Gaussian noise matrix, with independent rows $z_i \sim \mathcal{N}(0, \frac{1}{d}I_d)$. \\
$Y \in \mathbb{R}^{n \times d}$ & Noisy data matrix, with rows $y_i = x_i + z_i$. \\
$\tau$ & Inner product threshold defining $A$: $A_{ij}=\1(\langle y_i, y_j \rangle \ge \tau$). \\
\midrule
\multicolumn{2}{l}{\textbf{Truncation to Low Rank}} \\
\midrule
$k$ & Numerical rank of $\Sigma$.\\
$\Sigma_k$ & Projection onto the top $k$ eigenvalues of $\Sigma$, $\Sigma_k =\Pi^{(k)}_\Sigma\Sigma$ \\
$\Sigma_{k+1:d}$ & Projection onto the tail eigenvalues of $\Sigma$, $\Sigma_{k+1:d} = \Sigma-\Sigma_k$. \\

$\tilde p$ & Average edge probability for the truncated (strictly low-rank) adjacency matrix $\tilde A$. \\

\midrule
\multicolumn{2}{l}{\textbf{Spectral Approximation Variables}} \\
\midrule
$\Xalg$ & The $n \times k$ embedding matrix constructed from the top $k$ eigenvectors of $A$. \\
$\hat X$ & Rescaled embedding $\hat X=\mathbf W_k\Lambda_k^{1/2}$ \eqref{def:hatX}, i.e.\ $\Xalg$ with $\hat p_1$ replaced by $p_1$. \\
$\{\eta_i, w_i\}_{i=0}^{n-1}$ &  Eigenpairs of the adjacency matrix $A$, with $\eta_0\ge\cdots\ge\eta_{n-1}$. \\
$X_k$ & Clean data projected onto its top $k$ eigenspaces ($X_k = \Pi^{(k)}_{X}X$). \\
$\tilde A$ & Adjacency matrix induced by noisy projected data, $\tilde A_{ij} = \1(\langle \Pi^{(k)}_{X}x_i + z_i, \Pi^{(k)}_{X}x_j + z_j \rangle \ge \tau)$. \\
$\tilde Y$ & Matrix of low-rank signal plus noise, $\tilde Y = X_k + Z$, with rows $\tilde y_i=\Pi^{(k)}_{X}x_i + z_i$. \\
$\hat p$ & Empirical edge density $\hat p=\binom n2^{-1}|E(G)|$, an estimator of $p$. \\
$\Psi(t)$ & The complementary distribution function for a standard Gaussian variable $g$ on $\R$, i.e. $\Psi(t)=\P( g\ge t)$.\\
$\alpha(q)$ & Scaling function $\alpha(q)=\sqrt{2\pi/(d-k)}\exp\!\big(\tfrac{d-k}{2d}\Psi^{-1}(q)^2\big)$ \eqref{def:scaling}; satisfies $\alpha(\hat p)=1/(d_k\hat p_1)$. \\
$Y_k$ & Projection of the noisy data $Y$ onto the top $k$ eigenspaces of $YY^\top$, $Y_k = \Pi^{(k)}_{Y}(X + Z)$. \\
$\tilde Y_k$ & Projection of $\tilde Y$ onto the top $k$ eigenspaces of $YY^\top$, $\tilde Y_k = \Pi^{(k)}_{\tilde Y}(X_k + Z)$. \\
$a_i$, $w_i$ & Decomposed components of $\tilde y_i$, where $a_i \in \mathbb{R}^k$ contains signal+noise, and $w_i \sim \mathcal{N}(0, \frac{1}{d}I_{d-k})$. \\
$l_i$, $v_i$ & Magnitude ($l_i = \|w_i\|_2$) and direction ($v_i = w_i/\|w_i\|_2$) of the orthogonal noise component $w_i$. \\
$d_k$ & Dimension of the orthogonal noise component ($d_k = d - k$). \\

\midrule
\multicolumn{2}{l}{\textbf{Notation from \Cref{subsec:random_sampling}}} \\
\midrule
$\nu$ & Compactly supported measure on $R^d$.\\
$\Sigma^{\infty}$ & The $d\times d$ covariance matrix induced by the measure $\nu$: $\Sigma^{\infty} =\E_{\bfx\sim \nu}[\bfx\bfx^\top]$.\\
$\lambda_i$& The $i$-th largest eigenvalue of $\Sigma^{\infty}$.
\end{longtable}

\quad

\subsection*{Acknowledgements}  

The authors gratefully acknowledge the support of the National Science Foundation. Early stages of this work were supported under NSF grant DMS-1928930 while the second author was in residence at the Simons Laufer Mathematical Sciences Institute in Berkeley, California, during the Spring 2025 semester program \emph{Probability and Statistics of Discrete Structures}.
In later stages of this work, N.A.C.\ and S.P.\ were supported by NSF grant DMS-2553253.

\bibliographystyle{alphaurl}

\bibliography{refs.bib}

@misc{LiSchramm2024,
  author = {Li, Shuangping and Schramm, Tselil},
  title = {Spectral clustering in the Gaussian mixture block model},
  year = {2023},
  eprint = {2305.00979},
  archivePrefix = {arXiv},
  primaryClass = {math.ST}
}

@book{Vershynin2018,
    AUTHOR = {Vershynin, Roman},
     TITLE = {High-dimensional probability},
    SERIES = {Cambridge Series in Statistical and Probabilistic Mathematics},
    VOLUME = {47},
 PUBLISHER = {Cambridge University Press, Cambridge},
      YEAR = {2018},
     PAGES = {xiv+284},
      ISBN = {978-1-108-41519-4},
   MRCLASS = {60-01 (60B05 60B20 60E15 60Fxx 62H25)},
  MRNUMBER = {3837109},
MRREVIEWER = {Sasha\ Sodin},
       DOI = {10.1017/9781108231596},
       URL = {https://doi.org/10.1017/9781108231596},
}

@article{elkaroui2010information,
  title={On information plus noise kernel random matrices},
  author={El Karoui, Noureddine},
  journal={The Annals of Statistics},
  volume={38},
  number={5},
  pages={3191--3216},
  year={2010},
  publisher={Institute of Mathematical Statistics}
}

@article{cheng2013spectrum,
  title={The spectrum of random inner-product kernel matrices},
  author={Cheng, Xiuyuan and Singer, Amit},
  journal={Random Matrices: Theory and Applications},
  volume={2},
  number={04},
  pages={1350010},
  year={2013},
  publisher={World Scientific}
}

@article{lu2022equivalence,
  title={An equivalence principle for the spectrum of random inner-product kernel matrices},
  author={Lu, Yue M and Yau, Horng-Tzer},
  journal={Probability Theory and Related Fields},
  volume={184},
  number={1-2},
  pages={257--302},
  year={2022},
  publisher={Springer}
}

@unpublished{guionnet2023spectral,
  title={Spectral Phase Transitions in Non-Linear Wigner Spiked Models},
  author={Guionnet, Alice and Ko, Justin and Krzakala, Florent and Mergny, Pierre and Zdeborov{\'a}, Lenka},
  year={2023},
  eprint={2310.14055},
  archivePrefix={arXiv},
}

@article{ding2021learning,
    AUTHOR = {Ding, Xiucai and Ma, Rong},
     TITLE = {Learning low-dimensional nonlinear structures from
              high-dimensional noisy data: an integral operator approach},
   JOURNAL = {Ann. Statist.},
  FJOURNAL = {The Annals of Statistics},
    VOLUME = {51},
      YEAR = {2023},
    NUMBER = {4},
     PAGES = {1744--1769},
      ISSN = {0090-5364,2168-8966},
   MRCLASS = {62R07 (47G10 62R30 65J15 68T09)},
  MRNUMBER = {4658575},
MRREVIEWER = {Nicolas\ Keriven},
       DOI = {10.1214/23-aos2306},
       URL = {https://doi.org/10.1214/23-aos2306},
}

@book{dai2013approximation,
  title     = {Approximation theory and harmonic analysis on spheres and balls},
  author    = {Dai, Feng},
  year      = {2013},
  publisher = {Springer}
}

@book{horn2012matrix,
  title={Matrix Analysis},
  author={Horn, Roger A. and Johnson, Charles R.},
  year={2012},
  publisher={Cambridge University Press},
  edition={2nd}
}

@article{Baiketal2005,
  author  = {Baik, Jinho and Ben Arous, Gérard and Péché, Sandrine},
  title   = {Phase transition of the largest eigenvalue for nonnull complex sample covariance matrices},
  journal = {The Annals of Probability},
  year    = {2005},
  volume  = {33},
  doi     = {10.1214/009117905000000233}
}

@article{Abbe2018,
  author  = {Abbe, Emmanuel},
  title   = {Community Detection and Stochastic Block Models},
  journal = {Foundations and Trends{\textregistered} in Communications and Information Theory},
  year    = {2018},
  volume  = {14},
  pages   = {1--162},
  doi     = {10.1561/0100000067}
}

@article{yu2015useful,
  title={A useful variant of the {D}avis--{K}ahan theorem for statisticians},
  author={Yu, Yi and Wang, Tengyao and Samworth, Richard J},
  journal={Biometrika},
  volume={102},
  number={2},
  pages={315--323},
  year={2015},
  publisher={Oxford University Press},
  doi={10.1093/biomet/asv008}
}

@book{Hochstadt,
    AUTHOR = {Hochstadt, Harry},
     TITLE = {The functions of mathematical physics},
   EDITION = {Second},
      NOTE = {With a foreword by Wilhelm Magnus},
 PUBLISHER = {Dover Publications, Inc., New York},
      YEAR = {1986},
     PAGES = {xii+322},
      ISBN = {0-486-65214-9},
   MRCLASS = {33-02},
  MRNUMBER = {883962},
}

@book{EfFr,
    AUTHOR = {Efthimiou, Costas and Frye, Christopher},
     TITLE = {Spherical harmonics in {$p$} dimensions},
 PUBLISHER = {World Scientific Publishing Co. Pte. Ltd., Hackensack, NJ},
      YEAR = {2014},
     PAGES = {xii+143},
      ISBN = {978-981-4596-69-5},
   MRCLASS = {33C55},
  MRNUMBER = {3290046},
       DOI = {10.1142/9134},
       URL = {https://doi.org/10.1142/9134},
}

@unpublished{fernandez2026spectral,
  title={Spectral Concentration and Recovery in Sparse High-Dimensional Random Geometric Graphs},
  author={Fernandez V, Manuel and Zhu, Yizhe},
  year={2026},
  eprint={2607.14304},
  archivePrefix={arXiv},
  primaryClass={stat.ML},
}

@article{Sin06, author={Amit Singer},
  title={From graph to manifold {L}aplacian: the convergence rate},
  journal={Appl. Comput. Harmon. Anal.}, volume={21}, number={1}, pages={128--134}, year={2006}}

@incollection{GK06, author={Evarist Gin\'e and Vladimir Koltchinskii},
  title={Empirical graph {L}aplacian approximation of {L}aplace--{B}eltrami operators: large sample results},
  booktitle={High Dimensional Probability}, series={IMS Lecture Notes Monogr. Ser.},
  volume={51}, pages={238--259}, publisher={Inst. Math. Statist.}, year={2006}}

@article{DWW21, author={David B. Dunson and Hau-Tieng Wu and Nan Wu},
  title={Spectral convergence of graph {L}aplacian and heat kernel reconstruction in $L^\infty$ from random samples},
  journal={Appl. Comput. Harmon. Anal.}, volume={55}, pages={282--336}, year={2021}}

@article{EKW16, author={Noureddine El Karoui and Hau-Tieng Wu},
  title={Graph connection {L}aplacian methods can be made robust to noise},
  journal={Ann. Statist.}, volume={44}, number={1}, pages={346--372}, year={2016}}

@article{DW23, author={Xiucai Ding and Hau-Tieng Wu},
  title={Impact of signal-to-noise ratio and bandwidth on graph {L}aplacian spectrum from high-dimensional noisy point cloud},
  journal={IEEE Trans. Inform. Theory}, volume={69}, number={3}, pages={1899--1931}, year={2023}}

@article{DLMY23, author={Sofiia Dubova and Yue M. Lu and Benjamin McKenna and Horng-Tzer Yau},
  title={Universality for the global spectrum of random inner-product kernel matrices in the polynomial regime},
  journal={arXiv preprint arXiv:2310.18280}, year={2023}}

@article{BDER,
    AUTHOR = {Bubeck, S\'ebastien and Ding, Jian and Eldan, Ronen and
              R\'acz, Mikl\'os Z.},
     TITLE = {Testing for high-dimensional geometry in random graphs},
   JOURNAL = {Random Structures Algorithms},
  FJOURNAL = {Random Structures \& Algorithms},
    VOLUME = {49},
      YEAR = {2016},
    NUMBER = {3},
     PAGES = {503--532},
      ISSN = {1042-9832,1098-2418},
   MRCLASS = {05C80 (05C62 62H15)},
  MRNUMBER = {3545825},
MRREVIEWER = {Wojciech\ Kordecki},
       DOI = {10.1002/rsa.20633},
       URL = {https://doi.org/10.1002/rsa.20633},
}

@inproceedings{LMSY,
    AUTHOR = {Liu, Siqi and Mohanty, Sidhanth and Schramm, Tselil and Yang,
              Elizabeth},
     TITLE = {Testing thresholds for high-dimensional sparse random geometric graphs},
    BOOKTITLE = {S{TOC} '22---{P}roceedings of the 54th {A}nnual {ACM} {SIGACT}
              {S}ymposium on {T}heory of {C}omputing},
     PAGES = {672--677},
 PUBLISHER = {ACM, New York},
      YEAR = {2022},
      ISBN = {978-1-4503-9264-8},
   MRCLASS = {68Q87 (68W20)},
  MRNUMBER = {4490031},
       DOI = {10.1145/3519935.3519989},
       URL = {https://doi.org/10.1145/3519935.3519989}
}

@unpublished{DMSWX,
    author = {Du, Hang and Mao, Cheng and Sun, Nike and Wu, Yihong and Xu, Jiaming},
    title = {Resolution of the detection threshold conjecture for random geometric graphs in the $d>n$ regime},
    year={2026},
    eprint={2607.02013},
    archivePrefix={arXiv},
}

@book{Penrose-RGG,
    AUTHOR = {Penrose, Mathew},
     TITLE = {Random geometric graphs},
    SERIES = {Oxford Studies in Probability},
    VOLUME = {5},
 PUBLISHER = {Oxford University Press, Oxford},
      YEAR = {2003},
     PAGES = {xiv+330},
      ISBN = {0-19-850626-0},
   MRCLASS = {60-02 (05C80 60D05)},
  MRNUMBER = {1986198},
MRREVIEWER = {Ilya\ S.\ Molchanov},
       DOI = {10.1093/acprof:oso/9780198506263.001.0001},
       URL = {https://doi.org/10.1093/acprof:oso/9780198506263.001.0001}
}

@unpublished{MaMu26,
  title={Recovery of latent inner products from an anisotropic {G}aussian random geometric graph},
  author={Cheng Mao and Vidya Muthukumar},
  year={2026},
  eprint={2607.23723},
  archivePrefix={arXiv},
  primaryClass={stat.ML},
}

@unpublished{Araya20,
  title={Random geometric graphs on {E}uclidean balls},
  author={Araya Valdivia, Ernesto},
  year={2020},
  eprint={2010.13734},
  archivePrefix={arXiv},
  primaryClass={stat.ML},
}

@unpublished{MWX26,
  title={Random geometric graphs with smooth
kernels: sharp detection threshold and a spectral conjecture},
  author={Mao, Cheng and Wu, Yihong and Xu, Jiaming},
  year={2026},
  eprint={2602.14998},
  archivePrefix={arXiv},
  primaryClass={stat.ML},
}

@article{EMP22,
    AUTHOR = {Eldan, Ronen and Mikulincer, Dan and Pieters, Hester},
     TITLE = {Community detection and percolation of information in a
              geometric setting},
   JOURNAL = {Combin. Probab. Comput.},
  FJOURNAL = {Combinatorics, Probability and Computing},
    VOLUME = {31},
      YEAR = {2022},
    NUMBER = {6},
     PAGES = {1048--1069},
      ISSN = {0963-5483,1469-2163},
   MRCLASS = {62H30 (47N30 60J80)},
  MRNUMBER = {4496021},
       DOI = {10.1017/s0963548322000098},
       url = {https://doi.org/10.1017/s0963548322000098},
}
\end{document}